\documentclass[10pt]{amsart}

\expandafter\let\csname ver@amsthm.sty\endcsname\relax
\let\theoremstyle\relax

\let\qedhere\relax

\title{Acylindrical Actions on Trees and the Farrell--Jones Conjecture}

\author{Svenja Knopf}
\address{\scriptsize{WWU M\"unster\\
         FB 10 - Mathematik und Informatik\\
         Einsteinstr.~62,
         D-48149 M\"unster, Germany}}
\email{svenja.knopf@uni-muenster.de}

\date{April 2017}
\keywords{Acylindrical Actions on Trees, Farrell--Jones Conjecture, Waldhausen 
Nil-groups, $K$- and $L$-theory of group rings.}
\subjclass[2010]{18F25, 19D35, 20F65, 20F67} 
\thanks{This work was supported by the SFB 878 ``Groups, Geometry \& 
Actions''.}

\usepackage{amssymb} 
\usepackage{amsmath} 
\usepackage[shortlabels]{enumitem}
\usepackage{nicefrac}
\usepackage{tikz}
\usetikzlibrary{matrix}
\usepackage{longtable}
\usepackage{import}
\usepackage{calc}
\usepackage{xcolor}
\usepackage{graphicx}

\usepackage[hypertexnames=true, bookmarksnumbered=true, 
pageanchor=false, pdfencoding=auto]{hyperref}
\usepackage{amsthm}
\usepackage[noabbrev,nameinlink,capitalize]{cleveref}
\usepackage[figure]{hypcap}
\hypersetup{
 pdfauthor 	= Svenja Knopf
}

\theoremstyle{plain}
\newtheorem{theorem}{Theorem}[section]
\newtheorem{lemma}[theorem]{Lemma}         
\newtheorem{corollary}[theorem]{Corollary}  
\newtheorem{proposition}[theorem]{Proposition}

\newtheorem{mainthm}{Theorem}

\newtheorem{mainthmintro}{Theorem}

\theoremstyle{definition}
\newtheorem{definition}[theorem]{Definition}
\newtheorem{remark}[theorem]{Remark}
\newtheorem{notation}[theorem]{Notation} 
\newtheorem{choosingconst}[theorem]{Choosing Constants}

\newtheorem*{repcor*}{\cref{cor:KtheoryFJC}}

\newcommand{\Z}{\mathbb{Z}}
\newcommand{\Q}{\mathbb{Q}}
\newcommand{\R}{\mathbb{R}}
\newcommand{\N}{\mathbb{N}}

\newcommand{\anf }{a} 
\newcommand{\sm}[1]{\text{\tiny\colorbox{white}{\parbox[c][0pt][c]{\widthof{
$#1$}}{$#1$}}}} 
\newcommand{\kl}[1]{\text{\scriptsize{$#1$}}} 
\newcommand{\kll}[1]{\text{\scriptsize{$#1$}}} 
\newcommand{\meps}{m_\epsilon}
\newcommand{\binn}[1]{B^{\text{inn}}_{#1}}
\newcommand{\bext}[1]{B^{\text{ext}}_{#1}}

\newcommand{\thetaminus}{\Theta_{-1}}
\newcommand{\phimap}{\phi_{\Theta,\Theta^\prime,\tau}}

\newcommand{\FJ}{Farrell--Jones}
\newcommand{\NGrp}{Nil-groups}

\DeclareMathOperator{\coker}{coker} 
\DeclareMathOperator{\colim}{colim} 
\DeclareMathOperator{\diag}{diag} 
\DeclareMathOperator{\id}{id} 
\DeclareMathOperator{\pt}{pt}
\DeclareMathOperator{\ind}{ind}

\newcommand{\squote}[1]{\textquoteleft #1\textquoteright} 

\begin{document}

\begin{abstract}
We show that for groups acting acylindrically on simplicial trees the 
$K$- and $L$-theoretic \FJ\ Conjecture relative to the family of subgroups 
consisting of virtually cyclic subgroups and all subconjugates of vertex 
stabilisers holds. As an application, for amalgamated free products 
acting acylindrically on their Bass-Serre trees we obtain an identification of 
the associated Waldhausen \NGrp\ with a direct sum of \NGrp\ associated to 
certain virtually cyclic groups. This identification generalizes a result by 
Lafont and Ortiz. For a regular ring and a strictly acylindrical action these 
\NGrp\ vanish. In particular, all our results apply to amalgamated free products 
over malnormal subgroups.
\end{abstract}

\maketitle

\setcounter{tocdepth}{1}
\tableofcontents

\section*{Introduction}
The \FJ\ Conjecture (for a group $G$ and relative to the family 
$\mathcal{VCYC}$ of virtually cyclic subgroups) predicts an isomorphism 
between the $K$- resp.~$L$-groups of a group ring $RG$ and the value of a 
certain $G$-homology theory on a certain classifying space. Such an 
isomorphism can be seen as a way to (potentially) compute $K_\ast(RG)$ 
resp.~$L_\ast^{\langle -\infty\rangle}(RG)$ from smaller \squote{building 
blocks}, namely the $K$- resp.~$L$-groups of the group rings of the virtually 
cyclic subgroups of $G$. The conjecture has many implications. 
For example, if the $K$-theoretic version of the conjecture holds for a 
torsion-free group $G$, then both the Whitehead group $\text{Wh}(G)$ and the 
reduced projective class group $\tilde K_0(\Z G)$ vanish. 
Also, if both the $K$- and $L$-theoretic version hold for a group $G$, then the 
Borel Conjecture for $G$ follows. The \FJ\ Conjecture has been studied 
intensively over the past decade and is known for a large class of 
groups including hyperbolic groups \cite{BLR08,BL12b}, $\text{CAT}(0)$-groups 
\cite{BL12b, Weg12}, virtually solvable groups \cite{Weg15}, 
$\text{GL}_n(\Z)$ \cite{BLRR14} and mapping class groups \cite{BB16}. So far, 
no counter-example is known.

Proving the \FJ\ Conjecture for a group $G$ often starts with considering a 
suitable action of $G$ on a nice enough space $X$. Then it can be 
useful to regard the \FJ\ Conjecture not relative to $\mathcal{VCYC}$, but 
relative to a (larger) family $\mathcal{F}$ that includes all point-stabilisers 
of $G\curvearrowright X$: If one shows the conjecture for $G$ 
relative to $\mathcal{F}$ and all point-stabilisers are known to 
satisfy the conjecture relative to $\mathcal{VCYC}$, the \FJ\ Conjecture for 
$G$ relative to $\mathcal{VCYC}$ follows. 
Recently Bartels used this approach on relatively hyperbolic groups 
\cite{Bar17}, and the first objective of the present paper is to apply the 
same tactic to groups acting acylindrically on trees. 

The precise formulation of and necessary background on the \FJ\ conjecture are 
given in \cref{sec:backgroundFJC}. A detailed exposition on the numerous 
application of the \FJ\ Conjecture can be found in \cite{BLR08a}, \cite{LR05}, 
\cite{Lue10}. A survey on the various methods used to prove the conjecture in 
various cases can be found in 
\cite{Bar12}.\par\medskip

A discrete group $G$ acts \emph{$k$-acylindrically} on a simplicial tree $T$ if 
the pointwise stabiliser of every geodesic segment of length $k$ in $T$ is 
finite (we do not require these stabilisers to be trivial or to have uniform 
cardinality). An action $G\curvearrowright T$ is called \emph{acylindrical} if 
it is $k$-acylindrical for some $k$. Examples of acylindrical actions can be 
given by considering an amalgamated free product $G=A\ast_C B$, where $C$ is an 
almost malnormal subgroup in $A$ or $B$. 

In general, for a group action $G\curvearrowright T$ on a simplicial tree, 
denote by $\mathcal{F}_T$ the family of subgroups of $G$ which are subconjugate 
to vertex stabilisers, and by $\mathcal{F}_\partial$ the family of subgroups of 
$G$ that fix a pair of boundary points of $T$ pointwise.\par\medskip

\begin{mainthmintro}\label{thm:mainthmintro:relFJC}Let $G$ be a group acting 
acylindrically on a simplicial 
tree $T$ and let the family  
$\mathcal{F}:=\mathcal{F}_T\cup\mathcal{F}_\partial$ be as 
above. Let $\mathcal{F}_2$ be the family of subgroups of $G$ that contain a 
group in $\mathcal{F}$ as a subgroup of index $\leq 2$. Then 
\begin{enumerate}[a)]
\item $G$ satisfies the $K$-theoretic \FJ\ Conjecture relative to $\mathcal{F}$;
\item $G$ satisfies the $L$-theoretic \FJ\ Conjecture relative to 
$\mathcal{F}_2$.
\end{enumerate}
\end{mainthmintro}\par\medskip

In the special cases of $0$- or $1$-acylindrical actions this theorem was 
already known: If $G$ acts $0$-acylindrically on a tree $T$ (and $G$ is 
finitely generated), then $G$ is hyperbolic and the \FJ\ Conjecture relative to 
$\mathcal{VCYC}\subset\mathcal{F}_T\cup\mathcal{F}_\partial$ holds
\cite{BLR08, BL12b}. A more recent result of Bartels \cite{Bar17} for 
relatively hyperbolic groups covers, in particular, the $1$-acylindrical case.
We will see in \cref{sec:acyl} that it is easy to construct actions that are 
acylindrical but not $1$-acylindrical.\par\smallskip

If the vertex stabilisers of the action $G\curvearrowright T$ are known to 
satisfy the \FJ\ Conjecture relative to $\mathcal{VCYC}$, we obtain the 
following corollary and its $L$-theoretic analogue, which is formulated in 
\cref{cor:LtheoryFJCcor}.

\begin{repcor*}\textit{
Let $G$ be a group acting acylindrically on a simplicial tree $T$. If 
all vertex stabilisers of $G\curvearrowright T$ satisfy the $K$-theoretic \FJ\ 
Conjecture relative to $\mathcal{VCYC}$, then $G$ satisfies the $K$-theoretic 
\FJ\ Conjecture relative to $\mathcal{VCYC}$.}
\end{repcor*}

The \FJ\ Conjecture (with coefficients in additive categories) relative to 
$\mathcal{VCYC}$ has various useful inheritance properties. For example, if two 
groups $A$ and $B$ satisfy the conjecture, then so do $A\times B$ and $A\ast 
B$. 
If $A$ and $B$ are abelian, then an amalgamated free product $G=A\ast_C B$ 
satisfies the \FJ\ Conjecture relative to $\mathcal{VCYC}$ by \cite{GMR15}, but 
for a general amalgamated free product this inheritance property is not known. 
In this context the last corollary can also be interpreted as follows: Any 
amalgamated free product $G=A\ast_C B$ with $C$ being almost malnormal in 
either $A$ or $B$ acts acylindrically on its Bass-Serre tree. Thus, the above 
corollary implies that the class of groups satisfying the \FJ\ Conjecture 
relative to $\mathcal{VCYC}$ is closed under taking amalgamated free products 
over an almost malnormal subgroup.\par\medskip

Another angle on amalgamated free products and the computation of algebraic 
$K$-theory is the following: For a group $G=A\ast_C B$ one 
can ask whether there is a Mayer-Vietoris type exact sequence 
\begin{equation*}
\ldots \to K_n(RC)\to K_n(RA)\oplus K_n(RB)\to K_n(RG)\to K_{n-1}(RC)\to 
\ldots\ 
. 
\end{equation*}
In general this is not the case, and Waldhausen famously introduced an exact 
category $\mathfrak{Nil}(RC;R[A - C],R[B-C])$ encoding to what extent exactness 
of the above sequence fails in \cite{Wal78a, Wal78}. In particular, 
\emph{Waldhausen \NGrp}  $\widetilde{\text{Nil}}_n(RC; R[A-C], R[B-C])$ 
associated to the amalgamated free product $G=A\ast_C B$ are direct summands in 
the homotopy groups of a certain non-connective spectrum version of 
\mbox{$\mathfrak{Nil}(RC;R[A - C],R[B-C])$}, and the sequence
\begin{equation*}
\ldots \to K_n(RC)\to K_n(RA)\oplus K_n(RB)\to 
K_n(RG)/\widetilde{\text{Nil}}_{n-1} \to K_{n-1}(RC)\to \ldots \ 
\end{equation*}
becomes exact. Naturally, one is interested in identifying 
$\widetilde{\text{Nil}}_n(RC; R[A-C], R[B-C])$ with something that is known to 
vanish often. For a precise, yet concise definition of the groups 
\mbox{$\widetilde{\text{Nil}}_n(RC; R[A-C], R[B-C])$} we refer the reader to 
\cite[Section 3]{DQR11}.\par\medskip

For an amalgamated free product $G=A\ast_C B$ that acts acylindrically on the 
associated Bass-Serre tree \emph{and} is already known to satisfy the \FJ\ 
Conjecture relative to $\mathcal{VCYC}$, Lafont and Ortiz obtained an 
identification in \cite{LO09}. Building on \cref{thm:mainthmintro:relFJC}, we 
are able to lift their second assumption. Namely, we obtain the following 
generalisation of their theorem.

\begin{mainthmintro}
Let $G=A\ast_C B$ act acylindrically on the associated Bass-Serre tree 
$T$. Let $\mathcal{L}$ be a set of representatives for the orbits of the action 
$G\curvearrowright \partial T\times \partial T\setminus \diag$. Then for any 
ring $R$ there are  isomorphisms
\begin{align*}
\widetilde{\text{Nil}}_{n-1}(RC;&R[A- C], R[B- C])  \\
&\cong 
\bigoplus_{L\in\mathcal{L}} 
\coker(H_n^{G_L}(E_\mathcal{FIN}G_L;\mathbf{K}_R)\to 
H_n^{G_L}(\pt;\mathbf{K}_R))
\end{align*}
for $n\in\Z$.
\end{mainthmintro}
\noindent
More information on the right hand side of the above isomorphisms--in 
particular, when they are known to vanish--can be found in 
\cref{sec:waldnil}.\par\medskip

The structure of this paper is as follows: \cref{sec:backgroundFJC} contains 
the background on the \FJ\ Conjecture necessary for us. In particular, 
we review the method to prove the \FJ\ Conjecture for a group $G$ by 
establishing the existence of a suitable compact space $X$ and intricate 
$G$-equivariant covers of $G\times X$. In our case of a group acting on a tree 
$T$, the space $X$ will be the Bowditch compactification of $T$ and this is the 
content of \cref{sec:bowditch}. \cref{sec:acyl} then gives some context 
on groups acting acylindrically on trees. To make the paper more accessible to 
readers less familiar with the covering-construction-culture around the 
\FJ\ Conjecture, \cref{sec:relFJC} contains the proof of 
\cref{thm:mainthmintro:relFJC} modulo the lengthy and technical part and 
\cref{sec:waldnil} contains the application to Waldhausen \NGrp. Afterwards, 
the rest of the paper describes the plan to construct suitable covers of 
$G\times X$ in \cref{sec:tactic} and carries it out in Sections 
\ref{sec:definingthetasmall}-\ref{sec:coveringGtimesXipartialT}.\par\smallskip

\noindent
\textbf{Conventions.}
Any countable group in this paper is assumed to be equipped with an 
implicitly chosen proper (left-)invariant metric.
Given a metric space $(X,d)$, for $r>0$ and $x\in X$, we denote---if not 
explicitly stated otherwise---by $B_r^d(x)=B_r(x)$ the \emph{closed} ball of 
radius $r$ around $x$.
By a generalised metric on $X$ we mean a function $d: X\times X\to \R_{\geq 
0}\cup\{\infty\}$ which is symmetric, satisfies the triangle inequality and 
$d(x,y)=0$ if and only if $x=y$.
Finally, we stress that it is hard to draw an unbounded geodesic ray or a tree 
which is not locally finite. So all figures in this paper have to be 
regarded as \squote{incomplete illustrations} that hopefully still capture the 
essence of the notion or proof at hand.\par\smallskip

\noindent
\textbf{Acknowledgements.}
The results of this paper stem from my PhD thesis and I am  
indepted to my advisor Arthur Bartels for many helpful discussions, for his 
encouragement and his support. This work was supported by the SFB 878 ``Groups, 
Geometry \& Actions''.

\section{Background on the \FJ\ Conjecture}\label{sec:backgroundFJC}
The conjecture goes back to Farrell and Jones \cite{FJ93}, and 
since then, its formulation has been further conceptualised and broadly 
generalised. Its present form (with coefficients in additive 
categories) was introduced by Bartels and Reich \cite{BR07} (in the $K$-theory 
case) and Bartels and L\"uck \cite{BL10} (in the $L$-theory case). The 
existence of an axiomatic formulation of geometric conditions implying the \FJ\ 
Conjecture for a group $G$, which was also introduced by Bartels, L\"uck and 
Reich (cf.~\cref{subsec:geomcond}), allows us---for almost all of this 
paper---to display (or feign) ignorance about the precise nature of most 
objects appearing in the \FJ\ Conjecture. Thus, we restrict 
ourselves to stating the conjecture (as formulated in \cite{BR07} and 
\cite{BL10}) essentially without further comment on the objects 
involved.\par\medskip

A \emph{family $\mathcal{F}$ of subgroups of $G$} is a collection 
of subgroups of $G$ closed under conjugation and taking subgroups. Examples 
being the family $\mathcal{ALL}$ of all subgroups of $G$, the family 
$\mathcal{FIN}$  of all finite subgroups of $G$, and most importantly, the 
family $\mathcal{VCYC}$ of all virtually cyclic subgroups of $G$.

\begin{definition} ($K$-theoretic \FJ\ Conjecture relative to $\mathcal{F}$) A 
group $G$ satisfies the \emph{$K$-theoretic \FJ\ Conjecture relative to the 
family of subgroups $\mathcal{F}$} if for all additive $G$-categories 
$\mathcal{A}$ the $K$-theoretic assembly map 
\begin{equation*}
\text{asmb}_n(G,\mathcal{F},\mathcal{A}): H^G_n(E_\mathcal{F}G; 
\mathbf{K}_\mathcal{A})\to H^G_n(\pt;\mathbf{K}_\mathcal{A})\cong 
K_n\left(\int_G \mathcal{A}\right),
\end{equation*}
induced by the projection $E_\mathcal{F}G\to \pt$, is an isomorphism for all 
$n\in \Z$.
\end{definition}

\begin{definition}\label{def:LtheoreticFJC} ($L$-theoretic \FJ\ Conjecture 
relative to $\mathcal{F}$)
A group $G$ satisfies the \emph{$L$-theoretic \FJ\ Conjecture relative to the 
family of subgroups $\mathcal{F}$} if for all additive $G$-categories 
$\mathcal{A}$ with involution the $L$-theoretic assembly map 
\begin{equation*}
\text{asmb}_n(G,\mathcal{F},\mathcal{A}): H^G_n(E_\mathcal{F}G; 
\mathbf{L}^{-\infty}_\mathcal{A})\to 
H^G_n(\pt;\mathbf{L}^{-\infty}_\mathcal{A})\cong 
L^{\langle-\infty\rangle}_n\left(\int_G \mathcal{A}\right),
\end{equation*}
induced by the projection $E_\mathcal{F}G\to \pt$, is an isomorphism for all 
$n\in \Z$.
\end{definition}

As of now we use the term \emph{\FJ\ Conjecture (relative to $\mathcal{F}$)} 
when talking about both of the above definitions simultaneously. So in 
particular, and contrary to the existing literature, we do \emph{not} take 
$\mathcal{VCYC}$ to be the default family.

\subsection{A geometric condition implying the \FJ\ 
Conjecture}\label{subsec:geomcond}

Along with the proof of the $K$-theoretic \FJ\ Conjecture for hyperbolic 
groups,  Bartels, L\"uck and Reich showed that the existence of a particular 
group action $G\curvearrowright X$ admitting intricate $G$-invariant covers 
implies the \FJ\ Conjecture (cf.~\cite[Theorem 1.1]{BLR08}). Since then these 
geometric conditions have been reformulated, generalised and adapted in several 
ways (cf.~\cite[Section 1]{BL12b}, \cite[Theorem 1.1 and Sections 2 and 
3]{Weg12}, \cite[Definition 0.1 and Theorem 4.3]{Bar17}). We review here the 
formulation that suits us most while mainly following \cite{Bar17}.

\begin{definition}\label{def:wideFsubsets}Let $\mathcal{F}$ be a family of 
subgroups of a group $G$. A subset $U\subset X$ of a $G$-space $X$ is called an 
\emph{$\mathcal{F}$-subset} if there is $F\in\mathcal{F}$ such that $gU=U$ for 
$g\in F$ and $gU\cap U=\emptyset$ for $g\not\in F$. A collection $\mathcal{U}$ 
of subsets of $X$ is said to be \emph{$G$-invariant} if $gU\in\mathcal{U}$ for 
all $g\in G$ and $U\in\mathcal{U}$. An \emph{(open) $\mathcal{F}$-cover} of $X$ 
is a cover of $X$ consisting of (open) $\mathcal{F}$-subsets.
\end{definition}

\begin{definition}Let $\mathcal{U}$ be a collection of subsets of some space 
$X$. The \emph{order of $\mathcal{U}$} is $\leq N$ if each $x\in X$ is 
contained in at most $N+1$ members of $\mathcal{U}$.
If the collection $\mathcal{U}$ is a cover of $X$, the order of $\mathcal{U}$ 
will also be called the \emph{dimension of $\mathcal{U}$} and then is denoted 
by \emph{$\dim\mathcal{U}$}.
\end{definition}

For the connection of the next definition to amenable actions we refer the 
reader to \cite[Remark 0.4]{Bar17}.

\begin{definition}\label{def:finFamenable}(cf.~\cite[Definition 0.1]{Bar17}) 
Let $G$ be a countable group (with a chosen proper (left-)invariant metric) and 
let $\mathcal{F}$ be a family of subgroups. An action of $G$ on a space $X$ is 
\emph{$N-\mathcal{F}$-amenable} if for any $\alpha>0$ there exists an open 
$G$-invariant $\mathcal{F}$-cover $\mathcal{U}_\alpha$ of $G\times X$ (equipped 
with the diagonal action) with the following properties:
\begin{enumerate}[label=\alph*)]
\item the dimension of $\mathcal{U}_\alpha$ is at most $N$;
\item for all $(g,x)\in G\times X$ there is $U\in\mathcal{U}_\alpha$ with 
$B_\alpha(g)\times\{x\}\subset U$.
\end{enumerate}
An action is \emph{finitely $\mathcal{F}$-amenable} if it is 
$N-\mathcal{F}$-amenable for some $N$.
\end{definition}

Since we will not construct those types of covers in one go, we introduce the 
following useful terminology reminiscent of \cite[Assumption 1.4]{BLR08}.

\begin{definition}Let $\alpha>0$ be given. A collection $\mathcal{U}$ of 
subsets of $G\times X$ is \emph{wide for $Y$}, where $Y\subset G\times X$ is a 
subset, if for all $(g,x)\in Y$ there is $U\in\mathcal{U}$ with 
$B_\alpha(g)\times\{x\}\subset U$. The collection $\mathcal{U}$ is \emph{wide} 
if it is wide for $Y=G\times X$.
\end{definition}

\begin{definition}\label{def:controlledNdominated}(\cite[Definition 
1.5]{BL12b}) 
Let $X$ be a metric space and $N\in\N$. Then $X$ is \emph{controlled 
$N$-dominated} if, for every $\epsilon > 0$, there are a finite
CW-complex $K_\epsilon$ of dimension at most $N$, maps $i : X \to K_\epsilon, p 
: K_\epsilon \to X$ and a homotopy $H : X \times [0, 1] \to X$ between $p\circ 
i$ and $\text{id}_X$ such that for every $x\in X$ the diameter of $\{ H(x, t)\ 
|\ t\in [0, 1]\}$ is at most $\epsilon$.
$X$ is \emph{controlled finitely-dominated} if $X$ is controlled $N$-dominated 
for some $N$.
\end{definition}

\begin{theorem}\label{thm:tailoredFJCconditions} (cf.~\cite[Theorem 
4.3]{Bar17}) 
Let $G$ be a group and $\mathcal{F}$ be a family of subgroups. Let 
$\mathcal{F}_2$ be the family of subgroups of $G$ that contain a group in 
$\mathcal{F}$ as a subgroup of index $\leq 2$. If $G$ admits a finitely 
$\mathcal{F}$-amenable action on a compact contractible controlled 
finitely-dominated metric space $X$, then 
\begin{enumerate}[label=\alph*)]
\item $G$ satisfies the $K$-theoretic \FJ\ Conjecture relative to $\mathcal{F}$;
\item $G$ satisfies the $L$-theoretic \FJ\ Conjecture relative to 
$\mathcal{F}_2$.
\end{enumerate}
\end{theorem}

\begin{proof}b) is the conclusion of \cite[Theorem 1.1 (ii)]{BL12b}. The 
assumptions of \cite[Theorem 1.1]{BL12b} were phrased as \squote{$G$ is 
transfer 
reducible over $\mathcal{F}$} (cf.~\cite[Definition 1.8]{BL12b}). The notion of 
transfer reducibility is more general than $N-\mathcal{F}$-amenability of an 
action on a suitable space. (Transfer reducibility neither demands a strict 
action of $G$ on $X$, nor that $X$ is independent of the given $\alpha>0$.) 
Moreover, transfer reducibility demands $G$-invariant covers for the action of 
$G$ on $G\times X$ that is given by $h(g,x):=(hg,x)$. Nevertheless, using that 
$(g,x)\mapsto (g,g^{-1}x)$ is a $G$-invariant homeomorphism from $G\times X$ 
with the diagonal action to $G\times X$ with the action described above, it is 
straightforward to check from the respective definitions that the existence of 
a finitely $\mathcal{F}$-amenable action of $G$ on a compact contractible 
controlled finitely-dominated metric space $X$ implies that $G$ is transfer 
reducible over $\mathcal{F}$ (cf.~the comment behind \cite[Definition 
1.4]{BL12b}).

a) is the conclusion of \cite[Theorem 1.1]{Weg12}. The assumptions of 
\cite[Theorem 1.1]{Weg12} were phrased as \squote{$G$ is strongly transfer 
reducible over $\mathcal{F}$} (cf.~\cite[Definition 3.1]{Weg12}). Again, this 
notion is more general than $N-\mathcal{F}$-amenability and one can check, by 
translation of terminology, that the existence of a finitely 
$\mathcal{F}$-amenable action of $G$ on a compact contractible controlled 
finitely-dominated metric space $X$ implies that $G$ is strongly transfer 
reducible over $\mathcal{F}$.
\end{proof}

\subsection{The \FJ\ Conjecture and directed colimits}
The \FJ\ Conjecture described above has various useful inheritance properties, 
see for instance \cite[Section 5]{BEL08} and \cite[Section 2.3]{BFL14}. These 
sources have mostly been interested in the conjecture relative to the family 
$\mathcal{VCYC}$, and subsequently the author could not find the following 
lemma in the literature. 

\begin{lemma}\label{thm:fjcclosedunderdirectedlimits}Let $\{ G_i\ |\ i\in I\}$ 
be a directed system of groups (with not necessarily injective structure maps). 
Let $G=\colim_{i\in I} G_i$, with structure maps $\phi_i: G_i\to G$, and let 
$\mathcal{F}$ be a family of subgroups of $G$.
If for all $i\in I$ the group $G_i$ satisfies the $K$-theoretic 
(resp.~$L$-theoretic) \FJ\ Conjecture relative to $\phi_i^\ast\mathcal{F}:=\{ 
H\leq G_i\ |\ \phi_i(H)\in\mathcal{F}\}$, then $G$ satisfies the $K$-theoretic 
(resp.~$L$-theoretic) \FJ\ Conjecture relative to $\mathcal{F}$.
\end{lemma}

\begin{proof}In the $L$-theoretic case, if 
$\mathcal{F}=\mathcal{VCYC}$ and $\phi_i^\ast\mathcal{F}$ is replaced by 
$\mathcal{VCYC}$ as well ($\mathcal{VCYC}(G_i)$ that is), then this lemma is 
\cite[Corollary 0.8]{BL10}. This corollary is indicated in \cite{BL10} to come 
from \cite[Theorem 5.6]{BEL08}, which is formulated generally enough to also 
conclude the lemma in the $K$- and $L$-theoretic case provided the families 
$\mathcal{F}$ and $\phi_i^\ast\mathcal{F}$ of subgroups come from a class of 
groups closed under isomorphisms and taking quotients--a restriction that is 
not acceptable for our purposes. Luckily, using \cite[Theorem 5.2]{BEL08} 
instead, in combination with a straightforward generalisation of \cite[Theorem 
0.7]{BL10} to arbitrary families, allows to deduce \cite[Corollary 
0.8]{BL10} from \cite{BEL08} in a way that works verbatim for arbitrary 
families and the $K$-theoretic case as well.
\end{proof}

We wish to employ \cref{thm:tailoredFJCconditions} in order to prove 
\cref{thm:mainthmintro:relFJC}. The next section introduces the right compact 
contractible controlled finitely-dominated metric space for this.


\section{The observers' topology on a tree}\label{sec:bowditch}

For proper geodesic metric hyperbolic spaces (so in particular for a locally 
finite tree $T$) adding the Gromov boundary results in a compact metrisable 
space (see \cite[III.H.3, in particular 3.7, 3.18 (4)]{BH99}). Since we are 
interested mostly in trees that are not locally finite, changing the topology 
on $T$ itself becomes necessary.\par\medskip

We denote by $\partial T$ the geodesic/Gromov boundary of $T$, by $V(T)$ the 
set of vertices of $T$, by $V_0(T)$ the set of vertices of finite 
degree/valency of $T$ and set $V_\infty(T):=V(T)\setminus V_0(T)$. Furthermore, 
$d_T$ will always denote the path-metric on $T$.

\begin{notation}For two points $x,y$ in $T\cup\partial T$ we denote both the 
geodesic from $x$ to $y$ and its image by $[x,y]$. In particular, our 
terminology does not distinguish between geodesics, geodesic rays and 
bi-infinite geodesic rays. Moreover, we use $(x,y]$ and 
$[x,y)$ for $[x,y]\setminus \{x\}$ and $[x,y]\setminus \{y\}$, respectively 
(provided $x\in T$ and $y\in T$, respectively).
\end{notation}

Bowditch's construction of the observers' topology amounts to compactifying 
$\Delta(T):=V_\infty(T)\cup\partial T$. We will need to use a 
compactification---via the observers' topology---of all of $T\cup\partial T$ 
instead. Written in the more general language of $\R$-trees this can be found 
in \cite[Section 1]{CHL07}. 

\begin{definition}\label{def:obstop}The topology on $T\cup\partial T$ given by 
the basis 
\begin{equation*}
\{ M(z,A)\ |\ z\in T\cup\partial T, A\subset T\ \text{finite} \},
\end{equation*}
where $M(z,A):=\{ y\in T\cup\partial T\ |\ A\cap [z,y]=\emptyset \}$, is 
called the \emph{observers' topology}. The resulting space is the 
\emph{Bowditch compactification} of $T$ and will be denoted by 
\emph{$\overline{T}^{obs}$}. The \emph{Bowditch boundary of $T$} is the 
subspace $\Delta(T)=V_\infty(T)\cup\partial T\subset \overline{T}^{obs}$. We 
further set $\Delta_+(T):=V(T)\cup\partial T\subset \overline{T}^{obs}$.
\end{definition}

The above naming is justified since it is compatible with Bowditch's 
construction, i.e.~the subspace topology on $\Delta(T)\subset 
\overline{T}^{obs}$ is the topology constructed by Bowditch in \cite[Section 8, 
p.~51ff]{Bow12}. Note also that $M(z,A)$ consists of all points in 
$T\cup\partial T$ that can be reached from $z$ via a geodesic that does not run 
through a point of $A$. Therefore, if $x\in M(z,A)$, we have $M(x,A)=M(z,A)$.

\begin{lemma}\cite[Proposition 1.13]{CHL07} \label{lem:Tobsiscompact}The space 
$\overline{T}^{obs}$ is Hausdorff. Moreover, if $T$ is separable, then the 
space $\overline{T}^{obs}$ is separable and compact. 
\end{lemma}

\begin{lemma}\label{lem:Tobsis2ndcountableandmetrizable}If $T$ is countable 
(i.e.~has only countably many edges), then $\overline{T}^{obs}$ is 
second-countable and metrisable. In particular, (any 
subspace of) $\overline{T}^{obs}$ is separable.
\end{lemma}

\begin{proof}If $T$ is countable, then $T$ is separable. Choose a countable 
dense subset $Q$ of $T$ which contains all vertices of $T$. Then 
\begin{equation*}
\{ M(z,A)\ |\ z\in Q, A\subset Q\ \text{finite}\}
\end{equation*}
is countable and still a basis for the observers' topology, as $M(x,A)=M(z,A)$ 
holds for all $z\in M(x,A)$. So $\overline{T}^{obs}$ is second-countable.
Since, by the above lemma, $\overline{T}^{obs}$ is also compact Hausdorff, the 
space $\overline{T}^{obs}$ is metrisable. Finally, for 
metric spaces being second-countable and being separable are equivalent 
properties, so any subspace of $\overline{T}^{obs}$ is separable as well.
\end{proof}

We proceed to show that $\overline{T}^{obs}$ is contractible and controlled 
1-dominated (see \cref{def:controlledNdominated}). We split the proof into 
two lemmata.

\begin{lemma}\label{lem:overlineTobsprojectsontofinitesubtree}Let $d$ be a 
metric on $\overline{T}^{obs}$ compatible with the observers' topology. For 
every $\epsilon>0$ there is a finite subtree $K_\epsilon\subset T$ and a 
continuous projection $p: \overline{T}^{obs}\to K_\epsilon$ such that 
$d(x,p(x))<\epsilon$ for all $x\in\overline{T}^{obs}$.
\end{lemma}

\begin{proof}For this proof we denote the open ball of radius $r$ around 
$x\in\overline{T}^{obs}$ with respect to $d$ by $B_r(x)$. Certainly,
\begin{equation*}
\{ B_{\nicefrac{\epsilon}{2}}(x)\ |\ x\in \overline{T}^{obs}\}
\end{equation*}
is an open cover of $\overline{T}^{obs}$. Since the $M(x,A)$ with $x\in 
T\cup\partial T$ and $A$ a finite subset of $T$ form an open neighbourhood 
basis at $x$ for the observers' topology, there is for each 
$B_{\nicefrac{\epsilon}{2}}(x)$ an open set $M(x,A_x)\subset 
B_{\nicefrac{\epsilon}{2}}(x)$. So the collection
\begin{equation*}
\{ M(x,A_x)\ |\ x\in\overline{T}^{obs} \}
\end{equation*}
is an open cover of $\overline{T}^{obs}$ as well.
Since $\overline{T}^{obs}$ is compact, we can find a finite subcover $M(x_1, 
A_1),\ldots, M(x_l, A_l)$. Without loss of generality we can assume that the 
$x_i$ lie in $V(T)$. Let $K_\epsilon$ be the finite subtree of $T$ spanned by 
$x_1,\ldots, x_l$. We define $p: \overline{T}^{obs}\to K_\epsilon$ as follows: 
Fix some point $b\in K_\epsilon$ and define $p: \overline{T}^{obs}\to 
K_\epsilon$ by sending $x$ to the last vertex on $[b,x]$ that is contained in 
$K_\epsilon$.

To show $d(x,p(x))<\epsilon$, we argue as follows: For $x\in K_\epsilon$ we 
have $d(x,p(x))=0$, since $p_{|K_\epsilon}=\id_{K_\epsilon}$. For any point 
$x\not\in K_\epsilon$ there is an $x_{i_x}$ such that $x\in M(x_{i_x},A_i)$. 
Since $K_\epsilon$ is the tree spanned by $x_1,\ldots, x_l$, the point $p(x)$ 
lies on $[x,x_{i_x}]$, and hence $p(x)\in M(x_{i_x},A_i)$. Therefore, 
$d(p(x),x)< \epsilon$. The proof that $p$ is indeed continuous is not very 
iluminating and only uses the fact that inside each $B_\delta(x)$ one can find 
suitable $M(x,A)$ that are guaranteed to be maped into the given 
$B_{\tilde\epsilon}(p(x))$.
\end{proof}

\begin{lemma}\label{lem:overlineTobsiscontrolled1dominated} The space 
$\overline{T}^{obs}$ is controlled 1-dominated (by the $K_\epsilon$ of
\cref{lem:overlineTobsprojectsontofinitesubtree}). In particular, 
$\overline{T}^{obs}$ is contractible.
\end{lemma}

\begin{proof}Let $d$ be a metric on $\overline{T}^{obs}$ compatible with the 
observers' topology. Again, for this proof denote the open ball of radius $r$ 
around $x\in\overline{T}^{obs}$ with respect to $d$ by $B_r(x)$.
Let $\epsilon>0$ and let $p: \overline{T}^{obs}\to K_\epsilon$ be as in the 
previous lemma, i.e.~we have a finite subtree $K_\epsilon$ of $T$, fixed a 
point $b\in K_\epsilon$ and defined $p(x)$ as the last vertex on $[b,x]$ that 
also lies in $K_\epsilon$.
We now have to show that there is a homotopy $H:\overline{T}^{obs}\times 
[0,1]\to\overline{T}^{obs}$ between $\id_{\overline{T}^{obs}}$ and $i\circ p: 
\overline{T}^{obs}\to K_\epsilon\to\overline{T}^{obs}$, where 
$i:K_\epsilon\to\overline{T}^{obs}$ is the inclusion, such that the set 
$\{H(x,t)\ |\ t\in I\}$ has diameter $\leq 2\epsilon$ for all 
$x\in\overline{T}^{obs}$.

For $x\in\overline{T}^{obs}$ let $\gamma_{p(x),x}: [0,\infty)\to T$ be the 
unique generalised geodesic in $T$ from $p(x)$ to $x$, i.e.~if $x\in T$, then 
$\gamma(t)$ is stationary for $t\geq d_T(p(x),x)$. We define
\begin{equation*}
H(x,t)=\gamma_{p(x),x}\left(\frac{t}{1-t}\right)
\end{equation*}
where for $t=1$ and $x\in T$ this is to be read as $H(x,1)=x$, and for $t=1$ 
and $x\in\partial T$ this is to be read as $H(x,1)=[\gamma_{p(x),x}]$. So  
$H(-,0)=p$ and $H(-,1)=id_{\overline{T}^{obs}}$. Furthermore, note that the set 
$\{H(x,t)\ |\ t\in I\}$ is exactly the image of the generalised geodesic 
$\gamma_{p(x),x}$. Any point on $\gamma_{p(x),x}$ is mapped by $p$ to $p(x)$. 
Thus, this set has diameter $\leq 2\epsilon$ by the previous lemma. 
Again the proof that $H$ is indeed contiunous is an exercise in choosing 
suitable $M(-,A)$ inside the appropiate $B_\delta(x)$. 
\end{proof}

We will also need to know that the (small inductive) dimension\footnote{Recall 
that for separable metrisable spaces, the small inductive dimension coincides 
with the covering dimension (and the strong inductive dimension) \cite[Theorem 
IV.1, p.~90]{Nag65}.} of $\Delta_+(T)$ is 0. 

\begin{definition}Let $X$ be a regular space. The \emph{small inductive 
dimension $\ind X$} is $\leq n$ if for all $x\in X$ and all neighbourhoods 
$V\subset X$ of $x$ there is an open set $U\subset X$ such that $x\in U\subset 
V$ and $\ind \partial U\leq n-1$. The small inductive dimension of the empty set 
is defined as $\ind \emptyset := -1$.
\end{definition}

\begin{lemma}\label{lem:overlineTobshasdimension1}The space $\Delta_+(T)$ has 
(small inductive) dimension 0.
\end{lemma}

\begin{proof}By definition,  
$\mathcal{B}:=\{ M(z,A)\ |\ z\in T\cup \partial T, A \subset T\ \text{finite}\}$
is a basis for the topology of $\overline{T}^{obs}$. So if 
$V$ is a neighbourhood of some $y$ in $\Delta_+(T)$, then there is a set 
$M(z,A)\in\mathcal{B}$ with $y\in M(z,A)\cap \Delta_+(T)=:U\subset V$. In 
particular $U$ is open in $\Delta_+(T)$ and we wish to show that the boundary 
of $U$ (in $\Delta_+(T)$) is empty.

Since $y\in M(z,A)$, we can write $U=M(y,A)\cap \Delta_+(T)$. Now let $x_n\to 
x$ be a convergent sequence in $\Delta_+(T)$ with $x_n\in U$ for 
all $n\in\N$. If the $x_n$ are contained in a finite subtree $S$ of $T$, 
then--since $x_n\in\Delta_+(T)$--the sequence must have a constant subsequence 
and $x\in U$ follows. 
So assume from now on that the sequence is not contained in a finite subtree of 
$T$. If $x\not\in U$, then there must be an $a\in A$ lying in the interior of 
$[z,x]$. Thus $M(x,a)\cap\Delta_+(T)$ is an open set containing $x$, but none 
of the $x_n$. A contradiction to $x_n\rightarrow x$.
\end{proof}


\section{Acylindrical actions on trees}\label{sec:acyl}

The original definition of a $k$-acylindrical action of a group on a simplicial 
tree is due to Sela \cite[p.~528]{Sel97}. And in recent years Osin established 
Bowditch's more general notion of an acylindrical action of a group on a 
path-metric space \cite[p.~284]{Bow08} as part of his definition of 
acylindrically hyperbolic groups \cite{Osi16}. We will however use Delzant's 
definition of a $(k,\mathcal{FIN})$-acylindrical action (see \cite[Definition, 
p.~1215]{Del99}), which in the case of simplicial actions on trees 
generalises both Sela's and Bowditch's definition.

\begin{definition}\label{def:acylaction} Let $k\geq 0$. An action 
$G\curvearrowright T$ on a simplicial tree is \emph{$k$-acylindrical} if the 
pointwise stabiliser of any geodesic segment of length $k$ is finite.
\end{definition}

We say the action is \emph{uniformly $k$-acylindrical}, if there is a uniform 
bound on the cardinality of the pointwise stabiliser of geodesic segments of 
length $k$. The action is \emph{strictly $k$-acylindrical}, if the pointwise 
stabiliser of any geodesic segment of length $k$ is trivial. And we say an 
action is \emph{(strictly/uniformly) acylindrical}, if it is 
(strictly/uniformly) $k$-acylindrical for some $k\in\N$.

\begin{remark}In Sela's original definition of $k$-acylindricity the action was 
required to be strictly $k+1$-acylindrical in the above sense. And Bowditch's 
definition applied to a simplicial tree is equivalent to a uniformly 
$k$-acylindrical action in the above sense.
\end{remark}

In the case $k=0$ the action $G\curvearrowright T$ has finite point 
stabilisers. Hence, if $G$ is finitely generated, then the existence of a 
$0$-acylindrical action on a simplicial tree implies that $G$ is word hyperbolic 
(since $G$ is finitely generated we can assume without loss of generality that 
$G$ acts cocompactly on $T$).

In the case $k=1$ the action $G\curvearrowright T$ has finite edge stabilisers. 
Since any tree is, in particular, a fine hyperbolic graph in the sense of 
Bowditch (cf.~\cite[Definition after Proposition 2.1, p.~11]{Bow08}), if $G$ is 
finitely generated the existence of a $1$-acylindrical action on a simplicial 
tree with finitely generated vertex stabilisers implies that $G$ is strongly 
relatively hyperbolic, i.e.~relatively hyperbolic in the sense of Bowditch 
(cf.~\cite[Definition 2]{Bow08}). Examples of $1$-acylindrical actions on 
simplicial trees can be readily given: For groups $A$ and $B$ with an (up to 
isomorphism) common finite subgroup $C$, the amalgamated free product $A\ast_C 
B$ acts $1$-acylindrically on the Bass-Serre tree associated to $A\ast_C 
B$.

Plenty of examples for $k$-acylindrical actions that are not $1$-acylindrical 
can be given using the classical notion of an (almost) malnormal subgroup. 

\begin{definition}A subgroup $H\leq G$ is \emph{malnormal} resp.~\emph{almost 
malnormal} if for all $g\in G\setminus H$ the intersection $H\cap gHg^{-1}$ is 
trivial resp.~finite. 
We call a subgroup $H\leq G$ \emph{uniformly almost malnormal} if there is a 
bound $K>0$ such that for all $g\in G\setminus H$ the intersection $H\cap 
gHg^{-1}$ has at most $K$ elements.
\end{definition}

From the explicit description of the Bass-Serre tree associated to an 
amalgamated free product (see \cite[p.~33]{Ser80}) the following lemma follows 
readily.

\begin{lemma}\label{lem:ex2acylactions} (cf.~\cite[p.~528]{Sel97}) Let $C\leq 
A$ be (uniformly) almost malnormal and let $C\leq B$. Then the action of 
$G:=A\ast_{C}B$ on the Bass-Serre tree associated to $G$ is (uniformly) 
3-acylindrical. In particular, if $C$ is infinite, then $G$ does not act 
1-acylindrical on its Bass-Serre tree.
\end{lemma}

\noindent
Similar statements can be formulated for HNN-extensions and graphs of groups 
using the notion of an almost malnormal collection of subgroups.\par\medskip

There are plenty of examples for (almost) malnormal subgroups (for an overview 
on results giving such groups see for instance \cite{dlHW-O}). In particular, 
peripheral subgroups in relatively hyperbolic groups are uniformly almost 
malnormal (see for instance \cite[Lemma 4.20]{DS08}). On the other hand, it is 
worth noting that not all acylindrical actions arise in this fashion from 
(almost) malnormal subgroups. As can be seen (with some translation of 
terminology) from the type of example given in \cite[p.~946 and p.~951]{KS71b}, 
there are amalgamated free products $A\ast_C B$ that act 3-acylindrically on 
their associated Bass-Serre tree, with $C$ not being almost malnormal in $A$ or 
$B$. 

\begin{notation}\label{not:family}Let $G\curvearrowright T$ be an action of a 
group on a tree. We denote by $\mathcal{F}_T$ the family of subgroups 
$\mathcal{F}_T:=\{ H\leq G\ |\ \exists\ x\in T:\ H\leq G_x \}$ and by 
$\mathcal{F}_\partial$ the family of subgroups $\mathcal{F}_{\partial}:=\{ 
H\leq G\ |\ \exists\ (x,y)\in\partial T\times\partial T\setminus \diag:\ H\leq 
G_x\cap G_y\}$. Here $G_x$ denotes the stabiliser of a point $x\in T\cup\partial 
T$.
\end{notation}

It is known that $\mathcal{F}_\partial\subset\mathcal{VCYC}$ holds for 
acylindrical amalgamations (see \cite[Claim 2]{LO09}). The author 
expects the next lemma also to be well-known, but could not find a reference 
for this particular statement either. 

\begin{lemma}\label{lem:stabofpartialT}Let $G$ be a group acting acylindrically 
on a simplicial tree $T$. Then each element in $\mathcal{F}_\partial$ is either 
finite or virtually cyclic of type I (i.e.~surjects onto $\Z$).
\end{lemma}

\begin{proof}
For $g\in G$, if the translation length $||g||$ of $g$ is $>0$, denote by $C_g$ 
its translation axis. If $||g||=0$, denote by $C_g$ the subtree of $T$ 
which is pointwise fixed by $g$.
Let $\xi\in\partial T$. Then for all $g\in G_\xi$ the set $C_g$ has unbounded 
intersection with each geodesic ray $\rho$ representing $\xi$. ($C_g\cap\rho$ 
being empty or bounded immediately gives a contradiction to $g\in G_\xi$, since 
$T$ is a tree; cf.~\cref{fig:unboundedintersection} for an exemplary case). 
\begin{figure}[h!]
\centering
\def\svgwidth{\textwidth}
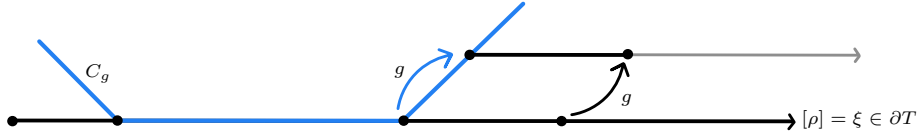
\caption{Bounded intersection of $C_g$ and $\rho$ implies $g\not\in G_\xi$.}
\label{fig:unboundedintersection}
\end{figure}

Let $(\xi_1,\xi_2)\in \partial T\times\partial T\setminus \diag$. Then any 
$g\in G_{\xi_1}\cap G_{\xi_2}$ must satisfy $[\xi_1,\xi_2]\subset C_g$ (and 
equality holds if $||g||>0$). Hence, $G_{\xi_1}\cap G_{\xi_2}$ can only contain 
finitely many elements with translation length $0$. In particular, if the group 
$G_{\xi_1}\cap G_{\xi_2}$ consists only of elements with translation length 
$0$, this group is finite.
So from now on we assume that there is at least one element $\overline{g}\in 
G_{\xi_1}\cap G_{\xi_2}$ with $||\overline{g}||>0$. Since $||g||\in \N$ there 
is $g_0\in G_{\xi_1}\cap G_{\xi_2}$ with $||g_0||>0$ minimal. 

The translation length $||h||$ of any $h\in G_{\xi_1}\cap G_{\xi_2}$ with 
$||h||>0$ must be divisible by $||g_0||$: Assume $||g_0||$ does not divide 
$||h||$. Then there are integers $\alpha,\beta$ such that $||g_0|| > 
\text{gcd}(||g_0||,||h||) =\alpha ||g_0|| + \beta ||h||$. Hence, the element 
$g_0^{\alpha}h^\beta$ acts as a shift on $[\xi_1,\xi_2]$ and therefore lies in 
$G_{\xi_1}\cap G_{\xi_2}$. Since $g_0^{\alpha}h^\beta$ has translation length 
less than $||g_0||$, this is a contradiction.

Let $\{g_1,\ldots, g_K\}$ be the set of all elements in $G_{\xi_1}\cap 
G_{\xi_2}$ with translation length $0$. If $h\in G_{\xi_1}\cap 
G_{\xi_2}$ and $||h||=0$, then $h=g_i$ for some $1\leq i\leq K$ and, in 
particular, $h=g_ig_0^0\in g_i\langle g_0\rangle$. If $h\in G_{\xi_1}\cap 
G_{\xi_2}$ with $||h||=\alpha||g_0||>0$, then one of the two elements 
$hg_0^{\pm\alpha}$ fixes $[\xi_1,\xi_2]$ pointwise and thus must lie in 
$\{g_1,\ldots, g_K\}$. In other words, we have $h=g_ig_0^\alpha$ for some 
$\alpha\in\Z$ and $1\leq i\leq K$. Hence, $\Z=\langle g_0\rangle $ is of finite 
index in $G_{\xi_1}\cap G_{\xi_2}$. Moreover, sending $g_ig_0^\alpha$ to 
$g_0^\alpha$ gives the desired surjection onto $\Z$.
\end{proof}

\section{\texorpdfstring{The proof of Theorem A modulo 
$\mathcal{F}$-amenability of 
$G\curvearrowright\Delta(T)$}{The proof of Theorem A module 
F-amenability}}\label{sec:relFJC}

\begin{mainthm}\label{thm:mainthm:relFJC} Let $G$ be a group acting 
acylindrically on a simplicial tree $T$ and let the family  
$\mathcal{F}:=\mathcal{F}_T\cup\mathcal{F}_\partial$ be as in 
\cref{not:family}. Let $\mathcal{F}_2$ be the family of subgroups of $G$ that 
contain a group in $\mathcal{F}$ as a subgroup of index $\leq 2$. Then 
\begin{enumerate}[label=\alph*)]
\item $G$ satisfies the $K$-theoretic \FJ\ Conjecture relative to $\mathcal{F}$;
\item $G$ satisfies the $L$-theoretic \FJ\ Conjecture relative to 
$\mathcal{F}_2$.
\end{enumerate}
\end{mainthm}\par\medskip

We would like to use the geometric conditions from \cref{subsec:geomcond} to 
conclude the above theorem by basically showing that $G$ acts 
finitely $\mathcal{F}$-amenable on $\overline{T}^{obs}$. The major work, which 
also turns out to be lenghty and technical, towards this is to establish 
the following proposition.

\begin{proposition}\label{prop:finFamenable}
Let $G$ be a countable group and $T$ be a countable tree. Let 
$G\curvearrowright T$ be a (not necessarily strictly or uniformly) acylindrical 
action without a global fixed point and let 
$\mathcal{F}:=\mathcal{F}_T\cup\mathcal{F}_\partial$ (as in 
\cref{not:family}). Then the action $G\curvearrowright \Delta(T)$ of $G$ on 
the Bowditch boundary of $T$ 
is finitely $\mathcal{F}$-amenable.
\end{proposition}

To make this paper more accessible to readers unfamiliar with 
the covering-construction-culture around the \FJ\ Conjecture, we postpone the 
proof of the above proposition to the Sections 
\ref{sec:tactic}-\ref{sec:coveringGtimesXipartialT}.

\begin{proof}[Proof of \cref{thm:mainthm:relFJC} (modulo the proof of 
\cref{prop:finFamenable})] 
By \cref{thm:fjcclosedunderdirectedlimits}, it suffices to prove the theorem 
for finitely generated $G$. Furthermore, if $G\curvearrowright T$ had a global 
fixed point, then $\mathcal{F}=\mathcal{ALL}$ and the conclusion of the theorem 
is imminent. Thus, we can assume without loss of generality that there 
is a minimal $G$-invariant subtree $T_{\min}$ of $T$ such that $G$ acts 
cocompactly on $T_{\min}$. In particular, $T_{\min}$ has only countably many 
edges and the restricted action is still acylindrical. If the \FJ\ Conjecture 
holds for a group $G$ relative to a family $\mathcal{G}$, then it also holds for 
$G$ relative to any larger family $\mathcal{G}^\prime\supset\mathcal{G}$ (this 
is a special case of the transitivity principle \cite[Theorem 2.10]{BFL14}). 
Hence, it suffices to show the theorem for countable $G$ and $T$ and fixed point 
free acylindrical actions $G\curvearrowright T$.

Assuming \cref{prop:finFamenable}, we proceed to show that the action of $G$ on 
$\overline{T}^{obs}$ is finitely $\mathcal{F}$-amenable: By 
\cref{prop:finFamenable}, there is $N^\prime$ such that the action of $G$ on the 
Bodwitch boundary $\Delta(T)$ is $N^\prime-\mathcal{F}$-amenable. Let 
$N:=N^\prime+2$. To produce, for all $\alpha>0$, an open $G$-invariant 
$\mathcal{F}$-cover of $G\times\overline{T}^{obs}$ that is wide and of dimension 
at most $N$, we first tend to covering $G\times\overline{T}^{obs}\setminus 
G\times\Delta(T)$.

For a moment we consider $T$ with the path-metric topology and define two kinds 
of open sets of $T$. Let $\mathcal{I}$ be all open edges of $T$ and 
$\mathcal{B}$ be all open balls of radius $\epsilon$ (i.e.~some fixed small 
positive number) around vertices of $T$. Then, both $\mathcal{I}$ and 
$\mathcal{B}$ are $G$-invariant collections of open $\mathcal{F}$-subsets of $T$ 
(with the path-metric topology). Since the observers' topology coincides with 
the path-metric topology on finite subtrees, it follows that 
$\mathcal{V^\prime}:=\{ U\setminus V_\infty(T)\ |\ 
U\in\mathcal{I}\cup\mathcal{B}\ \}$ consists of open sets of 
$\overline{T}^{obs}$. Moreover, $\mathcal{V^\prime}$ is a $G$-invariant 
collection of open $\mathcal{F}$-subsets of $\overline{T}^{obs}$, since 
$V_\infty(T)$ is $G$-invariant. As $\mathcal{I}\cup\mathcal{B}$ is of order 1, 
so is $\mathcal{V^\prime}$. Define $\mathcal{V}:=\{ G\times U^\prime\ |\ 
U^\prime\in\mathcal{V}^\prime \}$. This is a $G$-invariant collection of open 
$\mathcal{F}$-subsets of $G\times \overline{T}^{obs}$ which still has order 1. 
For $\xi\in \overline{T}^{obs}\setminus\Delta(T)$ there is an 
$U^\prime\in\mathcal{V}^\prime$ with $\xi\in U^\prime$. Thus---independent of 
$\alpha>0$---for all $(g,\xi)\in G\times(\overline{T}^{obs}\setminus\Delta(T))$ 
there is an element of $\mathcal{V}$ that contains $B_\alpha(g)\times\{\xi\}$. 
In other words, $\mathcal{V}$ is wide for 
$G\times(\overline{T}^{obs}\setminus\Delta(T))$.

Now, let $\alpha>0$ be fixed. By \cref{prop:finFamenable}, there is an open 
$G$-invariant $\mathcal{F}$-cover $\mathcal{U}_\alpha$ of $G\times\Delta(T)$ of 
dimension at most $N^\prime$ which is wide. Of course the sets in 
$\mathcal{U}_\alpha$ are not open sets of $G\times\overline{T}^{obs}$, but it is 
possible to thicken the collection $\mathcal{U}_\alpha$ to an open collection 
$\mathcal{U}_\alpha^+$ of $G\times\overline{T}^{obs}$ without losing any 
of the desired properties ($G$-invariace, order, wideness) of 
$\mathcal{U}_\alpha$ (see \cite[Appendix B]{Bar17}, \cite[Lemma 4.14]{BL12a}). 
Let $\mathcal{U}_\alpha^+$ be the result of this thickening process. So 
$\mathcal{U}_\alpha^+$ is a $G$-invariant collection of open 
$\mathcal{F}$-subsets of $G\times\overline{T}^{obs}$ that is wide for 
$G\times\Delta(T)$ and has order at most $N^\prime$. Defining 
$\mathcal{V}_\alpha:=\mathcal{V}\cup\mathcal{U}^+_\alpha$ gives the desired open 
$G$-invariant $\mathcal{F}$-cover of $G\times\overline{T}^{obs}$ which is wide 
for all of $G\times\overline{T}^{obs}$ and is of dimension at most 
$N=N^\prime+2$. 
Since we have now established finitely $\mathcal{F}$-amenability for the action 
$G\curvearrowright\overline{T}^{obs}$, to apply \cref{thm:tailoredFJCconditions} 
we merely have to recall that $\overline{T}^{obs}$ is a compact contractible 
controlled 1-dominated metric space (see \cref{lem:Tobsiscompact},
\cref{lem:Tobsis2ndcountableandmetrizable} and  
\cref{lem:overlineTobsiscontrolled1dominated}).
\end{proof}

Inheritance properties now allow us to deduce the \FJ\ Conjecture
relative to $\mathcal{VCYC}$ if it is known (relative to 
$\mathcal{VCYC}$) for all groups in $\mathcal{F}$.

\begin{corollary}\label{cor:KtheoryFJC}Let $G$ be a group acting acylindrically 
on a simplicial tree $T$. 
If all vertex stabilisers of $G\curvearrowright T$ satisfy the $K$-theoretic 
\FJ\ Conjecture relative to $\mathcal{VCYC}$, then $G$ satisfies the 
$K$-theoretic \FJ\ Conjecture relative to $\mathcal{VCYC}$.
\end{corollary}

\begin{proof} 
By \cref{thm:mainthm:relFJC}, $G$ satisfies the $K$-theoretic \FJ\ Conjecture 
relative to $\mathcal{F}_T\cup\mathcal{F}_\partial$. Hence, by 
\cref{lem:stabofpartialT}, $G$ satisfies  the $K$-theoretic \FJ\ Conjecture 
relative to the larger family $\mathcal{F}_T\cup\mathcal{VCYC}$. Since the \FJ\ 
Conjecture relative to $\mathcal{VCYC}$ is closed under taking subgroups (see 
\cite[Theorem 2.8]{BFL14}), if $H\in\mathcal{F}_T$, then $H$ satisfies the 
$K$-theoretic \FJ\ Conjecture relative to $\mathcal{VCYC}$ by assumption. If 
$H\in\mathcal{VCYC}$, $H$ satisfies the $K$-theoretic \FJ\ Conjecture trivially. 
Thus, by the transitivity principle for the \FJ\ Conjecture (see \cite[Theorem 
2.10]{BFL14}) the claim follows.
\end{proof}

The $L$-theoretic case works analogously, but we have to take overgroups of 
index 2 of vertex stabilisers of $G\curvearrowright T$ into account.  

\begin{corollary}\label{cor:LtheoryFJCcor} Let $G$ be a group acting 
acylindrically on a simplicial tree $T$. Assume that all (subgroups of $G$ that 
are) overgroups of index at most 2 of vertex stabilisers of $G\curvearrowright 
T$ satisfy the $L$-theoretic \FJ\ Conjecture relative to $\mathcal{VCYC}$. Then 
$G$ satisfies the $L$-theoretic \FJ\ Conjecture relative to $\mathcal{VCYC}$.
\end{corollary}

These two corollaries have a counterpart for the \FJ\ Conjecture with finite 
wreath products: A group $G$ is said to satisfy the \emph{\FJ\ Conjecture with 
finite wreath products relative to a family $\mathcal{F}$} if for all finite 
groups $F$ the wreath product $G\wr F$ satisfies the \FJ\ Conjecture relative 
to $\mathcal{F}$. Recalling that we proved \cref{thm:mainthm:relFJC} basically 
by showing that the group $G$ involved is strongly transfer reducible over 
$\mathcal{F}_T\cup\mathcal{F}_\partial$ (cf.~the proof of 
\cref{thm:tailoredFJCconditions}), we can obtain the following variant of the 
last two corollaries by using \cite[Theorem 5.1]{BLRR14} and the transitivity 
principle.

\begin{corollary}Let $G$ be a group acting acylindrically on a simplicial tree 
$T$. Assume that all vertex stabilisers of $G\curvearrowright T$ satisfy the 
\FJ\ Conjecture with finite wreath products relative to $\mathcal{VCYC}$. Then 
$G$ satisfies the \FJ\ Conjecture with finite wreath products relative to 
$\mathcal{VCYC}$.
\end{corollary}

\section{The proof of Theorem B}\label{sec:waldnil}

\begin{mainthm}\label{thm:mainthm:nilterme}Let $G=A\ast_C B$
act acylindrically on the associated Bass-Serre tree $T$. Let $\mathcal{L}$ be 
a set of representatives for the orbits of the action $G\curvearrowright 
\partial T\times \partial T\setminus \diag$. Then for any ring $R$ there are  
isomorphisms
\begin{align*}
\widetilde{\text{Nil}}_{n-1}(RC;&R[A- C], R[B- C])  \\
&\cong 
\bigoplus_{L\in\mathcal{L}} 
\coker(H_n^{G_L}(E_\mathcal{FIN}G_L;\mathbf{K}_R)\to 
H_n^{G_L}(\pt;\mathbf{K}_R))
\end{align*}
for $n\in\Z$.  
\end{mainthm}\par\medskip

We now need some knowledge on the objects appearing in the $K$-theoretic \FJ\ 
Conjecture. In general, for any family $\mathcal{G}$ of subgroups of $G$, 
$E_{\mathcal{G}}(G)$ denotes the classifying space for the family $\mathcal{G}$. 
Any $G$-CW-complex $X$ with $X^H\simeq\pt$ for $H\in\mathcal{G}$ and 
$X^H=\emptyset$ for $H\not\in\mathcal{G}$ is a model for $E_\mathcal{G}G$. 

We next give a suitable model for $E_{\mathcal{F}}G$, where 
$\mathcal{F}=\mathcal{F}_T\cup\mathcal{F}_\partial$ 
is again defined as in \cref{not:family}. In particular, we do not yet require 
$G$ to be an amalgamated product or the action of $G$ on $T$ to be 
acylindrically. The author expects the following construction to be well known, 
but could not find a reference in the literature.\par\medskip

Let $\mathcal{L}$ be a set of representatives for the orbits of the action 
$G\curvearrowright \partial T\times\partial T\setminus \diag$. We think of 
$\mathcal{L}$ as a collection of bi-infinite geodesic rays $L$ in $T$ and we 
denote by $G_L$ the stabiliser of the element in $\mathcal{L}$ given by $L$. In 
particular, each $G_L$ is an element of $\mathcal{F}_\partial$. Denote by 
$\text{cone}(L)$ the simplicial complex (and hence CW-complex) obtained by 
taking the simplicial join $p_L\ast L$.
Both $L$ and $\text{cone}(L)$ are $G_L$-CW-complexes, and thus $G\times_{G_L} 
L$ and $G\times_{G_L} \text{cone}(L)$ are $G$-CW-complexes. Each $G\times_{G_L} 
L$, when viewed as a CW-complex, is the disjoint union of lines. The complex 
$G\times_{G_L} L$ contains one line for each element in the set $\{gL\ |\ g\in 
G\}$. Similarly, $G\times_{G_L} \text{cone}(L)$ viewed as a CW-complex is the 
disjoint union of coned off lines and contains one coned off line for each 
element in the set $\{ \text{cone}(gL)\ |\ g\in G\}$.

Since $T$ is a $G$-CW-complex as well, we can form the following pushout of 
$G$-CW-complexes, where $i$ is the canonical inclusion induced by the inclusion 
of each line into its cone and $j_L$ is the $G$-map given by sending $(g,t)\in 
G\times_{G_L} L$ to $gt\in gL\subset T$.
\begin{equation}\label{eq:defofY}
\begin{tikzpicture}[baseline=(current  bounding  box.center)]
  \matrix (m) [matrix of math nodes,row sep=3em,column sep=4em,minimum 
width=2em]
  {
     \coprod_{L\in\mathcal{L}} G\times_{G_L} L & T \\
     \coprod_{L\in\mathcal{L}} G\times_{G_L} \text{cone}(L) & Y \\};
  \path[-stealth]
    (m-1-1) edge node [left] {$i$} (m-2-1)
            edge node [above] {$\coprod_{L\in\mathcal{L}}{j_L}$} (m-1-2)
    (m-2-1) edge (m-2-2)
    (m-1-2) edge (m-2-2);
\end{tikzpicture}
\end{equation}

\begin{lemma}\label{lem:modelforEFG}The $G$-CW-complex $Y$ given by the above 
pushout is a model for $E_\mathcal{F}G$.
\end{lemma}

\begin{proof} The $G$-CW-complex $Y$ is obtained from $T$ by equivariantly 
gluing coned off lines onto $T$ along their \squote{baselines}. Thus, the 
intersection of (the images of) different cones in $Y$ lies in $T$. In 
particular, if $L\neq L^\prime$, then the cone point $p_L$ over $L$ is distinct 
from the cone point $p_{L^\prime}$ over $L^\prime$. It follows that each $p_L$ 
is fixed exactly by the elements in $G_L$. Moreover, since $Y$ is a 
$G$-CW-complex, elements of $G$ that do not lie in some $G_L$ can only fix 
points in $T$.
Let $H\in\mathcal{F}=\mathcal{F}_T\cup\mathcal{F}_\partial$. If 
$H\not\in\mathcal{F}_\partial$, then $H$ fixes only points in $T$, and $T^H$ is 
non-empty and contractible.

If $H\in\mathcal{F}_\partial$ and $H$ contains an element of positive 
translation length (with respect to $G\curvearrowright T$), then $H$ can not 
fix a point in $T$. Furthermore, such a group $H$ fixes exactly one cone point 
$p_L$ for a unique line $L$ and---since it does not fix any point in $T$---it 
does not fix any other point in $\text{cone}(L)$. Hence, in this case $Y^H=p_L$.

If $H\in\mathcal{F}_\partial$ and all $h\in H$ have translation length zero, 
then any line $L$ for which $Hp_L =p_L$ holds must be fixed pointwise by $H$. 
Hence, for any such line $\text{cone}(L)$ must be fixed pointwise by $H$ as 
well. There might be more than one such line, but $Y^H$ is still contractible: 
Since different cones can only intersect in $T$, we can first contract $Y^H$ to 
$T^H$ by simultaneously retracting all cones $\text{cone}(L)$ in $Y^H$ to their 
respective lines $L$, which all lie in $T^H$. Then we contract $T^H$ to a point.

Conversely, if for some $H\leq G$ the set $Y^H$ is non-empty, then (since $Y$ 
is a $G$-CW-complex) $H$ must fix at least one 0-cell of $Y$. Thus, $H$ either 
fixes at least one vertex of $T$ or is a subgroup of at least one of the $G_L$. 
Then, by definition, $H$ lies in 
$\mathcal{F}_T\cup\mathcal{F}_\partial=\mathcal{F}$.
\end{proof}

In the sequel we will also use the fact that for every ring $R$ there is an 
equivariant homology theory $H_\ast^?(-;\mathbf{K}_R)$ such that for all groups 
$G$ the $G$-homology theory $H_\ast^G(-;\mathbf{K}_R)$ is the $G$-homology 
theory appearing in the formulation of the $K$-theoretic \FJ\ Conjecture with 
coefficients in the ring $R$ (see \cite[Theorem 6.1]{BEL08}). In particular, we 
assume the reader is comfortable with the properties of an equivariant homology 
theory (see \cite[Section 1]{Lue02} for a definition).

\begin{proof}[Proof of \cref{thm:mainthm:nilterme}]
Davis, Quinn and Reich, in \cite[Lemma 3.1]{DQR11}, identified the group 
$\widetilde{\text{Nil}}_{n-1}(RC;R[A- C], R[B- C])$ with the cokernel of the map
\begin{equation*}
\text{asmb}_n(G,\mathcal{F}_T,R): H_n^G(E_{\mathcal{F}_T}G;\mathbf{K}_R)\to 
H_n^G(\pt;\mathbf{K}_R).
\end{equation*}
By \cref{thm:mainthm:relFJC} a), this cokernel is isomorphic to the cokernel 
of the map
\begin{equation*}
H_n^G(T;\mathbf{K}_R)=H_n^G(E_{\mathcal{F}_T}G;\mathbf{K}_R)\to 
H_n^G(E_\mathcal{F}G;\mathbf{K}_R),
\end{equation*}
which can be determined by exploiting the properties of 
$H_\ast^?(-;\mathbf{K}_R)$ and the explicit model for $E_\mathcal{F}G$ 
constructed above.
Applying the $G$-homology theory $H_\ast^G(-;\mathbf{K}_R)$ (abbreviated to 
$H_\ast^G(-)$ in the following) to the pushout (\ref{eq:defofY}) yields a 
Mayer-Vietoris sequence
\begin{equation*}
\ldots \to \bigoplus_{L\in\mathcal{L}} H^G_n(G\times_{G_L} L) \to 
H_n^G(T)\oplus 
\bigoplus_{L\in\mathcal{L}} H^G_n(G\times_{G_L} \text{cone}(L)) \to 
H_n^G(E_\mathcal{F}G)\to \ldots \ .
\end{equation*}
The space $\text{cone}(L)$ is $G_L$-equivariantly contractible and $L$ is a 
model for $E_\mathcal{FIN}G_L$. Thus, using the induction structure of 
$H_\ast^?$ for the $G$-CW-complexes $G\times_{G_L} L = \text{ind}_{G_L}^G L$ 
and \mbox{$G\times_{G_L} \text{cone}(L)=\text{ind}_{G_L}^G (\text{cone}(L))$}, 
the above sequence becomes
\begin{equation}\label{eq:les}
\ldots \to \bigoplus_{L\in\mathcal{L}} H^{G_L}_n(E_\mathcal{FIN}G_L) 
\stackrel{(f_1,f_2)}{\longrightarrow} H_n^G(T)\oplus 
\bigoplus_{L\in\mathcal{L}} 
H^{G_L}_n(\pt) \to H_n^G(E_\mathcal{F}G)\to \ldots \ .
\end{equation}
Furthermore, the map $f_2$ is induced by the projections 
$\text{pr}_L:E_\mathcal{FIN}G_L\to \pt$. In other words, 
$f_2=\bigoplus_{L\in\mathcal{L}} \text{asmb}_n(G_L,\mathcal{FIN},R)$. Since the 
relative assembly map 
\begin{equation*}
H_\ast^{G^\prime}(E_\mathcal{FIN}H)\to H_\ast^{G^\prime}(E_\mathcal{VCYC}H) 
\end{equation*}
is split injective for all groups $G^\prime$ and rings $R$ \cite[Theorem 
1.3]{Bar03} and all $G_L$ are virtually cyclic by \cref{lem:stabofpartialT}, 
the maps $\text{asmb}_n(G_L,\mathcal{FIN},R)$ are split injective. Hence, $f_2$ 
is split injective as well and the long exact sequence (\ref{eq:les}) gives 
rise to a short exact sequence
\begin{equation*}
0 \to \bigoplus_{L\in\mathcal{L}} \mathcal{H}^{G_L}_n(E_\mathcal{FIN}G_L) 
\stackrel{(f_1,f_2)}{\longrightarrow} \mathcal{H}_n^G(T)\oplus 
\bigoplus_{L\in\mathcal{L}} \mathcal{H}^{G_L}_n(\pt) \to 
\mathcal{H}_n^G(E_\mathcal{F}G)\to 0 
\end{equation*}
for every $n\in\Z$. By basic yoga with short exact sequences one can obtain 
from any short exact sequence of abelian groups
\begin{equation*}
0 \to V_0\stackrel{(i,j)}{\longrightarrow} V_2\oplus V_1 \to V_3 \to 0
\end{equation*}
with $j$ injective, a short exact sequence of the form
$0 \to V_2 \to V_3 \to \coker(i)\to 0$.
Thus, we obtain for all $n\in\Z$ the short exact sequence
\begin{equation*}
0 \to H_n^G(T;\mathbf{K}_R)\stackrel{i_\ast}{\longrightarrow} 
H_n^G(E_\mathcal{F}G;\mathbf{K}_R)\longrightarrow \coker 
\bigoplus_{L\in\mathcal{L}} \text{asmb}_n(G_L,\mathcal{FIN},R) 
 \to 0
\end{equation*}
from which the claim follows.
\end{proof}

Exploiting what is known about the relative assembly map, one can now collect 
the following vanishing results for Waldhausen \NGrp.

\begin{corollary}\label{cor:vanishingnil}
Let $G=A\ast_C B$ act acylindrically on its Bass-Serre tree $T$  and let $R$ be 
a regular ring. Then $\widetilde{\text{Nil}}_n(RC;R[A- C], R[B- C])$ vanishes 
rationally. Furthermore, 
\begin{enumerate}[label=\alph*)]
\item if the action $G\curvearrowright T$ is strictly acylindrical (e.g.~when 
$C$ is malnormal in $A$ or $B$), then $\widetilde{\text{Nil}}_n(RC;R[A- C], 
R[B- C])=0$; 
\item if $\Q\subset R$, then $\widetilde{\text{Nil}}_n(RC;R[A- C], R[B-C])=0$.
\end{enumerate}
\end{corollary}

\begin{proof}
By \cite[Theorem 0.3]{LS15}, the relative assembly map 
\begin{equation*}
H_n^{G_L}(E_\mathcal{FIN}G_L;\mathbf{K}_R)\to 
H_n^{G_L}(E_\mathcal{VCYC}G_L;\mathbf{K}_R)
\end{equation*}
is rationally bijective. Since all $G_L$ are virtually cyclic, 
$\widetilde{Nil}_n(RC;R[A- C], R[B- C])$ vanishes rationally by 
\cref{thm:mainthm:nilterme}.

a) If $G\curvearrowright T$ is strictly acylindrical, then any $G_L$ is either 
trivial or $\Z$. Since $R$ is assumed to be regular, the fundamental theorem of 
algebraic $K$-theory \cite[5.3.30 Theorem]{Ros94} implies that all 
$H_n^{G_L}(E_\mathcal{FIN}G_L;\mathbf{K}_R)\to 
H_n^{G_L}(\pt;\mathbf{K}_R)$ are isomorphism.

b) By \cite[Lemma 21.24]{KL05}, the relative assembly map for $G_L$ is 
integrally an isomorphism. Since all $G_L$ are virtually cyclic, the claim 
follows.
\end{proof}

In light of Waldhausen's notion of regular coherent groups (see 
\cite[Section 19]{Wal78}) a remark on \cref{cor:vanishingnil} a) is in 
order. 

\begin{remark} For all regular coherent groups $C$ and all regular rings $R$, 
the group $\widetilde{\text{Nil}}_n(RC; R[A-C], R[B-C])$ already vanishes 
integrally by \cite[Theorem 11.2]{Wal78}. However, it is easy to find examples 
of amalgamated free products $G=A\ast_C B$ acting strictly acylindrically on 
their Bass-Serre tree, such that $C$ is not regular coherent, because it 
contains torsion: By \cite[Theorem C]{Kap99} any torsion-free hyperbolic group 
$A^\prime$ (that is not virtually cyclic) contains a subgroup $H\cong F_2$ that 
is malnormal in $A^\prime$. Then $C=\Z/2\Z\times F_2$ is malnormal in 
$A:=D_\infty\times A^\prime$, but is not regular coherent.
\end{remark}

\begin{remark}
In the cases where $\widetilde{\text{Nil}}_n(RC; R[A-C], R[B-C])$ does not 
vanish by any of the above statements, the right hand side of the isomorphisms 
in \cref{thm:mainthm:nilterme} has been further identified: By 
\cref{lem:stabofpartialT} the groups $G_L$ that contribute to the direct sum 
are virtually cyclic of type I. For any such 
finite-by-$\Z$ group $V=H\rtimes_\alpha\Z$ it is known (see \cite[Lemma 
3.1]{DQR11}) that the cokernel of the relative assembly map 
$H_n^{V}(E_\mathcal{FIN}V;\mathbf{K}_R)\to 
H_n^V(E_\mathcal{VCYC}V;\mathbf{K}_R)$ is isomorphic to the direct sum of two 
Farrell Nil-groups associated to $RH$ and $\alpha$. (These type of Nil-groups 
were introduced by Farrell in his PhD thesis \cite{Far7172} and a modern 
definition encompassing all degrees can be found in \cite{Gru08}). In the 
case present, where $H$ is finite, these Farrell Nil-groups are known to be 
either trivial or infinitely generated: in lower degrees, this was shown 
independently by Grunewald \cite{Gru07} and Ramos \cite{Ram07}. Very recently, 
Lafont, Prassidis and Wang extended this result to all degrees \cite{LPW16} and 
obtained a structure result for these groups (provided they have finite 
exponent). 
\end{remark}

\section{\texorpdfstring{The tactic for showing finitely 
$\mathcal{F}$-amenability}{The tactic for showing finitely 
F-amenability}}\label{sec:tactic}

The rest of this paper is concerned with the proof of \cref{prop:finFamenable}. 
Thus, from now on, $G$ is a countable group acting $k$-acylindrically and 
without a global fixed point on a countable tree $T$. Without loss of generality 
we can assume $k\geq 1$, which spares us some notational exceptions.

Showing $\mathcal{F}$-amenability of the action $G\curvearrowright \Delta(T)$ 
amounts to providing intricate \squote{wide} covers for $G\times \Delta(T)$. A 
group $G$ acting 1-acylindrically on a tree is relatively hyperbolic, so in this 
case the desired covers are already given by Bartels' methods in \cite{Bar17} 
and our tactic is to adapt those methods for the case of a $k$-acylindrical 
action.

In \cite{Bar17}, first a notion of $\Theta$-small angles \squote{between 
edges} is introduced. Its crucial feature is that it functions like a proper 
$G$-invariant metric on the set of edges $E(T)$ incident to a given vertex. 
Then, a geodesic is called $\Theta$-small if the angle between any two incident 
edges lying on this geodesic is $\Theta$-small. In particular, any 
$\Theta$-small geodesic ending in a vertex can only be extended in finitely many 
directions without loosing the property of being $\Theta$-small. Bartels uses 
the notion of $\Theta$-small geodesics to divide $G\times\Delta(T)$ in two 
subsets. For trees, the first one, $G\times_\Theta\partial T$, consists of all 
pairs $(g,\xi)\in G\times\partial T$ such that, starting from a fixed vertex in 
$T$ of finite valence, the point $g^{-1}\xi$ can be reached with a 
$\Theta$-small geodesic ray. So this set is of the form $G\cdot (\{1\}\times 
\partial T^\prime)$ for a locally finite subtree $T^\prime\subset T$. 
Furthermore, by definition of a relatively hyperbolic group, $G$ acts 
cocompactly on $E(T)$ and each edge stabiliser is finite. In the proof of 
Bartels, these three facts play an essential role in showing that 
$G\times_\Theta\partial T$ admits suitable covers. The second set, the rest, is 
covered by specially tailored sets that are constructed ad hoc, using the same 
notion of $\Theta$-small angles.

In our case of groups acting $k$-acylindrically on trees, edge stabilisers are 
in general not finite. The entities that have finite stabilisers instead are the 
geodesic segments of length $k$ in $T$. In trying to define a notion of 
$\Theta$-small geodesics by measuring \squote{angles} between \squote{incident 
segments of length $k$} on a geodesic, several difficulties arise.  First, from 
a na\"ive geometric point of view, it is not clear when two segments of length 
$k$ should be called incident. The second problem appears when determining which 
points $g^{-1}\xi\in\partial T$ can be reached from a given fixed vertex $v_0$ 
via a small geodesic. Since on a geodesic of length $< k$ there is no segment of 
length $k$, and thus no \squote{angle}, any such geodesic is automatically 
$\Theta$-small. Hence, one can start from $v_0$ in infinitely many directions 
along a $\Theta$-small geodesic, and the set $G\times_\Theta\partial T$ would be 
of the form $G\cdot (\{1\}\times \partial T^\prime)$ where $T^\prime$ is a tree 
that is \emph{not} locally finite.  Moreover, the action of $G$ on the set 
$E^k(T)$ of geodesic segments of length $k$ is not cocompact anymore. We resolve 
these difficulties as follows.

In \cref{sec:definingthetasmall}, a proper $G$-invariant metric on $E^k(T)$ is 
defined. This metric is constructed to have the additional property that the 
action of $G$ on $E^k(T)$ is \squote{cocompact along $\Theta$-small geodesics}, 
and this property is sufficient for our purposes. Using this metric we can 
measure the distance between any two segments of length $k$. Our notion of 
$\Theta$-small geodesics is also defined in \cref{sec:definingthetasmall}, where 
\squote{incident segments of length $k$} on a geodesic are segments on this 
geodesic whose midpoints have distance 1 in $T$. For the set 
$G\times_\Theta\partial T$, instead of starting to measure from a single vertex, 
we start to measure from one of two fixed geodesic segments of length $k$. These 
segments are ensured to be connected by a $\Theta$-small geodesic. This way 
$G\times_\Theta\partial T$ can again be written as $G\cdot(\{1\}\times \partial 
T^\prime)$ for a locally finite subtree $T^\prime$ of $T$, and wide covers for 
$G\times_\Theta\partial T$ are constructed in Sections 
\ref{sec:segmentflowspace}--\ref{sec:coveringGtimesXipartialT} utilising the 
long thin covers for coarse flow spaces provided by \cite[Theorem 1.1]{Bar17}.

Covers for the remaining set $G\times\Delta(T)\setminus G\times_\Theta\partial 
T$ are defined ad hoc in \cref{sec:widenessofWvTheta}. In general, a single such 
cover will not be wide, but we show that one can combine $k+2$ of these covers 
to obtain a wide cover for $G\times\Delta(T)\setminus G\times_\Theta\partial T$. 
Since \squote{incident segments of length $k$} overlap, this is more involved 
than in the case $k=1$, and we elaborate on the difficulties that arise from 
overlapping segments at the beginning of \cref{sec:widenessofWvTheta}.

Finally, formally sewing together the results of Sections 
\ref{sec:definingthetasmall}-\ref{sec:coveringGtimesXipartialT} 
gives \cref{prop:finFamenable}:

\begin{proof}[Proof of \cref{prop:finFamenable}] The proposition follows 
immediately from \cref{prop:wideness} and \cref{prop:longthin} as for any given 
$\alpha>0$ we can take the union of the cover $\mathcal{W}_\alpha$ from 
\cref{prop:wideness} and the cover 
$\mathcal{U}_{\Theta_{-1}(\alpha),\alpha}$ from \cref{prop:longthin}.
\end{proof}

\section{\texorpdfstring{Defining $\Theta$-small 
geodesics}{Defining small geodesics}}\label{sec:definingthetasmall}
In order to define $\Theta$-small geodesics we need a $G$-invariant proper 
metric on the set $E^k(T)$ of all geodesic segments of length $k$ in $T$ whose 
endpoints are vertices of $T$. Since the action of $G$ on $E^k(T)$ is, in 
general, not cocompact, we have to construct the metric in a way that the 
action on $E^k(T)$ becomes \squote{cocompact along small geodesics}.

\begin{lemma}\label{lem:propermetriconEkplus1T}
Let $G$ be a countable group and $X$ be a countable discrete space. 
Furthermore, let $G\curvearrowright X$ have finite point stabilisers. Let $\pi: 
X\to G\backslash X$ be the canonical projection and let $F=\{x_0,x_1,\ldots\}$ 
be a set of representatives for the orbits of the action $G\curvearrowright X$. 
Then there is a $G$-invariant proper metric $d_{X}$ on $X$ compatible with the 
discrete topology such that the following holds:
For all $\Theta>0$ there is $n_0\in\N$ such that if $x,y\in X$ with 
$d_{X}(x,y)\leq\Theta$, then $\pi(x)=[x_j],\pi(y)=[x_{j^\prime}]$ for some 
$j,j^\prime\leq n_0$. In (other) words: If $x$ and $y$ belong to different 
orbits and their distance is $\leq \Theta$, then the orbits of $x$ and $y$ both 
are among the first $n_0=n_0(\Theta)$ orbits.
\end{lemma}

\begin{proof}
Abels, Manoussos and Noskov showed that for a locally compact group $H$ acting 
properly on a locally compact, $\sigma$-compact metrisable space $Z$, there is 
an $H$-invariant proper metric on $Z$ compatible with the topology on $Z$ 
\cite[Theorem 1.1]{AMN11}. Thus, in the situation of the lemma, the existence 
of a $G$-invariant proper metric on $X$ follows directly from 
\cite[Theorem 1.1]{AMN11}. 
Their construction is split in several steps and one of those steps 
\cite[Section 7]{AMN11} assures that the constructed metric is orbitwise 
proper\footnote{For discrete spaces this amounts to saying that for all $r>0$ 
and $x\in X$ the image of $B_r^d(x)$, the (closed) ball of radius $r$ around 
$x$, under $\pi: X\to G\backslash X$ is finite.}. This is obtained by taking 
any proper continuous function $f: G\backslash X\to[0,\infty)$ and adding 
$d^\prime(x,y):=|f(\pi(x))-f(\pi(y))|$ to the result of their previous 
construction step. Since in our case $G\backslash X$ is indeed countable 
(instead of only $\sigma$-countable), we can choose $f$ such that orbits are 
forced further and further apart to obtain the additonally desired property the 
lemma claims. The rest of the construction of Abels, Manoussos and Noskov 
retains this property. A detailed elaboration on this modification 
can be found in the author's PhD thesis \cite[Lemma 2.3]{Kno16}.
\end{proof}

Recall that we assume $G$ to act $k$-acylindrically on $T$ for some $k\geq 1$, 
that the term \emph{geodesic} refers to both finite geodesics as 
well as (bi-infinite) geodesic rays in $T$ and that $d_T$ denotes the 
path-metric on $T$.
We set up the following long list of carefully crafted choices and notations, 
some of which are illustrated in \cref{fig:notation1}.

\begin{longtable}[l]{c|p{0.81\textwidth}}
$E^k(T)$ & Denote by $E^{k}(T)$ the set of all geodesic segments of length $k$ 
in $T$ whose endpoints are vertices. \\
$E^{k}(T)\cap S$ & For any subtree $S\subset T$ we write $\sigma\in 
E^{k}(T)\cap S$ if $\sigma\subset S$ and $\sigma\in E^{k}(T)$.\\
$d_{E^k}$ & The set $E^k(T)$ is countable and the obvious action of $G$ on 
$E^{k}(T)$ has finite point stabilisers, since the action $G\curvearrowright T$ 
is $k$-acylindrical. We fix a metric $d_{E^{k}}$ on $E^{k}(T)$ according to 
\cref{lem:propermetriconEkplus1T}.\\
$\sigma$ lies on $\gamma$ & We will often think of an element $\sigma$ in 
$E^k(T)$  as an \squote{elongated point}, which may or may not lie on a given 
geodesic $\gamma$. Thus, we will frequently use phrases like \emph{$\sigma$ 
lies on $\gamma$}, meaning that $\sigma$ is a subset of the image of $\gamma$.\\
$o(\gamma), t(\gamma)$ & For a finite geodesic $\gamma: [a,b]\to T$ we denote 
by $o(\gamma):=\gamma(a)$ and $t(\gamma):=\gamma(b)$ its start- and endpoint, 
respectively. We also extend this notation to (bi-infinite) geodesic rays in 
the canonical way.\\
$m(\gamma), m(\sigma)$ & For a finite geodesic $\gamma: [a,b]\to T$ we denote 
by $m(\gamma):=\gamma(\frac{a+b}{2})$ its midpoint. Analogously, for an element 
$\sigma\in E^k(T)$ we denote by $m(\sigma)$ its midpoint. Note that, if $k$ is 
even, the midpoint of $\sigma$ is 
a vertex of $T$.\\
$\sigma_v(\gamma)$ & For a geodesic $\gamma$ and any vertex $v$ on $\gamma$ we 
denote (if it exists) by $\sigma_v(\gamma)$ the geodesic segment of length $k$ 
on $[o(\gamma),v]$ that ends in $v$.\\
 $v+i, v-i$ & For a geodesic $\gamma$, any $i\in\N$ and a vertex $v$ on 
$\gamma$ we denote (if it exists) by $v+i$ (resp.~$v-i$) the unique vertex on 
$[v,t(\gamma)]$ (resp.~$[o(\gamma),v]$) that has distance $i$ to $v$ (with 
respect to $d_T$). If we want to specify along which geodesic we take this 
$i$th successor and $i$th predecessor, we write $v+_{\gamma} i$ and 
$v-_{\gamma} 
i$, respectively.\\
$w_0,w_0^\prime$ & We fix two vertices $w_0,w_0^\prime\in V_0(T)$ with 
$d_T(w_0,w_0^\prime)\geq 5k$. (Such points exist since we can always assume
$V_0(T)$ to be unbounded with respect to $d_T$ by replacing $T$ with its 
barycentric subdivision.)\\
$\theta_0$ & Let $\theta_0:=\max\{ d_{E^k}(\sigma,\sigma^\prime)\ |\ 
\sigma,\sigma^\prime\in E^k(T)\cap[w_0,w_0^\prime]\}$.\\
$\meps$ & Choose your favourite $\epsilon < \frac{1}{2}$. For the next two 
lines denote by $m$ the midpoint of $[w_0,w_0^\prime]$. We denote 
(throughout 
the rest of this paper) by $\meps$ the unique point on $[w_0,m]$ that has 
distance $\epsilon$ to $m$.\\ 
$\anf(g,\xi)$ & Now for every $(g,\xi)\in G\times\Delta(T)$ only one of the two 
geodesics $[gw_0,\xi]$ and $[gw_0^\prime,\xi]$ contains the point $g\meps$. We 
denote by $\anf(g,\xi)\in\{gw_0,gw_0^\prime\}$ the vertex such that 
$[\anf(g,\xi),\xi]$ contains $g\meps$. Note that, for all $g,h\in G$, we have 
$h\anf(g,\xi)=\anf(hg,h\xi)$.
\end{longtable}

Next we use the metric $d_{E^k}$ to define when a geodesic is $\Theta$-small.

\begin{definition}\label{def:smallnessofgeodesics}Let $\Theta>0$ and let 
$\gamma$ be a geodesic in $T$. Any term of the form 
$d_{E^{k}}(\sigma,\sigma^\prime)$ for
$\sigma,\sigma^\prime\in E^{k}(T)\cap \text{im}(\gamma)$ with 
$\text{diam}_{d_T}(\sigma\cap\sigma^\prime)=k-1$ is a \emph{measurement (on 
$\gamma$)}.
For a vertex $v$ on $\gamma$ the term 
$d_{E^k}(\sigma_v(\gamma),\sigma_{v+1}(\gamma))$ is called the 
\emph{measurement at $v$ (on $\gamma$)}. This is illustrated in 
\cref{fig:notation3}.\\ 
A measurement is \emph{$\Theta$-small} if it is $\leq\Theta$ and it is 
\emph{$\Theta$-large} if it is $>\Theta$. Finally, the geodesic $\gamma$ is 
\emph{$\Theta$-small} if all measurements on $\gamma$ are $\Theta$-small. In 
particular, any geodesic $\gamma$ of length $\leq k$ is automatically 
$\Theta$-small, since in this case there are no measurements on $\gamma$.
\end{definition}

\begin{figure}[!h]
\centering
\def\svgwidth{\textwidth}
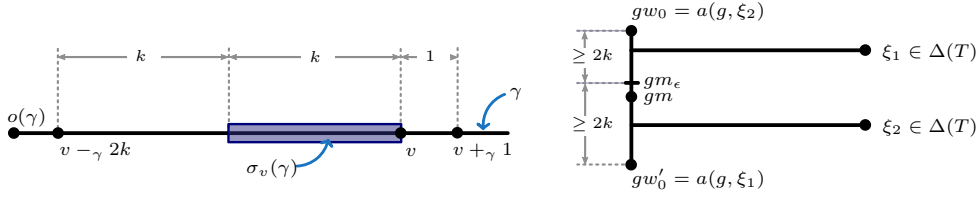
\caption{The vertices $v-_\gamma 2k$ and $v+_\gamma 1$ and the segment 
$\sigma_v(\gamma)$ on $\gamma$. The vertex $\anf(g,\xi)$ is well-defined, since 
the point $g\meps$ is not a vertex.}
\label{fig:notation1}
\end{figure}

\begin{figure}[h!]
\centering
\def\svgwidth{\textwidth}
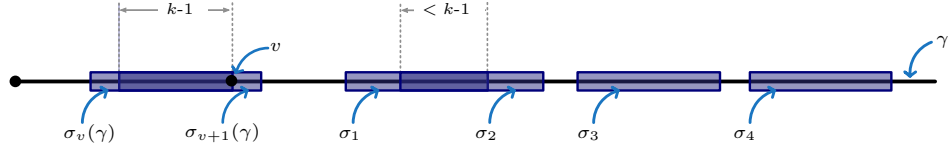
\caption{$d_{E^k}(\sigma_v(\gamma),\sigma_{v+1}(\gamma))$ is the measurement at 
$v$ on $\gamma$. Neither $d_{E^k}(\sigma_1,\sigma_2)$ nor 
$d_{E^k}(\sigma_3,\sigma_4)$ are measurements.}
\label{fig:notation3}
\end{figure}

\begin{notation}\label{not:GtimesXipartialT}For $\Theta>0$ we define
\begin{equation*}
G\times_{\Theta}\partial T:=\{ (g,\xi)\in G\times\partial T |\  
[\anf(g,\xi),\xi]\ \text{is}\ \Theta\text{-small}\}.
\end{equation*}
\end{notation}

We will find covers for $G\times_{\Theta}\partial T$ and 
its complement in $G\times\Delta(T)$ separately.\par\medskip

\section{\texorpdfstring{Covers for $G\times\Delta(T)\setminus 
G\times_\Theta\partial T$}{Covers for the 
'non-small' part of GxDelta(T)}}\label{sec:widenessofWvTheta}

The goal of this section is to establish the following proposition.

\begin{proposition}\label{prop:wideness}Retain the assumptions of 
\cref{prop:finFamenable}. Then there is $N\in\N$ such that for all $\alpha>0$ 
there is a $\thetaminus=\thetaminus(\alpha)\in\R$ and a $G$-invariant collection 
$\mathcal{W}_\alpha$ of open $\mathcal{F}_T$-subsets of $G\times\Delta(T)$ such 
that the order of $\mathcal{W}_\alpha$ is at most $N$ and $\mathcal{W}_\alpha$ 
is wide for $G\times\Delta(T)\setminus G\times_{\thetaminus(\alpha)}\partial T$.
\end{proposition}\par\medskip

We start by defining ad hoc the sets that will later on be used to build 
the collection $\mathcal{W}_\alpha$.

\begin{definition}Let $\Theta\geq\theta_0$ and $v\in V(T)$. Define the set 
$W(v,\Theta)\subset G\times \Delta(T)$ to be
\begin{align*}
W(v,\Theta):=\{  (g,\xi)\in G\times \Delta(T)\ |\ & v\in [g\meps,\xi], 
[\anf(g,\xi),v]\ \text{is $\Theta$-small},\\
				& v=\xi\ \text{or}\ 
d_{E^{k}}(\sigma_v(\gamma),\sigma_{v+_\gamma 1}(\gamma))>\Theta,\\
				& \text{where}\ \gamma:=[\anf(g,\xi),\xi] \}.
\end{align*} 
\end{definition}

This definition is illustrated below in \cref{fig:notation4}. Furthermore, 
note that,  if $v\in[g\meps,\xi]$, then the geodesic
$[\anf (g,\xi),v]$ is of length $\geq 2k$ since we have chosen $w_0,w_0^\prime$ 
with $d_T(w_0,w_0^\prime)\geq 5k$. Hence, if $v\in[g\meps,\xi]$, then the 
segment $\sigma_v(\gamma)$ is always defined and, if additionally $v\neq\xi$, 
so is the segment $\sigma_{v+_\gamma 1}(\gamma)$. \pagebreak

\begin{figure}[h!]
\centering
\def\svgwidth{\textwidth}
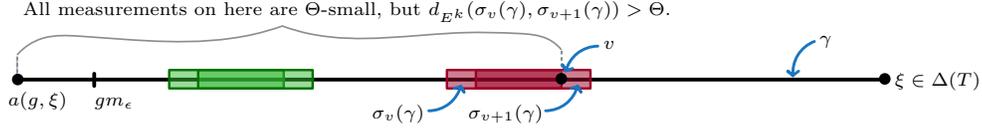
\caption{The point $(g,\xi)$ lies in $W(v,\Theta)$.}
\label{fig:notation4}
\end{figure}

If we were to drop the requirement of wideness in \cref{prop:wideness}, then 
one collection of the form $\mathcal{W}_\Theta:=\{W(v,\Theta)\ |\ v\in V(T)\}$ 
would be sufficient as the following two lemmata show.

\begin{lemma}\label{lem:propertiesofWvTheta}Let $\Theta\geq\theta_0$.
\begin{enumerate}[label=\alph*)]
\item For all $v\in V(T)$, the set $W(v,\Theta)\subset G\times\Delta(T)$ is 
open.
\item For all $g\in G$ and $v\in V(T)$, we have $gW(v,\Theta)=W(gv,\Theta)$.
\item If $v\neq w\in V(T)$, then $W(v,\Theta)\cap W(w,\Theta)=\emptyset$.
\end{enumerate}
\end{lemma}

\begin{proof}
a) Let $(g,\xi)\in W(v,\Theta)$ and let $\gamma:=[\anf(g,\xi),\xi]$. First, let 
$v\neq\xi$, and thus $[\anf(g,\xi),v+_\gamma 1]\subset[\anf(g,\xi),\xi]$. Note 
that in this case any point $(g,\xi^\prime)$ with 
\begin{equation*}
[\anf(g,\xi),v+_\gamma 1]=[\anf(g,\xi^\prime),v+_\gamma 1]\subseteq 
[\anf(g,\xi^\prime),\xi^\prime]
\end{equation*}
lies in $W(v,\Theta)$ as well (since the measurements relevant for $(g,\xi)\in 
W(v,\Theta)$ are the same measurements relevant for $(g,\xi^\prime)\in 
W(v,\Theta)$). Set $U:=M(\xi,v)\cap \Delta(T)$. Then, for all $\xi^\prime\in 
U$, the geodesic $[\anf(g,\xi),\xi^\prime]$ contains $v+_\gamma 1$, and hence 
$g\meps$ as well. Therefore, $[\anf(g,\xi),v+_\gamma 
1]=[\anf(g,\xi^\prime),v+_\gamma 1]$ and $(g,\xi)\in W(v,\Theta)$ follows.
If $v=\xi$, then we construct $U$ as follows: Let 
$\sigma:=\sigma_v([\anf(g,\xi),v])$. Since $d_{E^{k}}$ is proper, there are 
only finitely many $\sigma^\prime$ with 
$d_{E^{k}}(\sigma,\sigma^\prime)\leq\Theta$ and 
$\text{diam}_{d_T}(\sigma\cap\sigma^\prime)=k-1$. Denote them by 
$\sigma^\prime_1,\ldots, \sigma^\prime_m$. Set
\begin{equation*}
U:= M(v,\{t(\sigma^\prime_1),\ldots, t(\sigma^\prime_m)\}) \cap \Delta(T).
\end{equation*}
Since $[\anf(g,\xi),v]$ is $\Theta$-small by assumption, the segment 
$\sigma_{v-_\gamma 1}(\gamma)$ is some $\sigma^\prime_i$, and hence $v-_\gamma 
1$ is $t(\sigma_i^\prime)$. Therefore, for any $\xi^\prime\in U$, the geodesic 
$[\anf(g,\xi),\xi^\prime]$ contains $[\anf(g,\xi),v]$ and, in particular, 
$g\meps$. It follows that $[\anf(g,\xi^\prime),v]=[\anf(g,\xi),v]$, which is 
$\Theta$-small by assumption. And our choice of the $\sigma_i^\prime$ 
guarantees that the measurement at $v$ on $[\anf(g,\xi^\prime),\xi^\prime]$ is 
$\Theta$-large.

b) This follows from the $G$-invariance of $d_{E^{k}}$.

c) Assume we have $v\neq w$ and $(g,\xi)\in W(v,\Theta)\cap W(w,\Theta)$. Then, 
by definition of $W(v,\Theta)$ and $W(w,\Theta)$, both $v$ and $w$ lie on 
$[g\meps,\xi]$. Hence, without loss of generality we may assume $v$ lies on 
$[g\meps,w]$. 
Now, on the one hand, $(g,\xi)\in W(w,\Theta)$ implies that the measurement at 
$v$ on $[\anf(g,\xi),w]$ is $\Theta$-small. On the other hand, $(g,\xi)\in 
W(v,\Theta)$ implies that the measurement at $v$ on $[\anf(g,\xi),\xi]$ is 
$\Theta$-large. 
\end{proof}

\begin{lemma}\label{lem:propertiesofmathcalWTheta}Let $\Theta\geq\theta_0$. Then
$\mathcal{W}_\Theta:=\{ W(v,\Theta)\ |\ v\in V(T)\}$
is a $G$-invariant collection of open $\mathcal{F}_T$-subsets of $G\times 
\Delta(T)$ and covers $G\times\Delta(T)\setminus G\times_{\Theta}\partial T$. 
Furthermore, the order of $\mathcal{W}_\Theta$ is 0.
\end{lemma}

\begin{proof}By \cref{lem:propertiesofWvTheta} it is clear that 
$\mathcal{W}_\Theta$ is an open, $G$-invariant collection of order 0. Moreover, 
by combining \cref{lem:propertiesofWvTheta} b) and c), it follows that each 
$W(v,\Theta)$ is an \mbox{$\mathcal{F}_T$-subset} of $G\times\Delta(T)$. The 
only property left to show is that the collection $\mathcal{W}_\Theta$ covers 
the set $G\times\Delta(T)\setminus G\times_{\Theta}\partial T$:
Let $(g,\xi)\in G\times\Delta(T)\setminus G\times_{\Theta}\partial T$ and 
$\gamma:=[\anf (g,\xi),\xi]$. If $\gamma$ is $\Theta$-small, then $\xi\in 
V(T)$, and hence $(g,\xi)\in W(\xi,\Theta)$. If $\gamma$ is not $\Theta$-small, 
then there is a unique vertex $v$ on $\gamma$ such that $[\anf (g,\xi),v]$ is 
$\Theta$-small, but $[\anf (g,\xi),v+_\gamma 1]$ is not. Note that $v$ might 
lie on $[gw_0,gw_0^\prime]$, but due to our choice of $\theta_0$ the vertex $v$ 
can not lie between $\anf (g,\xi)$ and $g\meps$. Thus, $v\in [g\meps,\xi]$ and 
$(g,\xi)\in W(v,\Theta)$ holds.
\end{proof}

To obtain---for a given $\alpha>0$---a threshold 
$\thetaminus(\alpha)$ and a collection $\mathcal{W}_\alpha$ that is wide 
for $G\times\Delta(T)\setminus G\times_{\thetaminus(\alpha)}\partial T$, it is 
necessary to combine $k+2$ collections of the form $\mathcal{W}_{\Theta_i}$. 
The verification that we can indeed find 
suitable $\Theta_i$ such that 
$\mathcal{W}_\alpha:=\bigcup_{i=0}^{k+1}\mathcal{W}_{\Theta_i}$ is wide for 
$G\times\Delta(T)\setminus G\times_{\Theta} \partial T$ is rather technical, so 
we outline its underlying ideas first: Since all $\mathcal{W}_\Theta$ are 
$G$-invariant and all $G\times\Delta(T)\setminus 
G\times_{\Theta}\partial T$ are as well, it suffices to consider 
wideness for points of the form $(1,\xi)$.\\ 
For any $\Theta\geq\theta_0$ we already know that for each point $(1,\xi)\in 
G\times\Delta(T)\setminus G\times_{\Theta} \partial T$ we can find a vertex 
$v\in V(T)$ by \cref{lem:propertiesofmathcalWTheta} such that $(1,\xi)\in 
W(v,\Theta)$ holds. We do not expect $[\anf(h,\xi),v]$ to be $\Theta$-small for 
$h\in B_\alpha(1)$ as well, but we wish to calculate how small the geodesics 
$[\anf(h,\xi),v]$ are, given that $[\anf(1,\xi),v]$ is $\Theta$-small.

In the case $k=1$ this results in the following type of picture:
\begin{center}
\def\svgwidth{0.85\textwidth}
\begingroup%
  \makeatletter%
  \providecommand\color[2][]{%
    \errmessage{(Inkscape) Color is used for the text in Inkscape, but the package 'color.sty' is not loaded}%
    \renewcommand\color[2][]{}%
  }%
  \providecommand\transparent[1]{%
    \errmessage{(Inkscape) Transparency is used (non-zero) for the text in Inkscape, but the package 'transparent.sty' is not loaded}%
    \renewcommand\transparent[1]{}%
  }%
  \providecommand\rotatebox[2]{#2}%
  \ifx\svgwidth\undefined%
    \setlength{\unitlength}{380bp}%
    \ifx\svgscale\undefined%
      \relax%
    \else%
      \setlength{\unitlength}{\unitlength * \real{\svgscale}}%
    \fi%
  \else%
    \setlength{\unitlength}{\svgwidth}%
  \fi%
  \global\let\svgwidth\undefined%
  \global\let\svgscale\undefined%
  \makeatother%
  \begin{picture}(1,0.20782106)%
    \put(0,0){\includegraphics[width=\unitlength]{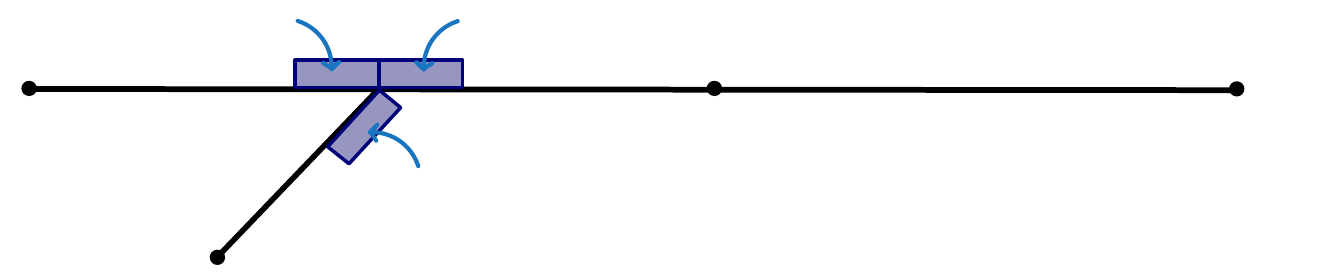}}%
    \put(0.54778114,0.15064858){\color[rgb]{0,0,0}\makebox(0,0)[lb]{\smash{\kll{v}}}}%
    \put(-0.00613281,0.15660314){\color[rgb]{0,0,0}\makebox(0,0)[lb]{\smash{\kll{\anf(1,\xi)}}}}%
    \put(0.95435071,0.13476235){\color[rgb]{0,0,0}\makebox(0,0)[lb]{\smash{\kll{\xi\in\Delta(T)}}}}%
    \put(0.17619784,0.0006817){\color[rgb]{0,0,0}\makebox(0,0)[lb]{\smash{\kll{\anf(h,\xi)}}}}%
    \put(0.1945261,0.18720921){\color[rgb]{0,0,0}\makebox(0,0)[lb]{\smash{\kll{\sigma_0}}}}%
    \put(0.35564243,0.18685738){\color[rgb]{0,0,0}\makebox(0,0)[lb]{\smash{\kll{\sigma_2}}}}%
    \put(0.31977447,0.06759292){\color[rgb]{0,0,0}\makebox(0,0)[lb]{\smash{\kll{\sigma_1}}}}%
  \end{picture}%
\endgroup%

\end{center}
Let $L_0$ be the tree spanned by $\{ a(h,\xi)\ |\ h\in B_\alpha(1)\}$. Since 
$B_\alpha(1)$ is finite, $L_0$ is as well. Thus, there is some $\Upsilon$ 
independent of $h$ such that $d_{E^1}(\sigma_0,\sigma_1)\leq\Upsilon$. Since 
$d_{E^1}(\sigma_0,\sigma_2)\leq\Theta$ by virtue of $(1,\xi)\in W(v,\Theta)$, we 
can estimate $d_{E^1}(\sigma_1,\sigma_2)$ to be $\leq\Theta+\Upsilon$. Hence, 
for all $h\in B_\alpha(1)$ the geodesic $[\anf(h,\xi),v]$ is $\Theta + 
\Upsilon$-small. If the measurement at $v$ on $[\anf (1,\xi),\xi]$ is not only 
$\Theta$-large, but in fact $\Theta + \Upsilon$-large, we can conclude 
$B_\alpha(1)\times\{\xi\}\subset W(v,\Theta + \Upsilon)$. However, in general, 
the measurement at $v$ on $[\anf (1,\xi),\xi]$  is not $\Theta + \Upsilon$-large 
and this requires that we \squote{push $v$ further to the right}. More 
precisely, we search for the first vertex $w$ on $[\anf (1,\xi),\xi]$ such that 
$[\anf (1,\xi),w]$ is $\Theta + \Upsilon$-small, but the measurement at $w$ is 
not. If such $w$ exists, then $B_\alpha( 1)\times\{\xi\}\subset W(w,\Theta + 
\Upsilon)$ follows. If such $w$ does not exist, then $[\anf(1,\xi),\xi]$ is 
$\Theta + \Upsilon$-small. If $\xi\in V_\infty(T)$, our tailored definition of 
the sets $W(v,\Theta+\Upsilon)$ implies that $B_\alpha(1)\times\{\xi\}\subseteq 
W(\xi,\Theta + \Upsilon)$ holds. And if $\xi\in\partial T$, we do not require 
our collection of sets to be wide at $(1,\xi)$ anyway. 
Since we only used sets of the form $W(- ,\Theta+\Upsilon)$, we thus obtain a 
wide collection for $G\times\Delta(T)\setminus G\times_{\Theta} \partial T$ and 
 this basically concludes the proof of \cref{prop:wideness} for the 
case $k=1$.

In the case $k\geq 2$, however, examining how small the geodesics 
$[\anf(h,\xi),v]$ are is more involved since we will have to make statements 
about pictures of the following type:
\begin{center}
\def\svgwidth{0.85\textwidth}
\begingroup%
  \makeatletter%
  \providecommand\color[2][]{%
    \errmessage{(Inkscape) Color is used for the text in Inkscape, but the package 'color.sty' is not loaded}%
    \renewcommand\color[2][]{}%
  }%
  \providecommand\transparent[1]{%
    \errmessage{(Inkscape) Transparency is used (non-zero) for the text in Inkscape, but the package 'transparent.sty' is not loaded}%
    \renewcommand\transparent[1]{}%
  }%
  \providecommand\rotatebox[2]{#2}%
  \ifx\svgwidth\undefined%
    \setlength{\unitlength}{380bp}%
    \ifx\svgscale\undefined%
      \relax%
    \else%
      \setlength{\unitlength}{\unitlength * \real{\svgscale}}%
    \fi%
  \else%
    \setlength{\unitlength}{\svgwidth}%
  \fi%
  \global\let\svgwidth\undefined%
  \global\let\svgscale\undefined%
  \makeatother%
  \begin{picture}(1,0.20782106)%
    \put(0,0){\includegraphics[width=\unitlength]{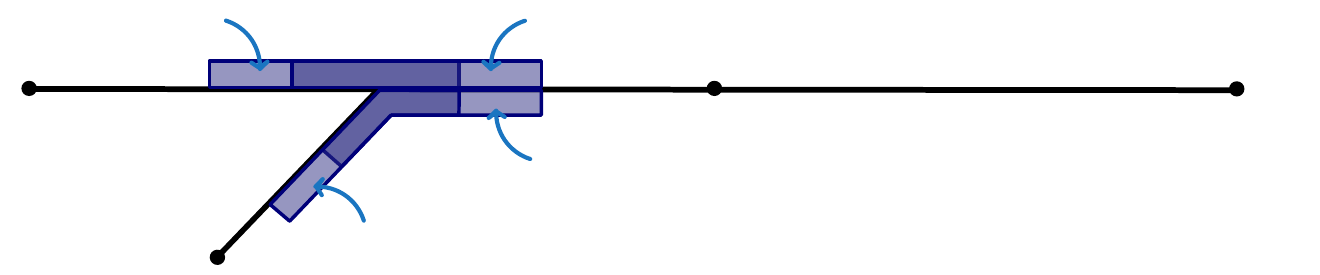}}%
    \put(0.54778114,0.15064858){\color[rgb]{0,0,0}\makebox(0,0)[lb]{\smash{\kll{v}}}}%
    \put(-0.00613281,0.15660314){\color[rgb]{0,0,0}\makebox(0,0)[lb]{\smash{\kll{\anf(1,\xi)}}}}%
    \put(0.95435071,0.13476235){\color[rgb]{0,0,0}\makebox(0,0)[lb]{\smash{\kll{\xi\in\Delta(T)}}}}%
    \put(0.17619784,0.0006817){\color[rgb]{0,0,0}\makebox(0,0)[lb]{\smash{\kll{\anf(h,\xi)}}}}%
    \put(0.14019183,0.18743134){\color[rgb]{0,0,0}\makebox(0,0)[lb]{\smash{\kll{\sigma_0}}}}%
    \put(0.40652423,0.18715368){\color[rgb]{0,0,0}\makebox(0,0)[lb]{\smash{\kll{\sigma_3}}}}%
    \put(0.27849962,0.02634885){\color[rgb]{0,0,0}\makebox(0,0)[lb]{\smash{\kll{\sigma_1}}}}%
    \put(0.40936346,0.07882517){\color[rgb]{0,0,0}\makebox(0,0)[lb]{\smash{\kll{\sigma_2}}}}%
  \end{picture}%
\endgroup%

\end{center}
We still know that $L_0$ (as defined before) is finite. Thus, we still have a 
bound $\Upsilon$ on all measurements on $[\anf(1,\xi),\anf(h,\xi)]$. However, 
we do not have a triangle inequality anymore to give an estimate for 
$d_{E^{k}}(\sigma_1,\sigma_2)$ based on $\Upsilon$ and the fact that 
$d_{E^{k}}(\sigma_0,\sigma_3)\leq \Theta$. So in order to obtain some estimate 
we use that $d_{E^{k}}$ is proper: since $[\anf (1,\xi),v]$ is $\Theta$-small, 
has a segment of length $\geq k$ in common with $L_0$ and the latter is finite, 
there are only finitely many possibilities for $\sigma_3$. Adding all possible 
$\sigma_3$ to $L_0$ still results in a finite tree $L_1$ and we find some bound 
$\Upsilon^\prime$ on the measurements lying on $L_1$. This obviates the need for 
a triangle inequality, since now both $\sigma_1$ and $\sigma_2$ lie on $L_1$. 
Hence, $d_{E^{k}}(\sigma_1,\sigma_2)$ is $\Upsilon^\prime$-small. The idea of 
\squote{pushing $v$ further to the right if necessary} stays the same, but 
various possible constellations of how the resulting vertex $w$ lies relative to 
$v$ necessitates more than one iteration step and careful case distinctions, 
both of which will be carried out in the following.\par\medskip

We start by making precise how to define an enlargement of $L_0$ suitable for 
all iteration steps.

\begin{choosingconst}\label{constants} Let $\alpha>0$. The following notions 
all depend on $\alpha$. However, since we will never compare different values 
of $\alpha$, this dependence is not reflected in the notation.
\begin{longtable}[r]{c|p{0.865\textwidth}}
 $L_0$ & The set $B_\alpha(1)$ is finite. As is the set $M:=\{ hw_0, 
hw_0^\prime\ |\ h\in B_\alpha(1)\}$. Let $L_0\subset T$ be the subtree of $T$ 
spanned by $M$. In particular, $L_0$ is a finite subtree of $T$.\\
 $\Upsilon_0$ & We define $\Upsilon_0$ as the diameter of the set $E^k(T)\cap 
L_0$ with respect to the metric $d_{E^k}$. Thus, for all 
$\sigma,\sigma^\prime\in E^{k}(T)\cap L_0$ the inequality 
$d_{E^{k}}(\sigma,\sigma^\prime)\leq\Upsilon_0$ holds. Note that $\Upsilon_0$ 
automatically has to be $\geq\theta_0$ as $[w_0,w_0^\prime]\subset L_0$.\\
 $L_{k}(\chi)$ & For $\chi>0$ we construct the complex $L_{k}(\chi)$ by 
enlarging $L_0=:L_0(\chi)$ iteratively $k$-times as follows: For $1\leq \nu\leq 
k$ let $L_\nu(\chi)$ be the union of $L_{\nu-1}(\chi)$ and all $\sigma\in 
E^{k}(T)$ such that there is $\sigma^\prime\in E^{k}(T)\cap L_{\nu-1}(\chi)$ 
with $\sigma\cap \sigma^\prime\neq\emptyset$ and 
$d_{E^{k}}(\sigma,\sigma^\prime)\leq\chi$. Since, to obtain $L_\nu(\chi)$ from 
$L_{\nu-1}(\chi)$, we only add segments that have non-trivial intersection with 
$L_{\nu-1}(\chi)$, each $L_\nu(\chi)$ is again a subtree of $T$. Since $L_0$ is 
finite and $d_{E^{k}}$ is proper, the tree $L_\nu(\chi)$ is finite for all 
$1\leq\nu\leq k$.\\
 $\Upsilon_1(\chi)$ & We define $\Upsilon_1(\chi)$ as the diameter of the set 
$E^k(T)\cap L_{k}(\chi)$ with respect to the metric $d_{E^k}$. Thus, for all 
$\sigma,\sigma^\prime\in E^{k}(T)\cap L_{k}(\chi)$ the inequality 
$d_{E^{k}}(\sigma,\sigma^\prime)\leq\Upsilon_1(\chi)$ holds. Note that 
$\Upsilon_1(\chi)\geq\Upsilon_0$.\\
$\Upsilon_m(\chi)$ & Let $\chi>0$. For $m\geq 0$ we define $\Upsilon_m(\chi)$ 
iteratively by $\Upsilon_m(\chi):=\Upsilon_1(\Upsilon_{m-1}(\chi))$. 
\end{longtable}
\end{choosingconst}

The proof of \cref{prop:wideness} requires some more notation.

\begin{notation}\label{not:for1xi} For a given $\alpha>0$ and $(1,\xi)\in 
G\times\Delta(T)$ we fix the following notation that depends on $\alpha$ and 
$\xi$ (cf.~\cref{fig:BinnandBext} for an illustration of some of the 
notation). The naming will not always reflect both of these dependencies, 
though. 
\begin{longtable}[r]{c|p{0.87\textwidth}}
$\gamma_h$ & For $h\in B_\alpha(1)$ let $\gamma_h:=[\anf(h,\xi),\xi]$.\\
$v_h$ & For $h\neq 1$ let $v_h$ be the unique first vertex that lies on both 
$\gamma_1$ and $\gamma_h$. Hence, we have $\gamma_1\cap\gamma_h=[v_h,\xi]$. 
Note that the case $v_h=\xi$ is possible.\\
$\binn{v}$ & For any vertex $v$ of $\gamma_1$ we set $\binn v:=\{h\in 
B_\alpha(1)\ |\ v\in V(\gamma_h)\}$. Thus, $h\in B_\alpha(1)$ is an element of 
$\binn v$ if and only if $\gamma_h$ contains $v$, which is the case if and only 
if $[v,\xi]\subseteq[v_h,\xi]$.\\
$\bext{v}$ & Furthermore, let $\bext v:=B_\alpha(1)\setminus \binn v$. Thus, 
$h\in B_\alpha(1)$ is an element of $\bext v$ if and only if 
$v\not\in[v_h,\xi]$.\\
$v_0(\chi)$ & For any $\chi>\theta_0$ denote---if it exists---by $v_0(\chi)$ 
the first vertex on $\gamma_1$ such that $[\anf(1,\xi),v_0(\chi)]$ is 
$\chi$-small, but $[\anf(1,\xi),v_0(\chi)+_{\gamma_1} 1]$ is not. In 
particular, the existence of $v_0(\chi)+_{\gamma_1} 1$ is necessary for the 
existence of $v_0(\chi)$. Note that $d_T(\anf(1,\xi),v_0(\chi))\geq 2k$: Since 
$[w_0,w_0^\prime]$ is $\theta_0$-small, it follows that $\meps\in 
[\anf(1,\xi),v_0(\chi)]$ and $[\anf(1,\xi),\meps]$ has length $\geq 2k$.
\end{longtable}
\end{notation}

\begin{figure}[h!]
\centering
\def\svgwidth{0.91\textwidth}
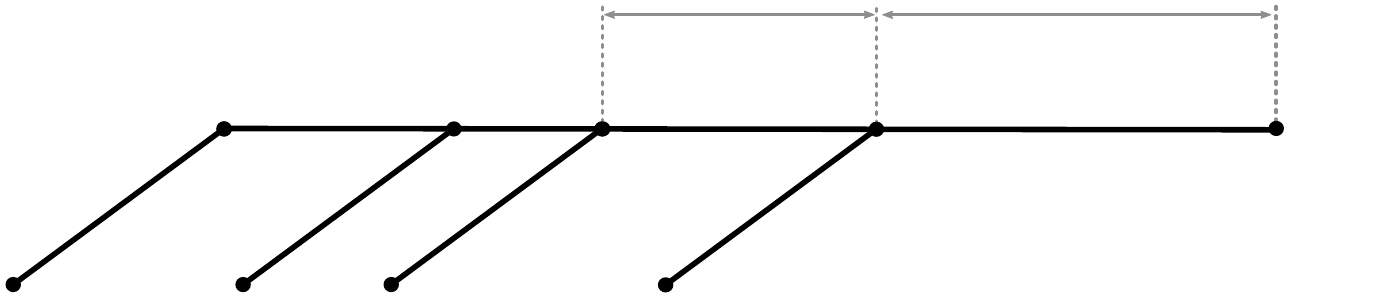
\caption{$h_1,h_2,h_3\in B_\alpha(1)$ belong to $\binn{v}$, but $h_4\in 
B_\alpha(1)$ belongs to $\bext{v}$.}
\label{fig:BinnandBext}
\end{figure}

We proceed by showing a series of technical lemmata, whose sole purpose is to 
increase readability of the proof of \cref{prop:wideness}. For all  
statements we will use extensively the constants defined in \cref{constants} 
and assume that $\alpha$ and $(1,\xi)$ are fixed. While looking at the pictures 
in the proofs it is also important to keep in mind the following observation: 
Our definition of $W(v,\Theta)$ allows for $\xi$ to lie on $[w_0,w_0^\prime]$. 
However, in what follows we often assume the existence of a $v_0(\chi)$ for 
some $\chi\geq\Upsilon_0$, which forces $\xi$ to lie outside of 
$[w_0,w_0^\prime]$.

\begin{lemma}\label{lem:Binnv0chiisBalpha1}Let $\chi\geq\Upsilon_0$. If 
$v_0(\chi)$ exists, then $\binn{v_0(\chi)}=B_\alpha(1)$.
\end{lemma}

\begin{proof}Assume there is $h\in B_\alpha(1)$ such that $h\not\in 
\binn{v_0(\chi)}$. Thus, $v_h\in (v_0(\chi),\xi]$. Then by uniqueness of 
geodesics in $T$ and the construction of $\Upsilon_0$ it follows that 
$[\anf(1,\xi),v_0(\chi)+_{\gamma_1}1]$ is $\Upsilon_0$-small 
(cf.~\cref{fig:BinnandBext} with $h=h_4$ and $v_0(\chi)=v$). A contradiction.
\end{proof}

In general, the \squote{concatenation} of two $\chi$-small geodesics $[x,y]$ 
(ending in a vertex of $T$) and $[y,z]$ (starting in a vertex of $T$) is not 
$\chi$-small. However, if two $\chi$-small geodesics overlap in a sufficiently 
large terminal segment, we can conclude in this way. This is made precise in 
the following lemma, whose proof is entirely encapsulated in \cref{fig:overlap}.

\begin{lemma}\label{lem:overlapping}Let $\chi\geq 0$ and let $[x,z]$ be a 
geodesic. If there are $y,y^\prime$ on $[x,z]$ such that the intersection 
$[x,y]\cap[y^\prime,z]=[y^\prime,y]$ has length $\geq k$ and both $[x,y]$ and 
$[y^\prime,z]$ are $\chi$-small, then $[x,z]$ is $\chi$-small.
\end{lemma}

\begin{figure}[h!]
\centering
\def\svgwidth{\textwidth}
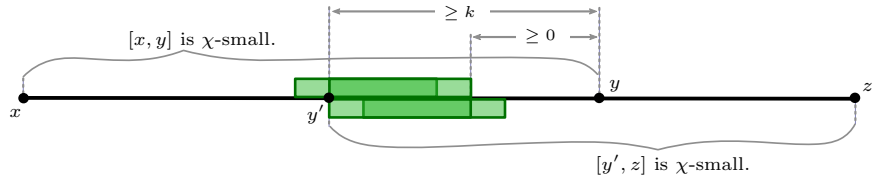
\caption{The \squote{union} of two $\chi$-small geodesics overlapping in a 
segment of length $\geq k$ is again $\chi$-small.}
\label{fig:overlap}
\end{figure}

In the sequel we will use this way of reasoning without explicitly referencing 
the above lemma.

\begin{lemma}\label{lem:smallnessinheritence}Let $\chi\geq\Upsilon_0$ and $x$ 
be a vertex on $\gamma_1$ such that $[\anf(1,\xi),x]$ is $\chi$-small. Then the 
geodesic $[\anf(h,\xi),x]$ is $\Upsilon_1(\chi)$-small for all $h\in 
B_\alpha(1)$.
\end{lemma}

\begin{proof} 
If $h\in \bext x$, then $[\anf(h,\xi),x]\subset 
[\anf(h,\xi),\anf(1,\xi)]\subset L_0$. Hence, $[\anf(h,\xi),x]$ is 
$\Upsilon_0$-small. If $h\in \binn x$, then we distinguish cases upon where 
$v_h$ lies on $[\anf(1,\xi),x]$:
\begin{enumerate}[label=\alph*),leftmargin=*,align=left]
\item This case is clear:
\begin{center}
\def\svgwidth{1.05\textwidth}
\begingroup%
  \makeatletter%
  \providecommand\color[2][]{%
    \errmessage{(Inkscape) Color is used for the text in Inkscape, but the package 'color.sty' is not loaded}%
    \renewcommand\color[2][]{}%
  }%
  \providecommand\transparent[1]{%
    \errmessage{(Inkscape) Transparency is used (non-zero) for the text in Inkscape, but the package 'transparent.sty' is not loaded}%
    \renewcommand\transparent[1]{}%
  }%
  \providecommand\rotatebox[2]{#2}%
  \ifx\svgwidth\undefined%
    \setlength{\unitlength}{400bp}%
    \ifx\svgscale\undefined%
      \relax%
    \else%
      \setlength{\unitlength}{\unitlength * \real{\svgscale}}%
    \fi%
  \else%
    \setlength{\unitlength}{\svgwidth}%
  \fi%
  \global\let\svgwidth\undefined%
  \global\let\svgscale\undefined%
  \makeatother%
  \begin{picture}(1,0.04069699)%
    \put(0,0){\includegraphics[width=\unitlength]{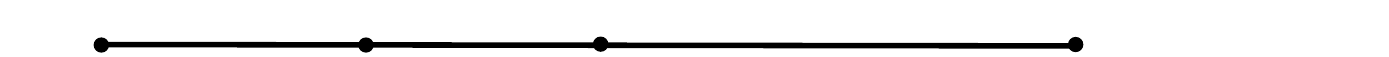}}%
    \put(0.00981292,0.02092148){\color[rgb]{0,0,0}\makebox(0,0)[lb]{\smash{\kll{\anf(1,\xi)}}}}%
    \put(0.43570708,0.02093927){\color[rgb]{0,0,0}\makebox(0,0)[lb]{\smash{\kll{x}}}}%
    \put(0.79063317,0.0032334){\color[rgb]{0,0,0}\makebox(0,0)[lb]{\smash{\kll{\xi\in\Delta(T)}}}}%
    \put(0.20193209,0.02098263){\color[rgb]{0,0,0}\makebox(0,0)[lb]{\smash{\kll{\anf(h,\xi)=v_h}}}}%
  \end{picture}%
\endgroup%

\end{center}
Thus, we assume from now on, that $\anf(h,\xi)\not\in[\anf(1,\xi),x]$.
\item This one is clear as well: 
\begin{center}
\def\svgwidth{1.05\textwidth}
\begingroup%
  \makeatletter%
  \providecommand\color[2][]{%
    \errmessage{(Inkscape) Color is used for the text in Inkscape, but the package 'color.sty' is not loaded}%
    \renewcommand\color[2][]{}%
  }%
  \providecommand\transparent[1]{%
    \errmessage{(Inkscape) Transparency is used (non-zero) for the text in Inkscape, but the package 'transparent.sty' is not loaded}%
    \renewcommand\transparent[1]{}%
  }%
  \providecommand\rotatebox[2]{#2}%
  \ifx\svgwidth\undefined%
    \setlength{\unitlength}{400bp}%
    \ifx\svgscale\undefined%
      \relax%
    \else%
      \setlength{\unitlength}{\unitlength * \real{\svgscale}}%
    \fi%
  \else%
    \setlength{\unitlength}{\svgwidth}%
  \fi%
  \global\let\svgwidth\undefined%
  \global\let\svgscale\undefined%
  \makeatother%
  \begin{picture}(1,0.15792056)%
    \put(0,0){\includegraphics[width=\unitlength]{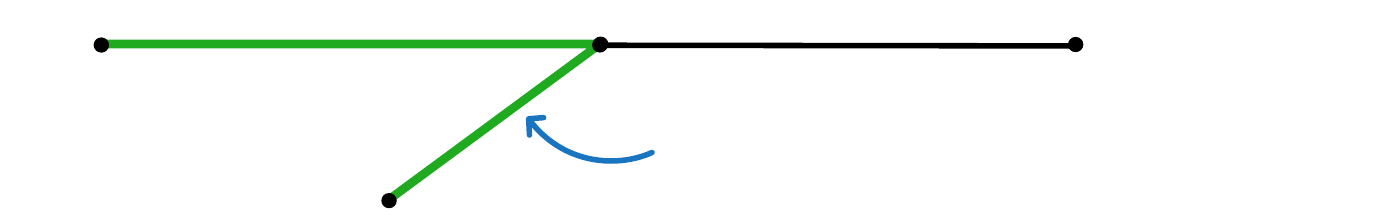}}%
    \put(0.01160886,0.13767019){\color[rgb]{0,0,0}\makebox(0,0)[lb]{\smash{\kll{\anf(1,\xi)}}}}%
    \put(0.43921444,0.13779153){\color[rgb]{0,0,0}\makebox(0,0)[lb]{\smash{\kll{x=v_h}}}}%
    \put(0.79063317,0.12045697){\color[rgb]{0,0,0}\makebox(0,0)[lb]{\smash{\kll{\xi\in\Delta(T)}}}}%
    \put(0.20857857,0.00918847){\color[rgb]{0,0,0}\makebox(0,0)[lb]{\smash{\kll{\anf(h,\xi)}}}}%
    \put(0.47605016,0.04332406){\color[rgb]{0,0,0}\makebox(0,0)[lb]{\smash{\kll{\text{Contained in}\ L_0\text{ and thus } \Upsilon_0\text{-small.}}}}}%
  \end{picture}%
\endgroup%

\end{center}
\item If $v_h$ lies in the interior of $[\anf(1,\xi),x]$ and 
$[\anf(1,\xi),v_h]$ has length $\geq k$,
\begin{center}
\def\svgwidth{1.05\textwidth}
\begingroup%
  \makeatletter%
  \providecommand\color[2][]{%
    \errmessage{(Inkscape) Color is used for the text in Inkscape, but the package 'color.sty' is not loaded}%
    \renewcommand\color[2][]{}%
  }%
  \providecommand\transparent[1]{%
    \errmessage{(Inkscape) Transparency is used (non-zero) for the text in Inkscape, but the package 'transparent.sty' is not loaded}%
    \renewcommand\transparent[1]{}%
  }%
  \providecommand\rotatebox[2]{#2}%
  \ifx\svgwidth\undefined%
    \setlength{\unitlength}{400bp}%
    \ifx\svgscale\undefined%
      \relax%
    \else%
      \setlength{\unitlength}{\unitlength * \real{\svgscale}}%
    \fi%
  \else%
    \setlength{\unitlength}{\svgwidth}%
  \fi%
  \global\let\svgwidth\undefined%
  \global\let\svgscale\undefined%
  \makeatother%
  \begin{picture}(1,0.21072739)%
    \put(0,0){\includegraphics[width=\unitlength]{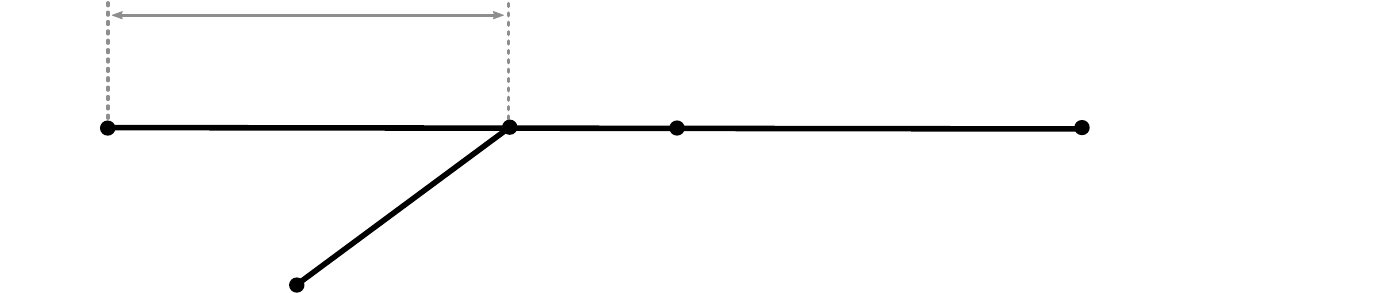}}%
    \put(0.01405707,0.12997562){\color[rgb]{0,0,0}\makebox(0,0)[lb]{\smash{\kll{\anf(1,\xi)}}}}%
    \put(0.37197635,0.13080408){\color[rgb]{0,0,0}\makebox(0,0)[lb]{\smash{\kll{v_h}}}}%
    \put(0.7952027,0.11346951){\color[rgb]{0,0,0}\makebox(0,0)[lb]{\smash{\kll{\xi\in\Delta(T)}}}}%
    \put(0.14364557,0.00269366){\color[rgb]{0,0,0}\makebox(0,0)[lb]{\smash{\kll{\anf(h,\xi)}}}}%
    \put(0.48927655,0.13013385){\color[rgb]{0,0,0}\makebox(0,0)[lb]{\smash{\kll{x}}}}%
    \put(0.18256952,0.19401254){\color[rgb]{0,0,0}\makebox(0,0)[lb]{\smash{$\sm{\geq k}$}}}%
  \end{picture}%
\endgroup%

\end{center}
then the first $k$ edges on $[v_h,x]$ (if there are this many) are added to 
$L_0$ in the process of building $L_{k}(\chi)$, since $[\anf(1,\xi),v_h]$ is 
$\chi$-small (cf.~\cref{constants}). Thus, $[\anf (h,\xi), v_h+_{\gamma_1} 
k]$ (resp. $[\anf(h,\xi),x]$ if $[v_h,x]$ is of length $< k$) is 
$\Upsilon_1(\chi)$-small. Since $[v_h,x]$ was $\chi$-small by assumption, it 
follows that $[\anf (h,\xi),x]$ is $\Upsilon_1(\chi)$-small.
\item If $v_h$ lies in the interior of $[\anf(1,\xi),x]$ and 
$[\anf(1,\xi),v_h]$ has length $< k$, there are two subcases.
The subcase where $x\in [\anf(1,\xi), \meps]$ is trivial:
\begin{center}
\def\svgwidth{1.05\textwidth}
\begingroup%
  \makeatletter%
  \providecommand\color[2][]{%
    \errmessage{(Inkscape) Color is used for the text in Inkscape, but the package 'color.sty' is not loaded}%
    \renewcommand\color[2][]{}%
  }%
  \providecommand\transparent[1]{%
    \errmessage{(Inkscape) Transparency is used (non-zero) for the text in Inkscape, but the package 'transparent.sty' is not loaded}%
    \renewcommand\transparent[1]{}%
  }%
  \providecommand\rotatebox[2]{#2}%
  \ifx\svgwidth\undefined%
    \setlength{\unitlength}{400bp}%
    \ifx\svgscale\undefined%
      \relax%
    \else%
      \setlength{\unitlength}{\unitlength * \real{\svgscale}}%
    \fi%
  \else%
    \setlength{\unitlength}{\svgwidth}%
  \fi%
  \global\let\svgwidth\undefined%
  \global\let\svgscale\undefined%
  \makeatother%
  \begin{picture}(1,0.2174245)%
    \put(0,0){\includegraphics[width=\unitlength]{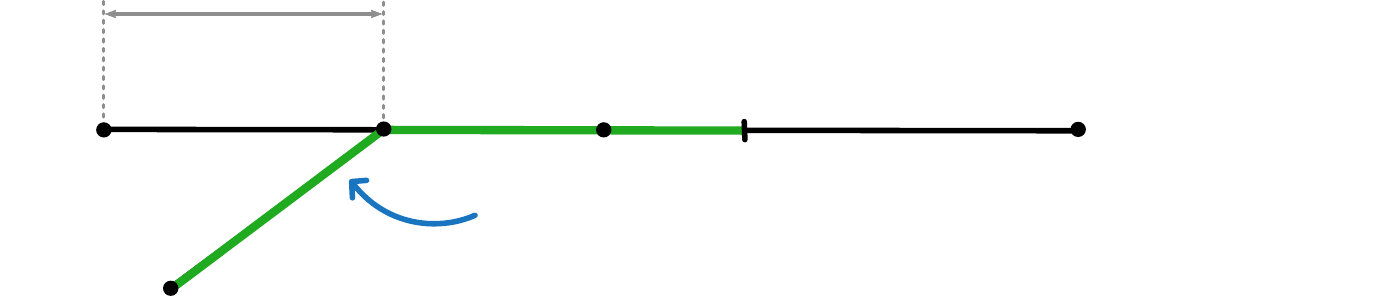}}%
    \put(0.00848528,0.13606961){\color[rgb]{0,0,0}\makebox(0,0)[lb]{\smash{\kll{\anf(1,\xi)}}}}%
    \put(0.28104061,0.13665543){\color[rgb]{0,0,0}\makebox(0,0)[lb]{\smash{\kll{v_h}}}}%
    \put(0.79245934,0.11885638){\color[rgb]{0,0,0}\makebox(0,0)[lb]{\smash{\kll{\xi\in\Delta(T)}}}}%
    \put(0.05390548,0.00641632){\color[rgb]{0,0,0}\makebox(0,0)[lb]{\smash{\kll{\anf(h,\xi)}}}}%
    \put(0.5378261,0.13586389){\color[rgb]{0,0,0}\makebox(0,0)[lb]{\smash{\kll{\meps}}}}%
    \put(0.43740389,0.13628525){\color[rgb]{0,0,0}\makebox(0,0)[lb]{\smash{\kll{x}}}}%
    \put(0.15626341,0.20427101){\color[rgb]{0,0,0}\makebox(0,0)[lb]{\smash{\sm{<k}}}}%
    \put(0.34858669,0.05762793){\color[rgb]{0,0,0}\makebox(0,0)[lb]{\smash{\kll{\text{Contained in}\ L_0\text{ and thus } \Upsilon_0\text{-small.}}}}}%
  \end{picture}%
\endgroup%

\end{center}
The subcase where $x\not\in [\anf(1,\xi), \meps]$ looks as follows:
\begin{center}
\def\svgwidth{1.05\textwidth}
\begingroup%
  \makeatletter%
  \providecommand\color[2][]{%
    \errmessage{(Inkscape) Color is used for the text in Inkscape, but the package 'color.sty' is not loaded}%
    \renewcommand\color[2][]{}%
  }%
  \providecommand\transparent[1]{%
    \errmessage{(Inkscape) Transparency is used (non-zero) for the text in Inkscape, but the package 'transparent.sty' is not loaded}%
    \renewcommand\transparent[1]{}%
  }%
  \providecommand\rotatebox[2]{#2}%
  \ifx\svgwidth\undefined%
    \setlength{\unitlength}{400bp}%
    \ifx\svgscale\undefined%
      \relax%
    \else%
      \setlength{\unitlength}{\unitlength * \real{\svgscale}}%
    \fi%
  \else%
    \setlength{\unitlength}{\svgwidth}%
  \fi%
  \global\let\svgwidth\undefined%
  \global\let\svgscale\undefined%
  \makeatother%
  \begin{picture}(1,0.21421857)%
    \put(0,0){\includegraphics[width=\unitlength]{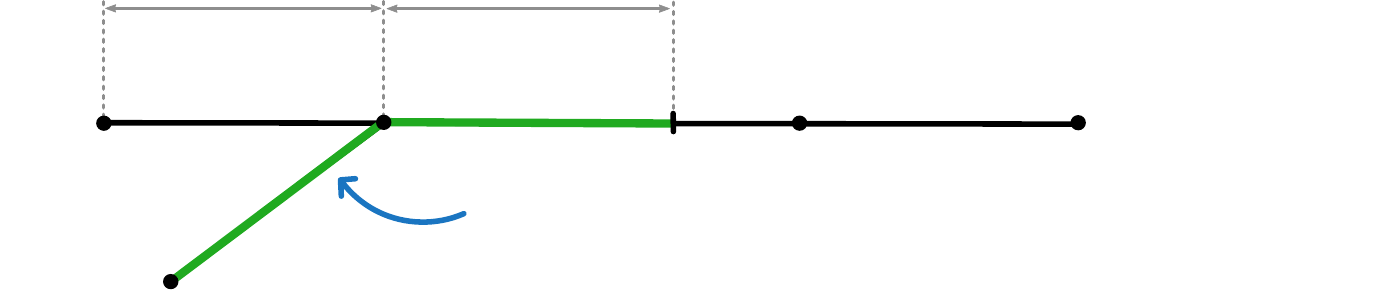}}%
    \put(0.01131371,0.13837729){\color[rgb]{0,0,0}\makebox(0,0)[lb]{\smash{\kll{\anf(1,\xi)}}}}%
    \put(0.28116192,0.13820575){\color[rgb]{0,0,0}\makebox(0,0)[lb]{\smash{\kll{v_h}}}}%
    \put(0.79245934,0.12045697){\color[rgb]{0,0,0}\makebox(0,0)[lb]{\smash{\kll{\xi\in\Delta(T)}}}}%
    \put(0.05249126,0.00837045){\color[rgb]{0,0,0}\makebox(0,0)[lb]{\smash{\kll{\anf(h,\xi)}}}}%
    \put(0.57824029,0.138){\color[rgb]{0,0,0}\makebox(0,0)[lb]{\smash{\kll{x}}}}%
    \put(0.48899517,0.138){\color[rgb]{0,0,0}\makebox(0,0)[lb]{\smash{\kll{\meps}}}}%
    \put(0.16016831,0.20511339){\color[rgb]{0,0,0}\makebox(0,0)[lb]{\smash{\sm{<k}}}}%
    \put(0.35641125,0.20442464){\color[rgb]{0,0,0}\makebox(0,0)[lb]{\smash{\sm{\geq k}}}}%
    \put(0.34058669,0.05562794){\color[rgb]{0,0,0}\makebox(0,0)[lb]{\smash{\kll{\text{Contained in}\ L_0\text{ and thus } \Upsilon_0\text{-small.}}}}}%
  \end{picture}%
\endgroup%

\end{center}
Since $[v_h,\meps]$ has length $\geq k$, $[\anf (h,\xi),\meps]$ is 
$\Upsilon_0$-small and $[v_h,x]$ is $\chi$-small, the geodesic 
$[\anf(h,\xi),x]$  is $\chi$-small.
\item This case is trivial again:
\begin{center}
\def\svgwidth{1.05\textwidth}
\begingroup%
  \makeatletter%
  \providecommand\color[2][]{%
    \errmessage{(Inkscape) Color is used for the text in Inkscape, but the package 'color.sty' is not loaded}%
    \renewcommand\color[2][]{}%
  }%
  \providecommand\transparent[1]{%
    \errmessage{(Inkscape) Transparency is used (non-zero) for the text in Inkscape, but the package 'transparent.sty' is not loaded}%
    \renewcommand\transparent[1]{}%
  }%
  \providecommand\rotatebox[2]{#2}%
  \ifx\svgwidth\undefined%
    \setlength{\unitlength}{400bp}%
    \ifx\svgscale\undefined%
      \relax%
    \else%
      \setlength{\unitlength}{\unitlength * \real{\svgscale}}%
    \fi%
  \else%
    \setlength{\unitlength}{\svgwidth}%
  \fi%
  \global\let\svgwidth\undefined%
  \global\let\svgscale\undefined%
  \makeatother%
  \begin{picture}(1,0.115218)%
    \put(0,0){\includegraphics[width=\unitlength]{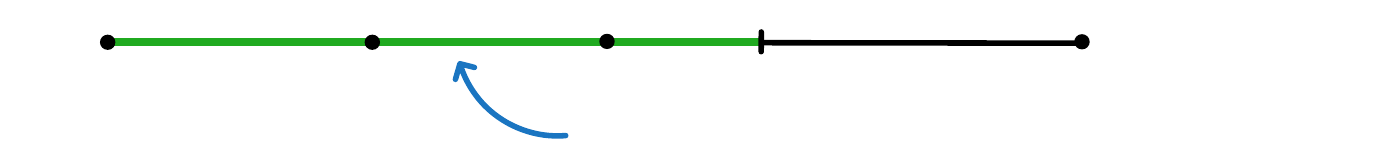}}%
    \put(0.20874335,0.0964853){\color[rgb]{0,0,0}\makebox(0,0)[lb]{\smash{\kll{\anf(1,\xi)=v_h}}}}%
    \put(0.44378394,0.09707111){\color[rgb]{0,0,0}\makebox(0,0)[lb]{\smash{\kll{x}}}}%
    \put(0.7952027,0.07973655){\color[rgb]{0,0,0}\makebox(0,0)[lb]{\smash{\kll{\xi\in\Delta(T)}}}}%
    \put(0.01480497,0.09634672){\color[rgb]{0,0,0}\makebox(0,0)[lb]{\smash{\kll{\anf(h,\xi)}}}}%
    \put(0.55306946,0.09750315){\color[rgb]{0,0,0}\makebox(0,0)[lb]{\smash{\kll{\meps}}}}%
    \put(0.41339807,0.00919037){\color[rgb]{0,0,0}\makebox(0,0)[lb]{\smash{\kll{\text{Contained in}\ L_0\text{ and thus } \Upsilon_0\text{-small.}}}}}%
  \end{picture}%
\endgroup%

\end{center}
\item And finally:
\begin{center}
\def\svgwidth{1.05\textwidth}
\begingroup%
  \makeatletter%
  \providecommand\color[2][]{%
    \errmessage{(Inkscape) Color is used for the text in Inkscape, but the package 'color.sty' is not loaded}%
    \renewcommand\color[2][]{}%
  }%
  \providecommand\transparent[1]{%
    \errmessage{(Inkscape) Transparency is used (non-zero) for the text in Inkscape, but the package 'transparent.sty' is not loaded}%
    \renewcommand\transparent[1]{}%
  }%
  \providecommand\rotatebox[2]{#2}%
  \ifx\svgwidth\undefined%
    \setlength{\unitlength}{400bp}%
    \ifx\svgscale\undefined%
      \relax%
    \else%
      \setlength{\unitlength}{\unitlength * \real{\svgscale}}%
    \fi%
  \else%
    \setlength{\unitlength}{\svgwidth}%
  \fi%
  \global\let\svgwidth\undefined%
  \global\let\svgscale\undefined%
  \makeatother%
  \begin{picture}(1,0.17684238)%
    \put(0,0){\includegraphics[width=\unitlength]{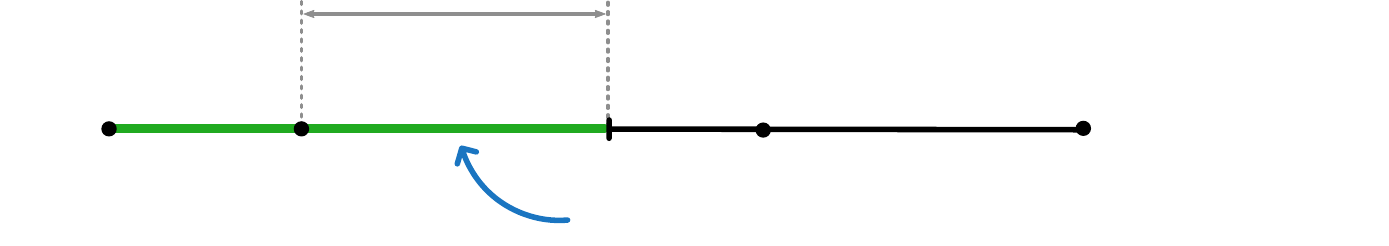}}%
    \put(0.22230234,0.0966517){\color[rgb]{0,0,0}\makebox(0,0)[lb]{\smash{\kll{\anf(1,\xi)=v_h}}}}%
    \put(0.55213584,0.09606592){\color[rgb]{0,0,0}\makebox(0,0)[lb]{\smash{\kll{x}}}}%
    \put(0.79617594,0.07902426){\color[rgb]{0,0,0}\makebox(0,0)[lb]{\smash{\kll{\xi\in\Delta(T)}}}}%
    \put(0.0186066,0.09704865){\color[rgb]{0,0,0}\makebox(0,0)[lb]{\smash{\kll{\anf(h,\xi)}}}}%
    \put(0.44504267,0.09879086){\color[rgb]{0,0,0}\makebox(0,0)[lb]{\smash{\kll{\meps}}}}%
    \put(0.30112525,0.16448827){\color[rgb]{0,0,0}\makebox(0,0)[lb]{\smash{\sm{\geq 2k}}}}%
    \put(0.41474574,0.00998985){\color[rgb]{0,0,0}\makebox(0,0)[lb]{\smash{\kll{\text{Contained in}\ L_0\text{ and thus } \Upsilon_0\text{-small.}}}}}%
  \end{picture}%
\endgroup%

\end{center}
Again, the geodesic $[\anf (h,\xi),x]$ is $\chi$-small since $[\anf 
(h,\xi),\meps]$ is $\Upsilon_0$-small, the geodesic $[\anf (1,\xi),x]$ is 
$\chi$-small and $[\anf(1,\xi),\meps]$ has length $\geq k$.\qedhere
\end{enumerate}  
\end{proof}

\begin{lemma}\label{lem:chiversionofBinnv0minuskisallofBalphafinal}Let 
$\chi\geq\Upsilon_0$. If $v_0(\chi)$ exists and 
$\binn{v_0(\chi)-_{\gamma_1}k}=B_\alpha(1)$, then one of the following 
statements holds:
\begin{enumerate}[1)]
\item $B_\alpha(1)\times\{\xi\}\subset W(v_0(\chi),\Upsilon_1(\chi))$;
\item $v_0(\Upsilon_1(\chi))$ does not exist, i.e.~$[\anf(1,\xi),\xi]$ is 
$\Upsilon_1(\chi)$-small;
\item $v_0(\Upsilon_1(\chi))$ exists
and $B_\alpha(1)\times\{\xi\}\subset W(v_0(\Upsilon_1(\chi)),\Upsilon_1(\chi))$.
\end{enumerate}
\end{lemma}

\begin{proof}In this proof we abbreviate $d_{E^{k}}$ to $d$.
For all $h\in B_\alpha(1)$, the geodesic $[\anf(h,\xi),v_0(\chi)]$ is 
$\Upsilon_1(\chi)$-small by \cref{lem:smallnessinheritence} (applied to 
$x=v_0(\chi)$). Since $B_\alpha(1)=\binn{v_0(\chi)-_{\gamma_1}k}$, the 
geodesics $[v_h,v_0(\chi)]\subset \gamma_h$ all have length $\geq k$. Hence, 
for all $h\in B_\alpha(1)$, the measurement at $v_0(\chi)$ on $\gamma_1$ is the 
measurement at $v_0(\chi)$ on $\gamma_h$. Thus, if the measurement at 
$v_0(\chi)$ on $\gamma_1$ is $\Upsilon_1(\chi)$-large, so is the measurement at 
$v_0(\chi)$ on $\gamma_h$ and $B_\alpha(1)\times\{\xi\}\subset 
W(v_0(\chi),\Upsilon_1(\chi))$ follows, which is statement 1) above.

Assume from now on that the measurement at $v_0(\chi)$ on $\gamma_1$ is 
$\Upsilon_1(\chi)$-small. If $v_0(\Upsilon_1(\chi))$ does not exist, then 
$[\anf(1,\xi),\xi]$ is $\Upsilon_1(\chi)$-small, which is statement 2) above.  
If $v_0(\Upsilon_1(\chi))$ exists, it lies in $(v_0(\chi),\xi]$. Thus, the 
intersection of the $\Upsilon_1(\chi)$-small geodesic $[\anf(h,\xi),v_0(\chi)]$ 
with the $\Upsilon_1(\chi)$-small geodesic 
$[v_0(\chi)-_{\gamma_1}k,v_0(\Upsilon_1(\chi))]$ has length $\geq k$ for all 
$h\in B_\alpha(1)$. Thus, all $[\anf(h,\xi),v_0(\Upsilon_1(\chi))]$ are still 
$\Upsilon_1(\chi)$-small and the measurement at $v_0(\Upsilon_1(\chi))$ on 
$\gamma_1$ is the measurement at $v_0(\Upsilon_1(\chi))$ on $\gamma_h$. Hence, 
$B_\alpha(1)\times\{\xi\}\subset W(v_0(\Upsilon_1(\chi)),\Upsilon_1(\chi))$, 
which is statement 3) above.
\end{proof}

\begin{lemma}\label{lem:iterationstep}Let $\chi\geq\Upsilon_0$. Assume that 
$v_0(\chi)$ exists and let $0\leq r\leq k-1$ such that 
$\bext{v_0(\chi)-_{\gamma_1}r}=\emptyset$, but 
$\bext{v_0(\chi)-_{\gamma_1}(r+1)}\neq\emptyset$.  Then one of the following 
statements holds:
\begin{enumerate}[1)]
\item $B_\alpha(1)\times\{\xi\}\subset W(v_0(\chi),\Upsilon_1(\chi))$;
\item $v_0(\Upsilon_2(\chi))$ does not exist, i.e.~$[\anf(1,\xi),\xi]$ is 
$\Upsilon_2(\chi)$-small;
\item $v_0(\Upsilon_2(\chi))$ exists and there is an $s\geq 1$ such that  
$\bext {v_0(\Upsilon_2(\chi))-_{\gamma_1} (r+s)}=\emptyset$, but\newline 
\mbox{$\bext{v_0(\Upsilon_2(\chi))-_{\gamma_1} (r+s+1)}\neq\emptyset$}.
\end{enumerate}
\end{lemma}\par\smallskip

\begin{proof} In this proof we abbreviate $d_{E^{k}}$ to $d$ again.
As in the previous lemma, for all $h\in B_\alpha(1)$ the geodesics 
$[\anf(h,\xi),v_0(\chi)]$ is $\Upsilon_1(\chi)$-small by 
\cref{lem:smallnessinheritence} (applied to $x=v_0(\chi)$). The measurement 
at $v_0(\chi)$ on $[\anf(h,\xi),\xi]$ is 
\begin{equation*}
d(\sigma_{v_0(\chi)}(\gamma_h),\sigma_{v_0(\chi)+1}(\gamma_h)).
\end{equation*}
Note that, since $\binn{v_0(\chi)}=B_\alpha(1)$ by 
\cref{lem:Binnv0chiisBalpha1}, 
$v_0(\chi)+_{\gamma_1}1=v_0(\chi)+_{\gamma_h}1$. So the $+1$ in the second 
argument above is unambiguous. However, in general 
$\sigma_{v_0(\chi)+1}(\gamma_h)\neq\sigma_{v_0(\chi)+1}(\gamma_1)$ and this is 
the point were we have to put in a little work.

If for all $h\in B_\alpha(1)$ the measurement at $v_0(\chi)$ on $\gamma_h$ is 
$\Upsilon_1(\chi)$-large, then
\begin{equation*}
B_\alpha(1)\times\{\xi\}\subset W(v_0(\chi),\Upsilon_1(\chi))
\end{equation*}
by definition of $W(v_0(\chi),\Upsilon_1(\chi))$, and this is statement 1) 
above. So assume from now on that there is an $h_1\in B_\alpha(1)$ such that 
the measurement at $v_0(\chi)$ on $\gamma_{h_1}$ is $\Upsilon_1(\chi)$-small. 
Moreover, we can assume that $v_0(\Upsilon_2(\chi))$ does exist, since the case 
that $v_0(\Upsilon_2(\chi))$ does not exist, is exactly statement 2) above.

Since $\bext{v_0(\chi)-_{\gamma_1}r}=\emptyset$, but 
$\bext{v_0(\chi)-_{\gamma_1}(r+1)}\neq\emptyset$, there is $h_0\in B_\alpha(1)$ 
with $v_{h_0}=v_0(\chi)-_{\gamma_1}r$. So for this $h_0$ the geodesic 
$[\anf(1,\xi),\anf(h_0,\xi)]$ runs through $v_0(\chi)-_{\gamma_1}r$.
Since, for all $h\in B_\alpha(1)$, the vertex $v_h$ is contained in 
$[\anf(1,\xi),v_{h_0}]$, it follows that 
$[\anf(h,\xi),v_0(\chi)-_{\gamma_1}r]\subset L_0$ holds for all $h$.
Assume $h_1\meps\not\in [\anf(h_1,\xi),v_0(\chi)]$. Then we must have 
$[\anf(1,\xi),v_0(\chi)+1]\subset L_0$ - a contradiction to the definition of 
$v_0(\chi)$, since $\chi\geq\Upsilon_0$. Thus, $h_1\meps\in 
[\anf(h_1,\xi),v_0(\chi)]$, and hence 
$d_T(\anf(h_1,\xi),v_0(\chi)-_{\gamma_1}r)\geq k$.
Combining the facts that $[\anf (h_1,\xi),v_0(\chi)]$ is 
$\Upsilon_1(\chi)$-small (by \cref{lem:smallnessinheritence}) and that 
$r+1\leq k$ with our assumption that the measurement at $v_0(\chi)$ on 
$\gamma_{h_1}$ is $\Upsilon_1(\chi)$-small, it follows that the segment 
$[v_0(\chi)-_{\gamma_1}r,v_0(\chi)+_{\gamma_1}1]$ is contained in 
$L_{k}(\Upsilon_1(\chi))$ (cf.~\cref{fig:addedsegment} and 
\cref{constants}). Note that, if $r=k-1$, then we indeed need all $k$ steps in 
the iteration process used to build $L_{k}(\Upsilon_1(\chi))$ from $L_0$ in 
order to add this whole segment.
\begin{figure}[h!]
\centering
\def\svgwidth{\textwidth}
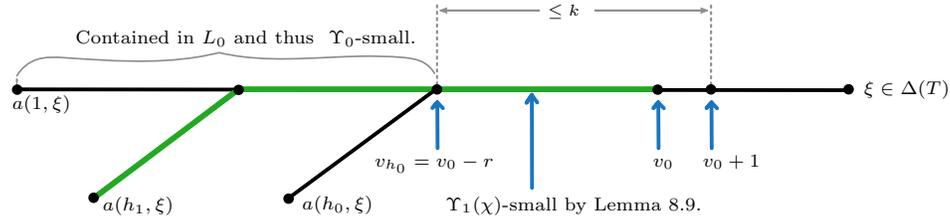
\caption{The segment $[v_0-_{\gamma_1}r,v_0+_{\gamma_1}1]$, with 
$v_0=v_0(\chi)$, is contained in $L_k(\Upsilon_1(\chi))$.}
\label{fig:addedsegment}
\end{figure}

\noindent
Now recall that $\Upsilon_2(\chi)=\Upsilon_1(\Upsilon_1(\chi))$. Hence, $[\anf 
(h,\xi),v_0(\chi)+1]\subset L_{k}(\Upsilon_1(\chi))$ is 
$\Upsilon_2(\chi)$-small for all $h\in B_\alpha(1)$. In particular, the vertex 
$v_0(\Upsilon_2(\chi))$ has to lie on $(v_0(\chi),\xi]$ and subsequently
 $s:=d_T(v_0(\chi),v_0(\Upsilon_2(\chi)))$ is $\geq 1$.
Moreover, it is
\begin{align*}
& \bext{v_0(\Upsilon_2(\chi))-_{\gamma_1}(r+s)}= 
\bext{v_0(\chi)-_{\gamma_1}r}=\emptyset\quad\text{and}\\
& \bext{v_0(\Upsilon_2(\chi))-_{\gamma_1}(r+s+1)}= 
\bext{v_0(\chi)-_{\gamma_1}(r+1)}\neq\emptyset.\qedhere
\end{align*}
\end{proof}

We are now ready to prove \cref{prop:wideness}.

{\setlength{\parindent}{0cm}
\begin{proof}[Proof of \cref{prop:wideness}]
Set $N:=k+1$  and let $\alpha>0$ be given. Further define the following 
constants depending on $\alpha$:
\begin{itemize}
\item Let $\Upsilon_0$ be as in \cref{constants}. 
\item For $0\leq i\leq k$ define $\Theta_{i}:=\Upsilon_{2i+1}(\Upsilon_0)$ 
(also cf.~\cref{constants}).
\item Choose $\thetaminus$ such that 
$\thetaminus>\Upsilon_{2(k+1)}(\Upsilon_0)$. Note that the iterative 
construction of $\Upsilon_{2(k+1)}(\Upsilon_0)$ guarantees  $\thetaminus\geq 
\Upsilon_{i}(\Upsilon_0)$ for $1\leq i\leq 2k+1$ and, therefore, also 
$\thetaminus> \Upsilon_0$.
\item Define $\Theta_{k+1}:=\Upsilon_1(\thetaminus)$.
\end{itemize}

By \cref{lem:propertiesofmathcalWTheta}, 
$\mathcal{W}_\alpha:=\bigcup_{i=0}^{k+1}\mathcal{W}_{\Theta_i}$ is 
a $G$-invariant collection of open $\mathcal{F}_T$-subsets of 
$G\times\Delta(T)$ and has order at most $k+1$. It remains to verify 
the wideness claim. That is, for each $(1,\xi)\in 
G\times\Delta(T)$ we either have to find $0\leq i\leq k+1$ such that 
$B_\alpha(1)\times\{\xi\}\subseteq W(v,\Theta_i)$ for some $v\in V(T)$ or 
conclude that $[\anf(1,\xi),\xi]$ is $\thetaminus$-small and $\xi$ lies in 
$\partial T$.\par\medskip

Let $(1,\xi)\in G\times\Delta(T)$ and recall the notation $\gamma_h, \binn{v}$ 
and $\bext{v}$ from \cref{not:for1xi}.

\underline{\textbf{Case 1:}} Assume that $\gamma_1$ is $\thetaminus$-small. If 
$\xi\in\partial T$, there is nothing to show. If $\xi\in V_\infty(T)$, then by 
\cref{lem:smallnessinheritence} (for $x=\xi$ and $\chi=\thetaminus$) 
and the definition of $W(\xi,\Upsilon_1(\thetaminus))$ it follows that
\begin{equation*}
B_\alpha(1)\times\{\xi\}\subset 
W(\xi,\Upsilon_1(\thetaminus))=W(\xi,\Theta_{k+1}).
\end{equation*}

\underline{\textbf{Case 2:}} Assume that $\gamma_1$ is not $\thetaminus$-small. 
In particular, $\gamma_1$ is not $\Upsilon_0$-small and $v_0(\Upsilon_0)$ 
exists. Thus, by \cref{lem:Binnv0chiisBalpha1}, 
$\bext{v_0(\Upsilon_0)}=\emptyset$ and there is a unique $r_0\geq 0$ such that
\begin{equation*}
\bext{v_0(\Upsilon_0)-_{\gamma_1}r_0}=\emptyset, \text{  but}\ 
\bext{v_0(\Upsilon_0)-_{\gamma_1}(r_0+1)}\neq\emptyset.
\end{equation*}
If $r_0\geq k$, we can apply 
\cref{lem:chiversionofBinnv0minuskisallofBalphafinal} with  
$\chi=\Upsilon_0$. Since $\gamma_1$ is not $\Upsilon_1(\Upsilon_0)$-small, it 
either follows that 
\begin{equation*}B_\alpha(1)\times\{\xi\}\subset 
W(v_0(\Upsilon_0),\Upsilon_1(\Upsilon_0))=W(v_0(\Upsilon_0),\Theta_0)
\end{equation*}
or that
\begin{equation*}B_\alpha(1)\times\{\xi\}\subset 
W(v_0(\Upsilon_1(\Upsilon_0)),
\Upsilon_1(\Upsilon_0))=W(v_0(\Upsilon_1(\Upsilon_0)),\Theta_0).
\end{equation*}

If $r_0\leq k-1$, we can apply \cref{lem:iterationstep} with $r=r_0$ and 
$\chi=\Upsilon_0$. Since $\gamma_1$ is not $\Upsilon_2(\Upsilon_0)$-small, it 
either follows that
\begin{equation*}
B_\alpha(1)\times\{\xi\}\subset 
W(v_0(\Upsilon_0),\Upsilon_1(\Upsilon_0))=W(v_0(\Upsilon_0),\Theta_0)
\end{equation*}
or that $v_0(\Upsilon_2(\Upsilon_0))$ exists and that there is an $s_0\geq 1$ 
such that
\begin{equation*}
\bext{v_0(\Upsilon_2(\Upsilon_0))-_{\gamma_1} (r_0+s_0)}=\emptyset, \text{  
but}\ \bext{v_0(\Upsilon_2(\Upsilon_0))-_{\gamma_1} (r_0+s_0+1)}\neq\emptyset.
\end{equation*}

This process can be iterated for $0\leq i\leq k$.
Assume that we have already constructed an $r_i\geq 0$ with 
\begin{equation*}
\bext{v_0(\Upsilon_{2i}(\Upsilon_0))-_{\gamma_1}r_i}=\emptyset, \text{  but}\ 
\bext{v_0(\Upsilon_{2i}(\Upsilon_0))-_{\gamma_1}(r_i+1)}\neq\emptyset,
\end{equation*}
where for $i=0$ we interpret $\Upsilon_0(\Upsilon_0)$ as $\Upsilon_0$.
If $r_i\geq k$, we can apply 
\cref{lem:chiversionofBinnv0minuskisallofBalphafinal} with  
$\chi=\Upsilon_{2i}(\Upsilon_0)$. Since $\gamma_1$ is not 
$\Upsilon_{2i+1}(\Upsilon_0)$-small (as long as $i\leq k$), it either follows 
that 
\begin{equation*}B_\alpha(1)\times\{\xi\}\subset 
W(v_0(\Upsilon_{2i}(\Upsilon_0)),\Upsilon_{2i+1}(\Upsilon_0))
\end{equation*}
or that
\begin{equation*}B_\alpha(1)\times\{\xi\}\subset 
W(v_0(\Upsilon_{2i+1}(\Upsilon_0)),\Upsilon_{2i+1}(\Upsilon_0)).
\end{equation*}

If $r_i\leq k-1$, we can apply \cref{lem:iterationstep} with $r=r_i$ and 
$\chi=\Upsilon_{2i}(\Upsilon_0)$. Since $\gamma_1$ is not 
$\Upsilon_{2(i+1)}(\Upsilon_0)$-small (as long as $i\leq k$), it either follows 
that
\begin{equation*}
B_\alpha(1)\times\{\xi\}\subset 
W(v_0(\Upsilon_{2i}(\Upsilon_0)),\Upsilon_{2i+1}(\Upsilon_0))
\end{equation*}
or that $v_0(\Upsilon_{2(i+1)}(\Upsilon_0))$ exists and that there is $s_i\geq 
1$ such that
\begin{equation*}
\bext{v_0(\Upsilon_{2(i+1)}(\Upsilon_0))-_{\gamma_1} (r_i+s_i)}=\emptyset, 
\text{  but}\ \bext{v_0(\Upsilon_{2(i+1)}(\Upsilon_0))-_{\gamma_1} 
(r_i+s_i+1)}\neq\emptyset.
\end{equation*}
In this case we define $r_{i+1}:=r_i+s_i$ and repeat the process. Since in each 
step $s_i\geq 1$, there is $0\leq j\leq k$ such that $r_j\geq k$. Hence, all 
constants $\Upsilon_{2i+1}(\Upsilon_0)$ considered in the process are among 
$\Theta_0,\ldots, \Theta_k$. 
\end{proof}
}

This concludes the construction of suitable $G$-invariant collections of open 
$\mathcal{F}$-subsets that are wide for 
$G\times\Delta(T)\setminus G\times_{\thetaminus}\partial T$ (with $\thetaminus$ 
depending on 
$\alpha>0$).

\section{\texorpdfstring{The segment-flow 
space $FS_\Theta$}{The segment-flow space}}\label{sec:segmentflowspace}
The aim of this and the following section is to construct---for all 
$\Theta$---a $G$-invariant collection of open $\mathcal{F}$-subsets
that is wide for $G\times_{\Theta}\partial T$. Our construction 
will be analogously to the one in \cite{Bar17} for relatively hyperbolic 
groups. For this, we first define another proper $G$-invariant metric on 
$E^{k}(T)$ which essentially coincides with the path-metric $d_T$ when 
restricted to segments lying on a small geodesic.\par\medskip

From now on assume that $k$ is even and recall that in this case the midpoint 
$m(\sigma)$ of any $\sigma\in E^k(T)$ is a vertex of $T$.

\begin{definition}Let $\Theta>0$ and let $\sigma,\sigma^\prime\in E^{k}(T)$. 
The geodesic $\gamma=[m(\sigma), m(\sigma^\prime)]$ is called \emph{extended 
$\Theta$-small}, if the following holds:
\begin{itemize}
\item $\gamma$ is $\Theta$-small (as in \cref{def:smallnessofgeodesics}).
\item If $d_T(m(\sigma),m(\sigma^\prime))\geq k$, then 
\begin{equation*}
d_{E^{k}}(\sigma, \sigma_{o(\gamma)+_{\gamma}k})\leq\Theta
\end{equation*}
and
\begin{equation*}
d_{E^{k}}(\sigma_{t(\gamma)},\sigma^\prime)\leq\Theta.
\end{equation*}
\item If $d_T(m(\sigma),m(\sigma^\prime))\leq k-1$, then $d_{E^{k}}(\sigma, 
\sigma^\prime)\leq\Theta$.
\end{itemize}
Remember that the $+_{\gamma}$-operation on vertices of $\gamma$ is meant as 
introduced in \cref{sec:definingthetasmall}.
\end{definition}

\begin{definition}\label{def:dThetaEkplus1}An \emph{oriented segment (of length 
$k$)} in $T$ is an element $\sigma\in E^{k}(T)$ together with a vertex $b\in 
\{o(\sigma),t(\sigma)\}$.
For $\sigma,\sigma^\prime\in E^{k}(T)$ denote by 
$\mathcal{M}_\Theta(\sigma,\sigma^\prime)$ the set of all sequences 
$(\sigma_0,b_0),(\sigma_1,b_1),\ldots, (\sigma_{n},b_n)$ of oriented segments 
of length $k$ in $T$ such that $\sigma=\sigma_0$ and $\sigma^\prime=\sigma_n$ 
(as unoriented segments) and the geodesics $[m(\sigma_i),m(\sigma_{i+1})]$ are 
extended $\Theta$-small. For brevity, elements of 
$\mathcal{M}_\Theta(\sigma,\sigma^\prime)$ are called \emph{admissible 
sequences (for the pair $(\sigma,\sigma^\prime)$)} and the vertices $b_i$ are 
dropped from the notation.
Finally, we define
\begin{equation*}
d_\Theta(\sigma,\sigma^\prime):=\min_{\mathcal{M}_\Theta(\sigma,\sigma^\prime)}
\sum_{i=0}^{n-1} 
\big(d(m(\sigma_i),m(\sigma_{i+1}))+d(o(\sigma_i),o(\sigma_{i+1}
))+d(t(\sigma_i),t(\sigma_{i+1}))\big),
\end{equation*}
where $d:=d_T$ denotes the path-metric on $T$.
\end{definition}

\begin{lemma}$d_\Theta$ is a proper $G$-invariant generalised metric on 
$E^{k}(T)$. Furthermore, if $\sigma,\sigma^\prime\in E^k(T)$ both lie on the 
same $\Theta$-small geodesic, then $d_\Theta(\sigma,\sigma^\prime)=3\cdot 
d_T(m(\sigma),m(\sigma^\prime))$.
\end{lemma}

\begin{proof}
The $G$-invariance of $d_\Theta$ follows from the $G$-invariance of both $d_T$ 
and $d_{E^k}$. Symmetry follows from 
$\mathcal{M}_\Theta(\sigma,\sigma^\prime)=\mathcal{M}_\Theta(\sigma^\prime,
\sigma)$. The triangle inequality holds, since we can concatenate an admissible 
sequence realising $d_\Theta(\sigma,\sigma^\prime)$ with an admissible sequence 
realising $d_\Theta(\sigma^\prime,\sigma^{\prime\prime})$ to obtain an 
admissible sequence for the pair $(\sigma,\sigma^{\prime\prime})$. Taking mid-, 
start- and endpoints of segments into account for the definition of 
$d_\Theta(\sigma,\sigma^\prime)$ guarantees that $d_\Theta$ is indeed a 
generalised metric (instead of a generalised pseudo-metric).

For two elements $\sigma$ and $\sigma^\prime$ in $E^k(T)$ lying on the same 
$\Theta$-small geodesic, the sequence $(\sigma, o(\sigma)), (\sigma^\prime, 
o(\sigma^\prime))$ is an admissible sequence and yields 
$d_\Theta(\sigma,\sigma^\prime)\leq 3\cdot d_T(m(\sigma), m(\sigma^\prime))$. 
Combined use of the triangle inequality for $d_T$ and the fact that 
$\sigma,\sigma^\prime$ lie on the same $\Theta$-small geodesic gives that the 
admissible sequence $(\sigma, o(\sigma)), (\sigma^\prime, o(\sigma^\prime))$ 
already realises $d_\Theta(\sigma,\sigma^\prime)$. It remains to show that 
$d_\Theta$ is proper:

Let $R>0$ and $\sigma\in E^{k}(T)$ be given. Note that, since the distance 
between two vertices (with respect to $d_T$) is an integer, if 
$\sigma\neq\sigma^\prime$, then 
$d_{\Theta}(\sigma,\sigma^\prime)\geq 1$. So, if $\sigma^\prime\in 
B_R^{d_\Theta}(\sigma)$, then there is an admissible sequence 
$(\sigma_0,b_0),(\sigma_1,b_1),\ldots, (\sigma_{r},b_r)$, $r\leq R$ for the 
pair $(\sigma,\sigma^\prime)$ such that $d_{T}(m(\sigma_i),m(\sigma_{i+1}))\leq 
R$ for all $0\leq i\leq r-1$.

Since $d_{E^{k}}$ is proper, given $\sigma_i$ there are only finitely many 
segments $\sigma\in E^{k}(T)$ with $d_{E^{k}}(\sigma_i,\sigma)\leq\Theta$. 
Thus, since $[m(\sigma_i),m(\sigma_{i+1})]$ has to be extended $\Theta$-small 
and of length $\leq R$ (with respect to $d_T$), there are only finitely many 
possibilities for $\sigma_1$. And given any of those possible $\sigma_1$, there 
are only finitely many possibilities for $\sigma_2$. Iterating this argument 
$R$ times we conclude that there are only finitely many possibilities for 
$\sigma_r=\sigma^\prime$.
\end{proof}

Next we will define a segment-flow space $FS_\Theta$ as a certain subset of the 
space $E^{k}(T)\times\Delta_+(T)\times\Delta_+(T)$. The definition given here 
is a slight adaptation of the coarse $\Theta$-flow space in \cite[Definition 
3.4]{Bar17} to our situation on trees. 

\begin{definition}Let $\Theta>0$. Let $Z:=\Delta_+(T)\times\Delta_+(T)$ be 
equipped with the subspace topology. Furthermore, let 
\begin{equation*}
FS_\Theta:=\{ (\sigma,\xi_-,\xi_+)\in E^k(T)\times Z\ |\ \sigma\subset 
[\xi_-,\xi_+]\ \text{and}\ [\xi_-,\xi_+]\ \text{is $\Theta$-small}\},
\end{equation*}
where the notion of a $\Theta$-small geodesic is as in 
\cref{def:smallnessofgeodesics}.
\end{definition}

Before we list properties of $FS_\Theta$, we cite a definition of Bartels, 
which we will need briefly while applying \cite[Theorem 1.1]{Bar17} to obtain 
long thin covers for $FS_\Theta$.

\begin{definition}(see \cite[p.~750, before Theorem 1.1]{Bar17}) Let $(V,d)$ 
be a generalised discrete metric space. Let $\alpha\geq 0$. A subset $S\subset 
V$ is \emph{$\alpha$-separated} if $d(x,y)\geq \alpha$ for any two distinct 
points $x,y\in S$.
A subspace $V_0\subset V$ has the \emph{$(D,R)$-doubling property} if for all 
$\alpha\geq R$ the following holds:  if $S\subset V_0$ is $\alpha$-separated 
and contained in a $2\alpha$-ball, then the cardinality of $S$ is at most $D$.
\end{definition}

The following lemma collects the properties of $E^k(T),Z$ and $FS_\Theta$ 
needed to apply \cite[Theorem 1.1]{Bar17} later on.

\begin{lemma}\label{lem:propertiesofFSxi}Let $\Theta>0$ and let $d_\Theta$ be 
as in \cref{def:dThetaEkplus1}. Let $Z$ and $FS_\Theta$ be as above. Then
\begin{enumerate}[a),leftmargin=1.7\parindent]
\item $(E^k(T),d_\Theta)$ is a discrete countable proper generalised metric 
space with a proper isometric $G$-action.
\item $Z$ is a separable metric space with an action of $G$ by homeomorphisms. 
$FS_\Theta$ is separable and metrisable as well.
\item $FS_\Theta$ is a closed $G$-invariant subspace of $E^k(T)\times Z$ 
(equipped with the diagonal action of $G$).
\item The (small inductive) dimension of $FS_\Theta$ is $0$.
\item There are $D,R>0$, independent of $\Theta$, such that for all 
$(\xi_-,\xi_+)\in Z$ the subspace $E^k(T)_{(\xi_-,\xi_+)}:=\{ \sigma\in E^k(T)\ 
|\ (\sigma,\xi_-,\xi_+)\in FS_\Theta \}$ of $E^k(T)$ has the $(D,R)$-doubling 
property.
\item For all $(\sigma,\xi_-,\xi_+)\in FS_\Theta$, the isotropy group 
$G_{(\xi_-,\xi_+)}=G_{\xi_-}\cap G_{\xi_+}$ belongs to the family 
$\mathcal{F}_\partial\cup\mathcal{F}_{(k)}$, where $\mathcal{F}_{(k)}:=\{H\leq 
G\ |\ \exists\ \sigma\in E^{k}(T):\ H\leq G_{o(\sigma)}\cap 
G_{t(\sigma)}\}\subset \mathcal{FIN}(G)$. In particular, these isotropy groups 
all belong to $\mathcal{F}=\mathcal{F}_T\cup\mathcal{F}_\partial$.
\end{enumerate}
\end{lemma}

\begin{proof}
a) is immediate since $d_\Theta$ only takes values in $\N$ and the action of 
$G$ on $E^{k}(T)$ has finite point stabilisers.

b) Since we assumed $T$ to be countable, $\Delta_+(T)$ is separable 
and metrisable by \cref{lem:Tobsis2ndcountableandmetrizable}. So $Z$ is as 
well. Since $E^k(T)$ is discrete and countable, $E^k(T)\times Z$ is separable 
and metrisable and hence is $FS_\Theta$. 

d) By the product theorem \cite[Theorem II.5, p.~20]{Nag65}\footnote{This and 
the following results referenced are formulated for the strong inductive 
dimension. However, recall that for separable metrisable spaces, the small 
inductive dimension coincides with the covering dimension (and the strong 
inductive dimension) \cite[Theorem IV.1, p.~90]{Nag65}.} and 
\cref{lem:overlineTobshasdimension1} the space $Z=\Delta_+(T)
\times\Delta_+(T)$ has dimension $0$.
Since a discrete space $V$ has dimension 0, $E^k(T)\times Z$ has dimension 0 by 
\cite[Theorem II.5, p.~20]{Nag65}. Finally, by \cite[Theorem II.3, 
p.~19]{Nag65}, the dimension does not increase when taking subspaces and hence 
$FS_\Theta$ has dimension $0$.

f) Since $G_{(\xi_-,\xi_+)}=G_{\xi_-}\cap G_{\xi_+}$, if $(\xi_-,\xi_+)\in 
\partial T\times\partial T\setminus \diag$, then 
$G_{(\xi_-,\xi_+)}\in\mathcal{F}_\partial$ by definition. If at least one of 
the points $\xi_-,\xi_+$ lies in $T$, then, by definition of $FS_\Theta$, the 
group $G_{\xi_-}\cap G_{\xi_+}$ is contained in the pointwise stabiliser 
$G_\sigma$ of some $\sigma\in E^{k}(T)$, and hence is finite.
It remains to show c) and e).

c) Since $E^k(T)\times Z$ is metrisable, we can work with sequences. Moreover, 
without loss of generality any convergent sequence has the form
\begin{equation*}
(\sigma,\xi_{-,n},\xi_{+,n})\to (\sigma,\xi_-,\xi_+)\in E^k(T)\times Z,
\end{equation*}
since $E^k(T)$ is discrete. Now, if all $(\sigma,\xi_{-,n},\xi_{+,n})$ lie in 
$FS_\Theta$, then $\sigma\subset [\xi_{-,n},\xi_{+,n}]$ holds for all $n\in\N$. 
Hence, the topology of $\Delta_+(T)$ forces $\sigma$ to lie on $[\xi_-,\xi_+]$ 
(though $\sigma=[\xi_-,\xi_+]$ is possible). The geodesic $[\xi_-,\xi_+]$ is 
$\Theta$-small, since the geodesics $[\xi_{-,n},\xi_{+,n}]$ are and any finite 
subsegment of $[\xi_-,\xi_+]$ is eventually contained in some 
$[\xi_{-,N},\xi_{+,N}]$. $G$-invariance of $FS_\Theta$ is immediate, since the 
notion of a $\Theta$-small geodesic is $G$-invariant.

e) We will show that for all $(\xi_-,\xi_+)$ the subspace 
$E^k(T)_{(\xi_-,\xi_+)}$ has the $(5,3)$-doubling property:
So let $\alpha\geq R=3$. First, note that $E^k(T)_{(\xi_-,\xi_+)}$
is empty if the geodesic $[\xi_-,\xi_+]$ is not $\Theta$-small or is of length 
$\leq k-1$. Otherwise $E^k(T)_{(\xi_-,\xi_+)}$ is equal to $E^{k}(T)\cap 
[\xi_-,\xi_+]$. In this case, we can take midpoints of segments to obtain a 
bijection between $E^k(T)_{(\xi_-,\xi_+)}$ and (a subset of) $\Z$. Since 
$d_\Theta(\sigma,\sigma^\prime)=3\cdot d_T(m(\sigma),m(\sigma^\prime))$ for all 
$\sigma,\sigma^\prime$ on $[\xi_-,\xi_+]$, any set of the form 
$B_{2\alpha}(\tilde \sigma)\cap E^k(T)_{(\xi_-,\xi_+)}$ consists of at most 
$\frac{4}{3}\alpha+1$ elements of $E^{k}(T)$. Furthermore, if 
$\sigma,\sigma^\prime\in E^k(T)_{(\xi_-,\xi_+)}$ belong to an 
$\alpha$-separated subset, then 
$d_T(m(\sigma),m(\sigma^\prime))\geq\frac{\alpha}{3}$ follows from 
$d_\Theta(\sigma,\sigma^\prime)=3\cdot d_T(m(\sigma),m(\sigma^\prime))$. Thus, 
there are at most $\frac{3}{\alpha}\cdot (\frac{4}{3}\alpha+1)\leq 5=:D$
elements in any $\alpha$-separated subset of $E^k(T)_{(\xi_-,\xi_+)}$ which is 
contained in a $2\alpha$-ball.
\end{proof}

Next, we define a family of maps from $G\times_{\Theta}\partial T$ to 
$FS_\Theta$ that formalise the idea of letting an element $\sigma$ of 
$E^{k}(T)$ \squote{flow along a $\Theta$-small geodesic $[\anf(g,\xi),\xi]$}. 
Morally, this family of maps could be thought of as a partial flow on 
$FS_\Theta$.

\begin{definition}Let $\tau\in\N$ and $\Theta^\prime\geq\Theta>\frac{k}{2}$. We 
define
\begin{equation*}
\phimap: G\times_{\Theta}\partial T\longrightarrow 
FS_{\Theta^\prime},\ 
(g,\xi)\mapsto (\sigma_\tau(g,\xi),\anf(g,\xi),\xi),
\end{equation*}
where $\sigma_\tau(g,\xi)$ is the unique geodesic segment of length $k$ on 
$[\anf(g,\xi),\xi]$ such that
\begin{equation*}
d_T\big(\anf(g,\xi),m(\sigma_\tau(g,\xi))\big)=\tau.
\end{equation*}
\end{definition}

\begin{lemma}$\phimap$ is continuous and $G$-equivariant.
\end{lemma}

\begin{proof}$G$-equivariance is clear, since the action of $G$ on $T$ is 
simplicial. 
Let $(g,\xi_n)\in G\times_\Theta\partial T$ be a sequence converging to some 
$(g,\xi)\in G\times_\Theta\partial T$. Since 
$\anf(g,\eta)\in\{gw_0,gw_0^\prime\}$ for arbitrary $\eta$, 
we can assume without loss of generality that $\anf(g,\xi_n)=\anf(g,\xi)$ for 
all $n\in\N$. The convergence of $\xi_n\to\xi$ then already implies that 
$\sigma_\tau(g,\xi_n)=\sigma_\tau(g,\xi)$ for all but finitely many $n$. 
\end{proof}

\section{\texorpdfstring{Covers for $G\times_\Theta\partial 
T$}{Covers for the 'small' part of 
GxDelta(T)}}\label{sec:coveringGtimesXipartialT}
Given some threshold $\Theta>0$, this section provides an open $G$-invariant 
$\mathcal{F}_\partial\cup\mathcal{F}_{(k)}$-cover of $G\times_{\Theta}\partial 
T$ that is wide for $G\times_{\Theta}\partial T$. We start by using 
\cite[Theorem 1.1]{Bar17} to obtain a family of covers for the flow space 
$FS_{\Theta^\prime}$ for any $\Theta^\prime$.

\begin{corollary}\label{cor:theorem11appliedtogetWxialphaonFSxifinal}Let 
$\Theta^\prime>0$. Let $FS_{\Theta^\prime}$ be as before and let $D$ and $R$ be 
as in 
\cref{lem:propertiesofFSxi}. Then there is $N\in\N$ independent of 
$\Theta^\prime$, such that for any $\alpha>0$ there is an open $G$-invariant 
$\mathcal{F}_\partial\cup\mathcal{F}_{(k)}$-cover 
$\mathcal{W}_{\Theta^\prime,\alpha}$ of $FS_{\Theta^\prime}$ satisfying
\begin{enumerate}[a)]
\item $\dim\mathcal{W}_{\Theta^\prime,\alpha}\leq N$,
\item for every $(\sigma,\xi_-,\xi_+)\in FS_{\Theta^\prime}$ there is 
$W\in\mathcal{W}_{\Theta^\prime,\alpha}$ such that 
\begin{equation*}
\big(B^{d_{\Theta^\prime}}_\alpha(\sigma)\times\{(\xi_-,\xi_+)\}\big)\cap 
FS_{\Theta^\prime}\subset W.
\end{equation*}
\end{enumerate}
\end{corollary}

\begin{proof}This is the conclusion of \cite[Theorem 1.1]{Bar17}. That the 
assumptions of this theorem are satisfied by $FS_{\Theta^\prime}$ follows from  
\cref{lem:propertiesofFSxi}.
\end{proof}

Pulling back these covers along $\phimap$ gives the following 
corollary.

\begin{corollary}\label{cor:phitauminus1aresuitablecoversmodulowidenessfinal}
There is $N$ such that for all $\alpha>0, \Theta^\prime\geq\Theta>\frac{k}{2}$ 
and $\tau\in\N$ the set
\begin{equation*}
\phimap^{-1}(\mathcal{W}_{\Theta^\prime,\alpha}):=\{ 
\phimap^{-1}(W)\ |\ W\in\mathcal{W}_{\Theta^\prime,\alpha}\}
\end{equation*}
is an open $G$-invariant $\mathcal{F}_\partial\cup\mathcal{F}_{(k)}$-cover of 
$G\times_{\Theta}\partial T$ of dimension $\leq N$.
\end{corollary}

Before we show that among these covers we can find one which is wide for 
$G\times_\Theta\partial T$, we will state two technical assertions.

\begin{lemma}\label{lem:elementsofBalpha1dontchangeximinus}Let $\alpha>0$. Let 
$v\in V(T)$ and $g_\tau v\to \xi\in\partial T$ for $\tau\to\infty$. Then, for 
all $h\in B_\alpha(1)$, the sequence $g_\tau hv$ also converges to $\xi$.
\end{lemma}

\begin{proof}Fix a basepoint $b\in T$. So $[b,g_\tau v]$ converges uniformly on 
compacta to $[b,\xi]$. Let $C:=\max\{ d_T(v,hv)\ |\ h\in B_\alpha(1)\}$. Since 
$d_T(g_\tau v,g_\tau hv)=d_T(v,hv)\leq C$ for all $\tau$, the sequence 
$[b,g_h\tau v]$ converges uniformly on compacta to $[b,\xi]$ if and only if the 
sequence $[b,g_\tau v]$ does.
\end{proof}

\begin{lemma}\label{lem:limitsareinpartialTfinal}Let $\Theta>0$. Let 
$\xi_{-,n}\to \xi_-$ and $\xi_{+,n}\to \xi_+$ be convergent sequences in 
$\Delta_+(T)$. Furthermore, assume that there is a segment $\sigma\in E^{k}(T)$ 
such that $\sigma$ lies on all geodesics $[\xi_{-,n},\xi_{+,n}]$  and that for 
all $r>0$ the set $B^{d_T}_r(\sigma):=\{ z\in T\ |\ d_T(z,\sigma)\leq r\}$ 
contains only finitely many of the $\xi_{-,n},\xi_{+,n}$. If the geodesics 
$[\xi_{-,n},\xi_{+,n}]$ are $\Theta$-small, then $\xi_-,\xi_+\in\partial T$ and 
$[\xi_-,\xi_+]$ is $\Theta$-small.
\end{lemma}

\begin{proof} Assume that $\xi_+\not\in\partial T$. If $\xi_{+,n}=\xi_+$ for 
all $n\geq n_0$, there is---in contradiction to the assumptions of the 
lemma---some $r>0$ such that $B_r^{d_T}(\sigma)$ contains infinitely many of 
$\xi_{+,n}$. So, for $\xi_{+,n}\to \xi_+\in V(T)$ to hold, the union of the
$[\xi_+,\xi_{+,n}]$ must contain infinitely many of the 
edges incident to $\xi_+$. Since $d_{E^k}$ is proper, this contradicts the 
other assumption of the lemma, that all $[o(\sigma),\xi_{+,n}]$ are 
$\Theta$-small. The argument for $\xi_-\in\partial T$ is the same and the 
geodesic $[\xi_-,\xi_+]$ is $\Theta$-small since all geodesics 
$[\xi_{-,n},\xi_{+,n}]$ are.
\end{proof}

\begin{proposition}\label{prop:longthin}Retain the assumptions of 
\cref{prop:finFamenable}. Then there is $N\in\N$ such that for all $\alpha>0$ 
and all $\Theta>0$ there is a $G$-invariant collection 
$\mathcal{U}_{\Theta,\alpha}$ of open 
$\mathcal{F}_\partial\cup\mathcal{F}_{(k)}$-subsets of $G\times \Delta(T)$ such 
that the order of $\mathcal{U}_{\Theta,\alpha}$ is at most $N$ and 
$\mathcal{U}_{\Theta,\alpha}$ is wide for $G\times_{\Theta}\partial T$.
\end{proposition}

\begin{proof}(following the proof of \cite[Proposition 3.2]{Bar17}) Let $N$ be 
as in \cref{cor:theorem11appliedtogetWxialphaonFSxifinal} and 
\cref{cor:phitauminus1aresuitablecoversmodulowidenessfinal}. 
Let $\alpha,\Theta>0$ be given. Recall that $L_0$ is the finite subtree of $T$ 
spanned by the vertices $\{hw_0, hw_0^\prime\ |\ h\in B_\alpha(1)\}$. Let 
$C=C(\alpha)$ be the diameter of $L_0$ with respect to $d_T$. (In particular, 
for any $\xi\in\Delta_+(T)$, $h\in B_\alpha(1)$ and $g\in G$ the distance of 
$\anf(g,\xi)$ to $\anf(gh,\xi)$ is bounded by $C$.) Set $\alpha^\prime:=3C$ and 
$\Theta^\prime:=\Upsilon_1(\Theta)$ (cf. \cref{constants}). Finally,
let $\mathcal{W}_{\Theta^\prime,\alpha^\prime}$ be as in 
\cref{cor:theorem11appliedtogetWxialphaonFSxifinal}.

We will show that there is $\tau\in\N$ such that 
$\phimap^{-1}(\mathcal{W}_{\Theta^\prime,\alpha^\prime})$ is ($\alpha$-)wide 
for $G\times_\Theta\partial T$. Once we have proven this, we can thicken 
$\phimap^{-1}(\mathcal{W}_{\Theta^\prime,\alpha^\prime})$ to a $G$-invariant 
collection $\mathcal{U}_{\Theta,\alpha}$ of open $\mathcal{F}$-subsets of 
$G\times\Delta(T)$ that is still ($\alpha$-)wide for $G\times_\Theta\partial T$ 
(see \cite[Appendix B]{Bar17}, \cite[Lemma 4.14]{BL12a}). Since  all 
$\phimap^{-1}(\mathcal{W}_{\Theta^\prime,\alpha^\prime})$ have dimension at most 
$N$ by \cref{cor:phitauminus1aresuitablecoversmodulowidenessfinal}, the 
collection $\mathcal{U}_{\Theta,\alpha}$ will have order $\leq N$ as well as the 
thickening process does not increase the order \cite[Lemma 
B.2]{Bar17}.\par\medskip

To show that for the given $\alpha,\Theta>0$ there is a 
$\phimap^{-1}(\mathcal{W}_{\Theta^\prime,\alpha^\prime})$ that is 
($\alpha$-)wide, start by assuming the contrary.

\underline{Assumption:} For all $\tau\in\N$ there is $(g_\tau,\xi_\tau)\in 
G\times_{\Theta}\partial T$ such that for all 
$W\in\mathcal{W}_{\Theta^\prime,\alpha^\prime}$
\begin{equation*}B_\alpha(g_\tau)\times\{\xi_\tau\}\not\subset 
\phimap^{-1}(W)\ .
\end{equation*}

Applying $\phimap$ to $(g_\tau,\xi_\tau)$ we obtain a segment 
$\sigma_\tau(g_\tau,\xi_\tau)$ of length $k$ on 
$[\anf(g_\tau,\xi_\tau),\xi_\tau]$. Let $\sigma_1$ be the first segment of 
length $k$ on $[w_0,w_0^\prime]$ and $\sigma_2$ be the last segment of length 
$k$ on $[w_0,w_0^\prime]$ (those segments exist since $[w_0,w_0^\prime]$ has 
length $\geq 5k$). Since $[\anf(g_\tau,\xi_\tau),\xi_\tau]$ contains 
$[\anf(g_\tau,\xi_\tau), g\meps]$, the first segment of length $k$ on 
$[\anf(g_\tau,\xi_\tau),\xi_\tau]$ is either $g_\tau \sigma_1$ or $g_\tau 
\sigma_2$. Furthermore, since all $[\anf(g_\tau ,\xi_\tau),\xi_\tau]$ are 
$\Theta$-small and their respective first segments of length $k$ lie in at most 
two different $G$-orbits of $G\curvearrowright E^{k}(T)$, by our special 
construction of $d_{E^{k}}$ (cf.~\cref{lem:propermetriconEkplus1T}),
it follows that the segments $\sigma_\tau(g_\tau,\xi_\tau)$ can only belong to 
a finite number of $G$-orbits of $G\curvearrowright E^{k}(T)$. 
Since $\phimap^{-1}(\mathcal{W}_{\Theta^\prime,\alpha^\prime})$ is a 
$G$-invariant collection we can assume (after passing to a subsequence) that 
there is a segment $\sigma\in E^{k}(T)$ such that 
$\sigma=\sigma_\tau(g_\tau,\xi_\tau)$ for all $\tau$.

Since $\Delta_+(T)$ is a closed subspace of $\overline{T}^{obs}$, it is 
compact as well and by passing to a further subsequence (twice) we can enforce 
the existence of $\xi_+=\lim_\tau \xi_\tau$ and $\xi_-=\lim_\tau \anf(g_\tau 
,\xi_\tau)$ in $\Delta_+(T)$.

Since $\tau$ increases, no $B^{d_T}_r(\sigma)$ contains infinitely many of the 
$\anf(g_\tau ,\xi_\tau)$. Moreover, by assumption, all $\xi_\tau$ lie in 
$\partial T$. Hence, by \cref{lem:limitsareinpartialTfinal}, we obtain that 
$\xi_-$ and $\xi_+$ both lie in $\partial T$ and that $[\xi_-,\xi_+]$ is 
$\Theta$-small. Furthermore, $\sigma$ lies on $[\xi_-,\xi_+]$. In other words, 
$(\sigma,\xi_-,\xi_+)\in FS_\Theta\subseteq FS_{\Theta^\prime}$.

Now, by \cref{cor:theorem11appliedtogetWxialphaonFSxifinal}, there is 
$W_0\in\mathcal{W}_{\Theta^\prime,\alpha^\prime}$ such that
\begin{equation*}
\big(B_{\alpha^\prime}^{d_{\Theta^\prime}}(\sigma)\times\{(\xi_-,\xi_+)\}
\big)\cap 
FS_{\Theta^\prime} \subset W_0.
\end{equation*}
As $d_{\Theta^\prime}$ is proper,  the ball 
$B_{\alpha^\prime}^{d_{\Theta^\prime}}(\sigma)$ is 
finite and so is $M:=B_{\alpha^\prime}^{d_{\Theta^\prime}}(\sigma)\cap 
[\xi_-,\xi_+]$. 
Since $\xi_-,\xi_+\in\partial T$ and $W_0$ is open, there are disjoint open 
neighbourhoods $U_-$ and $U_+$ in $\Delta_+(T)$ of $\xi_-$ and $\xi_+$, 
respectively, such that 
\begin{equation*}
(M\times U_-\times U_+)\cap FS_\Theta\subset W_0.
\end{equation*}
Without loss of generality we can assume that $U_+$ is of the form $M(\xi_+,a)$ 
(cf. \cref{def:obstop})
for some $a\in T$ with $d_T(t(\sigma),a)> C$ (the purpose of this technical 
assumption will only be clear at the very end of this proof).

By \cref{lem:elementsofBalpha1dontchangeximinus}, we still have $\lim_\tau 
\anf(g_\tau h ,\xi_\tau)=\xi_-$ for all $h\in B_\alpha(1)$. So we can find 
$\tau_{0,h}$ such that $\anf(g_\tau h ,\xi_\tau)\in U_-$ and $\xi_\tau\in U_+$ 
for all $\tau\geq\tau_{0,h}$. Since $B_\alpha(1)$ is finite and 
$B_\alpha(g_\tau)=g_\tau B_\alpha(1)$, there is a $\tau_0$ such that 
$\anf(g,\xi_\tau)\in U_-$ and $\xi_\tau\in U_+$ for all $\tau\geq\tau_0$ and 
$g\in B_\alpha(g_\tau)$. Since $M$ is finite as well, we can (if necessary) 
enlarge $\tau_0$ further to obtain that all elements of $M$ already lie on 
$[\anf(g_{\tau_0} ,\xi_{\tau_0}),\xi_{\tau_0}]$. We now claim that there is 
some $\tau\geq \tau_0$ such that 
\begin{equation*}
g_{\tau}B_\alpha(1)\times\{\xi_{\tau}\}=B_\alpha(g_{\tau})\times\{\xi_{\tau}\}
\subset\phimap^{-1}(W_0),
\end{equation*}
which would be a contradiction to the assumption we started with.
To prove this claim it is sufficient to find $\tau$ such that 
\begin{equation*}
\big(\sigma_\tau(g_\tau h,\xi_\tau),\anf(g_\tau h,\xi_\tau),\xi_\tau\big)\in 
( M\times U_-\times U_+)\cap FS_{\Theta^\prime}
\end{equation*}
for all $h\in B_\alpha(1)$. $\big(\sigma_\tau(g_\tau h,\xi_\tau),\anf(g_\tau 
h,\xi_\tau),\xi_\tau\big)\in FS_{\Theta^\prime}$ holds by 
\cref{lem:smallnessinheritence} and \cref{lem:overlapping}.
By the way $\tau_0$ was defined, it is guaranteed that $\xi_{\tau}\in U_+$ and 
$\anf(g_{\tau}h ,\xi_{\tau})\in U_-$ for all $\tau\geq\tau_0$. So it only 
remains to find $\tau\geq\tau_0$ large enough such that 
$\sigma_\tau(g_{\tau}h,\xi_{\tau})\in M$ holds for all $h\in B_\alpha(1)$.

Let $v_h$ be the first vertex on $[\anf(g_{\tau}h ,\xi_{\tau}),\xi_{\tau}]$ 
that also lies on $[\anf(g_{\tau} ,\xi_{\tau}),\xi_{\tau}]$, i.e.~
\begin{equation*}[\anf(g_{\tau}h ,\xi_{\tau}),\xi_{\tau}]\cap [\anf(g_{\tau} 
,\xi_{\tau}),\xi_{\tau}]=[v_h,\xi_{\tau}].
\end{equation*} 
Since $B_\alpha(1)$ is finite, there is a vertex $v=v_{h_0}$ which---among the 
$v_h$---has maximal distance to $\anf(g_\tau,\xi_\tau)$. In particular, 
$[v,\xi_{\tau}]$ is contained in  $[\anf(g_{\tau}h ,\xi_{\tau}),\xi_{\tau}]$ 
for all $h\in B_\alpha(1)$. Note that, since all $\anf(g_\tau h,\xi_\tau)$ lie 
in $U_-$, all $v_h$ lie in $U_-$ as well. In particular, $v$ lies in $U_-$.
Moreover, the vertex $v$ lies on the geodesic 
$[\anf(g_{\tau}h_0,\xi_{\tau}),\anf(g_{\tau}h,\xi_{\tau})]$ for all $h\in 
B_\alpha(1)$. Since the length of 
$[\anf(g_{\tau}h_0,\xi_{\tau}),\anf(g_{\tau}h,\xi_{\tau})]$ is bounded by $C$, 
the vertex $v$ has distance at most $C$ to any of the $\anf(g_{\tau}h 
,\xi_{\tau})$. Hence, (by choosing $\tau\geq\tau_0$ large enough) we can force 
$\sigma_\tau(g_{\tau}h,\xi_{\tau})$ to lie on $[v,\xi_{\tau}]$ for all $h\in 
B_\alpha(1)$. Since 
$d_T(\anf(g_\tau,\xi_\tau),m(\sigma_\tau(g_{\tau},\xi_{\tau})))=\tau$, the 
bound $C$ and the triangle inequality for $d_T$ imply
\begin{equation*}
\tau+C \leq d_T\big(\anf(g_{\tau} 
,\xi_{\tau}),m(\sigma_\tau(g_{\tau}h,\xi_{\tau}))\big)\geq \tau-C.
\end{equation*} 
Both $\sigma_\tau(g_{\tau}h,\xi_{\tau})$ and $\sigma_\tau(g_{\tau},\xi_{\tau})$ 
lie on $[v,\xi_{\tau}]$ and it follows that
\begin{equation*}
d_T\big(m(\sigma_\tau(g_{\tau}h,\xi_{\tau})), 
m(\sigma_\tau(g_{\tau},\xi_{\tau}))\big)\leq C.
\end{equation*}
Since the geodesic $[\anf(g_{\tau} ,\xi_{\tau}),\xi_{\tau}]$ is $\Theta$- 
and hence $\Theta^\prime$-small and contains both the segments 
$\sigma_\tau(g_{\tau}h,\xi_{\tau})$ and $\sigma_\tau(g_{\tau},\xi_{\tau})$, by 
definition of $d_{\Theta^\prime}$ it follows that
\begin{equation*}
d_{\Theta^\prime}(\sigma_\tau(g_{\tau}h,\xi_{\tau}),\sigma_\tau(g_{\tau},\xi_{
\tau }))=3\cdot 
d_{T}\big(m(\sigma_\tau(g_{\tau}h,\xi_{\tau})),m(\sigma_\tau(g_{\tau},\xi_{\tau}
))\big)\leq 3C=\alpha^\prime.
\end{equation*} 
Hence, $\sigma_\tau(g_{\tau}h,\xi_{\tau})\in 
B_{\alpha^\prime}^{d_{\Theta^\prime}}(\sigma_\tau(g_{\tau},\xi_{\tau}))$ for 
all $h\in B_\alpha(1)$. Recall that $\sigma=\sigma_\tau(g_{\tau},\xi_{\tau})$ 
and $U_+=M(\xi_+,a)$ where $a$ satisfies $d_T(t(\sigma),a)\geq C$. Therefore, 
$\sigma_\tau(g_{\tau}h,\xi_{\tau})$ indeed lies on $[\xi_-,\xi_+]$ and is 
contained in $B_{\alpha^\prime}^{d_{\Theta^\prime}}(\sigma)$, so it is 
contained in $M$.
\end{proof}

This concludes the last technical part that was used in the proof of 
\cref{thm:mainthm:relFJC}.


\providecommand{\bysame}{\leavevmode\hbox to3em{\hrulefill}\thinspace}
\providecommand{\MR}{\relax\ifhmode\unskip\space\fi MR }
\providecommand{\MRhref}[2]{%
  \href{http://www.ams.org/mathscinet-getitem?mr=#1}{#2}
}
\providecommand{\href}[2]{#2}


\end{document}